\documentclass[a4paper, 12pt, reqno]{amsart}
\usepackage{amsmath}
\usepackage{amssymb}
\usepackage{graphics}
\usepackage{mathrsfs}
\usepackage[pdflatex]{graphicx}
\usepackage{hyperref}
\usepackage[marginratio=1:1,tmargin=117pt, height=650pt]{geometry}
\usepackage{color}
\usepackage{amsthm}

\usepackage[all]{xy}

\usepackage{eurosym}

\usepackage[english]{babel}
\usepackage[T1]{fontenc}
\usepackage[latin1]{inputenc}
\usepackage{marginnote}

\usepackage{tikz, tikz-cd}
\usetikzlibrary{matrix, arrows, decorations.pathmorphing}

\newcommand{\red}[1]{{\color{black} #1}}
\newcommand{\redf}[1]{{\color{black} #1}}
\newcommand{\blue}[1]{{\color{black} #1}}
\newcommand{\blueb}[1]{{\color{black} #1}}
\newcommand{\vfour}[1]{{\color{black} #1}}
\newcommand{\vfive}[1]{{\color{black} #1}}
\newcommand{\vseven}[1]{{\color{black} #1}}

\definecolor{orange}{rgb}{1,0.68,0}
\definecolor{turquoise}{rgb}{0,1,1}

\newcommand{\ds}{\displaystyle}

\theoremstyle{plain}
\newtheorem{thm}{Theorem}[section]

\newtheorem{cor}[thm]{Corollary}

\newtheorem{lem}[thm]{Lemma}
\newtheorem{assump}[thm]{Assumption}

\theoremstyle{definition}
\newtheorem{dfn}[thm]{Definition}

\theoremstyle{remark} 

\newtheorem{rmk}[thm]{Remark}

\newcommand{\C}{{\mathbb{C}}}

\newcommand{\D}{{\mathbb{D}}}
\newcommand{\B}{\mathcal B}
\newcommand{\R}{{\mathbb{R}}}
\newcommand{\Z}{{\mathbb{Z}}}
\newcommand{\N}{{\mathbb{N}}}

\renewcommand{\Re}{\operatorname{Re}}
\renewcommand{\Im}{\operatorname{Im}}

\newcommand{\eps}{\varepsilon}

\newcommand{\vphi}{\varphi}
\newcommand{\sminus}{\setminus}
\newcommand{\barD}{\overline {\mathbb{D}}}
\newcommand{\calC}{\mathcal{C}}

\renewcommand{\Re}{\operatorname{Re}}
\renewcommand{\Im}{\operatorname{Im}}
\newcommand{\supp}{\operatorname{supp}}

\newcommand{\tinyhsp}{\hspace{.06667em}}

\renewcommand{\le}{\leqslant}
\renewcommand{\ge}{\geqslant}
\renewcommand{\subset}{\subseteq}
\renewcommand{\supset}{\supseteq}

\makeatletter
\newcommand{\displaybump}{\hbox to \@totalleftmargin{\hfil}}
\makeatother

\begin{document}

\bibliographystyle{amsalpha}

%\title[Finite order oscillating wandering domains]{Entire functions of finite order with oscillating wandering domains}
\title[Wandering domains for functions of finite order in the class $\mathcal B$]{Wandering domains for entire functions of finite order in the Eremenko-Lyubich class}
\author[D. Mart\'i-Pete \and M. Shishikura]{David Mart\'i-Pete \and Mitsuhiro Shishikura}
\address{Department of Mathematics\\ Kyoto University\\ Kyoto 606-8502\\ Japan}
\email{martipete@math.kyoto-u.ac.jp}
\email{mitsu@math.kyoto-u.ac.jp}
%\email{david.martipete@open.ac.uk}
%\address{Department of Mathematics\\ Faculty of Science\\ Kyoto University\\ Kyoto 606-8502\\ Japan}
%\email{martipete@kyoto-u.ac.jp}
\date{\today}
%\thanks{This research was partially supported by...}
\thanks{This work was partially supported by the Grant-in-Aid JP16F16807 (first author) and the Grant-in-Aid 26287016 and 15K13444 (second author) from the Japan Society for the Promotion of Science.}
%\dedicatory{}

\maketitle

\begin{abstract}
Recently Bishop constructed the first example of a bounded-type transcendental entire function with a wandering domain using a new technique called quasiconfomal folding. It is easy to check that his method produces \blueb{an} entire function of infinite order. We construct the first examples of \blueb{entire} functions of finite order in the class $\mathcal B$ with wandering domains. \blueb{As in Bishop's case,} these wandering domains are of oscillating type, that is, they have an unbounded non-escaping orbit. To construct such functions we use quasiregular interpolation instead of quasiconformal folding, which is much more straightforward. Our examples have order $p/2$ for any $p\in\N$ \blueb{and,} since the order of functions in the class $\mathcal B$ is at least $1/2$, we achieve the smallest possible order. Finally, we can modify the construction to obtain functions of finite order in the class $\mathcal B$ with any number of grand orbits of wandering domains, \blueb{including infinitely many.}% \blueb{and we show that such wandering domains are bounded sets.}%We also provide examples of entire functions in the class~$\mathcal B$ without zeros (of infinite order) that have oscillating wandering domains.
\end{abstract}

\section{Introduction}

Given an entire function $f$ and a point $z_0\in \C$, we define the orbit of $z_0$ under iteration by $f$ as the sequence $z_n:=f(z_{n-1})=f^n(z_0)$, for $n\in\mathbb{N}$, where $f^n$ denotes the composition of $f$ with itself $n$ times, $f^n=f\circ\cdots\circ f$. For each entire function~$f$, we partition the complex plane into the \textit{Fatou set} of $f$, or stable set,
$$
F(f):=\bigl\{z\in \C\, :\, \{f^n\}_{n\in\mathbb N} \mbox{ is a normal family in some neighbourhood of } z\bigr\}
$$
and the \textit{Julia set} of $f$, $J(f):=\C\setminus  F(f)$, where the chaotic behaviour takes place. We refer to a connected component of $F(f)$ as a \textit{Fatou component} of $f$. We also consider the partition of $\C$ into the \textit{escaping set},
$$
I(f):=\{z\in\C\, :\, f^n(z)\to\infty \text{ as } n\to\infty\},
$$
the set of points with bounded orbit, $K(f):=\{z\in\C\ :\ \{f^n(z)\}_n \text{ is bounded}\}$, and the set of points with unbounded non-escaping orbit $BU(f):=\C\setminus (I(f)\cup K(f))$. The book \cite{milnor06} is a basic reference on the iteration of holomorphic functions in one complex variable. See \cite{bergweiler93} for a survey on the dynamics of transcendental entire and meromorphic functions. 

Let $U$ be a Fatou component of $f$, and denote by $U_n$ the Fatou component of $f$ that contains $f^n(U)$, for $n\in \mathbb{N}$. We say that $U$ is a \textit{wandering domain} if $U_n\cap U_m\neq\emptyset$ implies that $m=n$. One of the main differences between the iteration of transcendental entire functions and that of polynomials is the existence of wandering domains. Sullivan \cite{sullivan85} proved that rational functions (and, in particular, polynomials) do not have wandering domains. This question had been open since the time of Fatou and Julia, and Sullivan's use of quasiconformal maps revolutionized complex dynamics. In the subsequent years, many other applications of quasicoformal maps in dynamics appeared, see \cite{shishikura87} and \cite{branner-fagella14} for a compilation of such techniques known as \textit{quasiconformal surgery}. Our construction also uses quasiconformal maps to produce oscillating wandering domains.%: he constructed an infinite dimensional space of quasiconformal deformations of a rational map which contradicted the fact that the space of all rational maps of a given degree has finite dimension.  

Fatou \cite{fatou19} showed that if $U$ is a wandering domain, then all the limit functions of converging partial sequences $(f^{n_k})_k$ on~$U$ are constant. Let $L(U)\subseteq \C\cup\{\infty\}$ denote the set of all possible limit values. For a transcendental entire function $f$, we define the \textit{singular set} of $f$ as $S(f):=\overline{\text{sing}(f^{-1})}$, which consists of all singular values (critical values and asymptotic values) of~$f$ as well as their accumulation points. We also consider the \textit{postsingular set}, $P(f)$, which is the closure of the union of the forward orbits of all points in $S(f)$. Baker \cite{baker70} improved the result of Fatou by showing that limit functions satisfy that $L(U)\subseteq (J(f)\cap \overline{P(f)})\cup\{\infty\}$. The last development in this direction is due to Bergweiler, Haruta, Kriete, Meier and Terglane \cite{bergweiler-haruta-kriete-meier-terglane93}\vseven{,} who showed that, \vseven{in fact}, $L(U)\subseteq (J(f)\cap P(f)')\cup\{\infty\}$, where $P(f)'$ is the derived set of $P(f)$.

According to their orbits and limit functions, wandering domains can be classified into the following types: we say that
\begin{itemize}
\item $U$ is an \textit{escaping} wandering domain if $U\subseteq I(f)$, that is, $L(U)= \{\infty\}$;
\item $U$ is an \textit{oscillating} wandering domain if $U\subseteq BU(f)$, that is, $\{\infty,a\}\subseteq L(U)$ for some $a\in\C$;
\item $U$ is a \textit{bounded orbit} wandering domain if $U\subseteq K(f)$, that is, $\infty\notin L(U)$.
\end{itemize}
However, it is not known whether wandering domains with bounded orbit exist.

The first example of a transcendental entire function with a wandering domain was given by Baker \cite{baker63, baker76} and consisted of an infinite product that had a sequence of multiply connected wandering domains escaping to infinity. In fact, if an entire function has a multiply connected Fatou component, then it must be a (fast) escaping wandering domain \cite{baker84, rippon-stallard05}. In \cite{kisaka-shishikura08}, Kisaka and Shishikura constructed functions with multiply connected wandering domains using quasiconformal surgery. Further examples of simply connected escaping wandering domains are due, for example, to Herman \cite{herman84}, Baker \cite[Example~5.3]{baker84} and Eremenko and Lyubich \cite[Example~2]{eremenko-lyubich87}. Regarding oscillating wandering domains, Eremenko and Lyubich \cite[Example~1]{eremenko-lyubich87} constructed a wandering domain $U$ such that $L(U)$ is an unbounded infinite set. Observe that every Fatou component in $BU(f)$ is necessarily an oscillating wandering domain \cite[Theorem~1.1]{sixsmith-osborne16}. \redf{Nowadays, the existence of wandering domains for certain classes of holomorphic maps continues to be an important question in complex dynamics \cite{mihaljevicbrandt-rempegillen13,astorg-buff-dujardin-peters-raissy16}.}%The set $BU(f)$ was studied systematically for the first time by Sixsmith and Osborne , who showed that for such wandering domains most boundary points lie in $BU(f)$ in the sense of harmonic measure. 

Singular values play an important role in the iteration of entire functions. Hence, among all transcendental entire functions, the classes of functions where we have some control on the singular set tend to have nicer properties. Let us introduce~the class of \textit{finite-type} transcendental entire functions, also known as the \textit{Speiser class},
$$
\mathcal S:=\{f \text{ transcendental entire function}\, :\, S(f) \text{ is finite}\}.
$$ 
Eremenko and Lyubich \cite[Theorem~3]{eremenko-lyubich92} and, independently, Goldberg and Keen \cite[Theorem~4.2]{goldberg-keen86} were able to modify Sullivan's proof to show that functions in the class $\mathcal S$ have no wandering domains. Eremenko and Lyubich \cite{eremenko-lyubich92} also \vseven{considered} the dynamics of the larger class of \textit{bounded-type} transcendental entire functions, 
$$
\mathcal B:=\{f \text{ transcendental entire function}\, :\, S(f) \text{ is bounded}\},
$$
which is now known as the \textit{Eremenko-Lyubich class}. They proved that if $f\in\mathcal{B}$, then $I(f)\subseteq J(f)$  \cite[Theorem~1]{eremenko-lyubich92}. In particular, this means that functions in the class $\mathcal B$ do not have escaping wandering domains. It was a long standing question whether functions in the class $\mathcal B$ could have wandering domains \vseven{at all}.

Recently, Bishop \cite{bishop15} developed a \vseven{revolutionary} new technique to construct transcendental entire functions in the classes $\mathcal S$ and $\mathcal B$ known as \textit{quasiconformal folding}. Given a bipartite tree $T$ with certain bounded geometry conditions and, for every component $\Omega$ of $\C\setminus T$, let $\tau_\Omega:\Omega\to\mathbb H_r=\{z\in\C : \text{Re}\,z > 0\}$ be a conformal map such that $\text{diam}\,\tau_\Omega(e)>\pi$ for every edge $e\in T\cap \partial \Omega$. Then \cite[Theorem~1.1]{bishop15} provides a function $f\in \mathcal S$ and a $K$-quasiconformal map $\phi$ such that $f=(\cosh \circ\, \tau)\circ \phi^{-1}$ outside a thin neighbourhood \vseven{$T(r_0)$} of $T$. The function~$f$ has critical values at $\pm 1$ and no asymptotic value. Bishop also gave a more general version of this construction, namely \cite[Theorem~7.2]{bishop15}, that replaces the tree $T$ by a graph and produces functions in the class $\mathcal B$ with critical points of arbitrarily high degree and allows \vseven{for} asymptotic values. Using this \redf{result}, Bishop \cite[Theorem~17.1]{bishop15} constructed the first example of a function $f\in\B$ with a wandering domain, which is \vseven{of} oscillating \vseven{type}. See  also Bishop's note \textit{Corrections for quasiconformal folding}\footnote{\url{http://www.math.stonybrook.edu/~bishop/papers/QC-corrections.pdf}}, which contains some clarifications about this construction.

Since then, several modifications of Bishop's construction by other authors have appeared. Fagella, Godillon and Jarque \cite{fagella-godillon-jarque15} produced two functions $f,g\in\mathcal B$ without wandering domains such that $f\circ g$ has an oscillating wandering domain. Lazebnik \cite{lazebnik17} proved that the wandering domains in Bishop's construction are bounded, and modified it to create a function $f\in\mathcal B$ with an oscillating wandering domain $U$ such that $f^n(U)$ is unbounded for all $n\in\N$. \vseven{Lazebnik \cite{lazebnik17} also constructed a function in $\mathcal B$ with a wandering domain $U$ such that $L(U)$ is an uncountable set.} Finally, Fagella, Jarque and Lazebnik \cite{fagella-jarque-lazebnik} created a function $f\in\mathcal B$ with an oscillating wandering domain \vseven{on whose orbit $f$ is univalent}. 

We define the \textit{order} of a transcendental entire function $f$ as
$$
\rho(f):=\limsup_{r\to+\infty}\frac{\log\log M(r,f)}{\log r}\in [0,+\infty],
$$
where, for $r>0$, $M(r,f):=\max_{|z|=r}|f(z)|$. Functions of finite order in the class $\mathcal B$ play an important role as Rottenfusser, R\"uckert, Rempe and Schleicher \cite[Theorem~1.2]{rrrs11} proved that the strong Eremeko conjecture holds for such functions and compositions thereof, namely every escaping point can be connected to $\infty$ by a curve of points that escape uniformly (see also \cite{baranski07} for the disjoint type case). Furthermore, Bara\'nski, Jarque and Rempe \cite{baranski-jarque-rempe12} showed that the escaping set of such functions contains an uncountable collection of curves known as a Cantor bouquet. \blue{Note that for a general transcendental entire function, Eremenko \cite{eremenko89} proved that the components of $\overline{I(f)}$ are all unbounded and conjectured that, in fact, all the components of $I(f)$ are unbounded, but this remains an open question.}

\vseven{Observe that} the function $f$ from \cite[Theorem~17.1]{bishop15} has infinite order as $f=g\circ \phi^{-1}$ with $g(z)=\cosh(\lambda\sinh z)$ for $z\in S_+\setminus T(r_0)$, which contains a horizontal band of fixed height around $\mathbb R_+$, and $\phi(z)=z+o(1)$ as $z\to \infty$, and possibly even faster growth on other components \vseven{of $\C\setminus T$}. Our main theorem provides the first example of a function of finite order in the class $\mathcal B$ that has a wandering domain.

\begin{thm}
\label{thm:main}
There exists a transcendental entire function $f\in\mathcal{B}$ of finite order with an oscillating wandering domain. 
\end{thm}

To prove this result, we construct a transcendental entire function $f\in\mathcal B$ of order~$1$ that is a modification of the function $g(z)=2\cosh z$. Note that $J(g)=\C$. We give a brief outline of the proof, which follows some of the ideas from \cite{bishop15} but is much simpler as it does not require quasiconformal folding and uses quasiregular interpolation instead. \blue{Geometrically, we collapse many simple critical points of $g$\vseven{, which are located at $\pi i\Z$,} into critical points of higher degree to obtain a certain contraction, and then adjust the position of the critical values which is a much finer tuning.} First, we produce a quasiregular map $g_{\mathbf{w}}$, \vseven{with} $\mathbf{w}=(w_n)_n$, that equals~$g$ on $\C\setminus E$, where $E$ is a disjoint union of squares~$\{E_n\}_n$ centered along the imaginary axis (see~Figure~\ref{fig:setup}). Inside each square~$E_n$, we place a disc~$D_n$ on which \blue{$g_\mathbf{w}$} is a power map that has a critical point of multiplicity~$d_n$ with critical value $w_n\in  \blue{\mathbb D(\tfrac{1}{2},\tfrac{1}{8})}$. To join these two holomorphic functions\vseven{,} in the doubly connected regions \vfive{$\{E_n\setminus D_n\}_n$} we use quasiregular interpolation. At this point it is important to make sure that the \blue{dilatation} of the interpolating map is idependent of $d_n$. Then, it follows from the Measurable Riemann Mapping Theorem that there is an entire function $f_{\mathbf{w}}$ and a quasiconformal map  $\phi_{\mathbf{w}}$ such that $f_{\mathbf{w}}=g_{\mathbf{w}}\circ\phi^{-1}_{\mathbf{w}}\in\mathcal{B}$ \blue{as $S(f_\mathbf{w})=\{\pm 2\}\cup
\overline{\{w_n\}}_n$ \redf{and the sequence $\{w_n\}_n$ converges to a point $w_\infty\in \overline{\D}(\tfrac{1}{2},\tfrac{1}{8})$ (see Remark~\ref{rmk:w-infty})}. \blue{In order to show that $f_{\mathbf{w}}$ has an oscillating wandering domain $U$, we need to estimate the quasiconformal map~$\phi_{\mathbf{w}}$.} To that end, the main ingredient will be a recent result on quasiconformal maps by Shishikura \cite[Lemma~9]{shishikura} (see Lemmas \ref{lem:qc-estimate1}, \ref{lem:qc-estimate2} and \vseven{\ref{lem:qc-estimate3}}, which are of independent interest). \vseven{In order to construct a wandering domain, we consider a specific inverse orbit $U_n$ of $\frac{1}{2}D_n$ 
(more precisely, $\phi_\mathbf{w}(\frac{1}{2}D_n)$) near $x_0=\frac{1}{2}$, with center $c_n(\mathbf{w})$ corresponding to the critical point in $D_n$. 
Our goal is to show that $f_\mathbf{w}^{n+2}(U_n) \subseteq g_\mathbf{w}(\tfrac{1}{2}D_n) \subset U_{n+1}$ by giving estimates on 
the size of the domains $\{U_n\}_n$ and by choosing the parameter $\mathbf{w}$ so that the critical value $w_n$ of $g_\mathbf{w}\vert_{D_n}$ hits 
the next center $c_{n+1}(\mathbf{w})$ for large $n$. The latter problem (a simultaneous shooting problem) is solved using 
Brouwer's Fixed Point Theorem.} \redf{Observe that the function $f_\mathbf{w}$ has only finitely many grand orbits of critical points and the wandering domain $U$ satisfies that} \redf{$L(U)=\{f_\mathbf{w}^n(w_\infty)\}_n\cup\{\infty\}$}}. 

\blue{We would like to emphasize that the map from our construction is much more simple than the previous example involving quasiconformal folding. We only alter the base map $g(z)=2\cosh z$ inside the squares \vfive{$\{E_n\}_n$}, and the parameters \vfive{$\{d_n\}_n$ (and also $\{R_n\}_n,\{h_n\}_n$) are given in the beginning with explicit formulas}. Using quasiconformal folding allows for a small support of the Beltrami coefficient $\mu_\phi$, but \vfive{one} often lose control over the order of the resulting function $f$, while in our case the support of $\mu_\phi$ is large in the Euclidean sense, but sparse, and we can achieve finite order. In \cite{bishop15}, since the logarithmic area of $\supp\mu_\phi$ decays exponentially, it follows from Dynkin's Theorem that \vfive{one} can normalise $\phi$ so that $\phi(z)=z+o(1)$ as $z\to\infty$. In contrast, our \vseven{quasiconformal mapping \vseven{$\phi_\mathbf{w}$} may have $\phi_\mathbf{w}(z)-z$ unbounded}, but we developed new estimates that allow us to have the necessary control over \vseven{$\phi_\mathbf{w}$}, \vseven{for example, we show that $\left|\log \frac{\phi_\mathbf{w}(z)}{z}\right|<\varepsilon$} (see Section~\ref{sec:qc-estimate}).}

\redf{Our shooting method is quite different from \cite[Theorem~17.1]{bishop15} (as well as \redf{\cite{fagella-godillon-jarque15,lazebnik17,fagella-jarque-lazebnik}}). Their argument to choose $\mathbf{w}$ is based on a sequence of successive perturbations so that in each step we obtain a new quasiconformal map $\phi$, and hence it is key to check that the new map $f=g\circ\phi^{-1}$ still satisfies the inclusion relations of domains that we set up in all the previous steps. This is guaranteed by the thinness of the support of the quasiconformal maps used to choose the position of the critical values $\{w_n\}_n$, which is achieved by giving a lower bound for the sequence $\{d_n\}_n$. Note that in order to avoid any circularity, it is important to first fix the sequence $\{d_n\}_n$ and then start choosing $\{w_n\}_n$ inductively (see Bishop's note\footnotemark[1]). To do so, it is key to obtain a uniform estimate for the inner radius of the sets $\{U_n\}_n$ in the closure of all quasiconformal maps $\phi$ that arise from this construction with sufficiently large sequences $\{d_n\}_n$ and $\{w_n\}_n$ in a compact polydisc, which is a compact subset of the space of homeomorphisms from $\C$ to~$\C$. In this paper we determine the $\{d_n\}_n$ explicitly at the beginning, and then use a shooting method to choose $\{w_n\}_n$ (see Section~\ref{sec:shooting}). 
%the authors of the present paper cannot foresee how to do that if it requires to keep changing the multiplicities $\{d_n\}_n$ of the critical points of the construction as suggested in these papers.
%Once the $\{d_n\}_n$ are fixed, and after checking that the construction has certain continuity properties, one can use a fixed point theorem to choose the sequence $\mathbf{w}$ as in our construction. 
Independently, for a different problem concerning meromorphic functions, Bishop and Lazebnik \cite{bishop-lazebnik} have developed a similar strategy using a fixed point theorem combined with Bishop's quasiconformal folding construction.} 

%The map $G$ will equal $g(z)=2\cosh z$ on $\C\setminus E$, where $E$ is a disjoint union of squares $E_n$ centered along the imaginary axis, and each square $E_n$ will contain a disc $D_n$ on which $g$ will be a power map which will create a critical point of multiplicity $d_n$. The goal of this section is to show that we can interpolate these two maps in the doubly connected domains $E_n\setminus D_n$ with a quasiconformal map independently of the values of $d_n$ that we choose on each disc~$D_n$ for $n\in\mathbb N$.

%To prove this result, \red{Outline the proof a bit, explain the main tool is quasiregular interpolation and, more precisely, linear interpolation.}

The \textit{lower order} of a transcendental entire function can be defined replacing the $\limsup$ in the definition of the order by $\liminf$. By a result of Heins \cite{heins48}, functions in the class $\mathcal{B}$ have lower order (and, hence, also order) greater than or equal to $1/2$. In the following theorem, we adapt the previous construction to obtain functions with order $p/2$ for $p\in\N$ \vseven{and, in particular, we achieve order} $1/2$, which is the smallest possible value in \vseven{class $\mathcal B$}.

\begin{thm}
\label{thm:oder-half}
For every $p\in \mathbb{N}$, there exists a transcendental entire function in the class $\mathcal{B}$ of order $p/2$ with an oscillating wandering domain. 
\end{thm}

The reason why we only obtain orders that are half integers is that our construction is based on the function $g(z)=2\cosh z$ which is an even function and therefore $\vseven{\breve{g}_p}(z)=2
\cosh(z^\frac{p}{2})$ is an entire function for every $p\in\N$.% \vseven{Note that} the functions \vseven{$f_{p,\mathbf{w}}$} that we obtain are not symmetric with respect to the real line.

%Entire functions with no zeros often have more restricted properties than general entire functions. For example, Baker \cite[Theorem~1]{baker87} proved that all but one of their Fatou components must be simply connected, and hence they do not have multiply connected wandering domains. Such functions are all of the form $f=\exp \circ\,g$ for some non-constant entire function $g$. P\'olya \cite{polya} showed that if an entire function with no zeros has finite order, then necessarily it is of the form $f=\exp\circ\,P$ with $P$ being a polynomial, and hence their order is a natural number. It is easy to see that these functions have at most as many critical values as the degree of $P$ and $0$ is the only finite asymptotic value. Thus, entire functions of finite order with no zeros belong to the class $\mathcal S$, and have no wandering domains. The first example of an entire function that is a holomorphic self-map of $\C^*$ and has a wandering domain was produced by Baker~\cite[Theorem~4]{baker87}. Further examples of general holomorphic self-maps of $\C^*$ with wandering domains were given in \cite{martipete3}. We construct an entire function $g\in\mathcal B$ so that $f=\exp\circ\, g\in\mathcal{B}$ has an oscillating wandering domain. By the previous remarks, $f$ must have infinite order. 

%\begin{thm}
%\label{thm:nozeros}
%There exists a function $F\in\mathcal{B}$ with no zeros that has an oscillating wandering domain. The function $F$ has infinite order but finite hyper-order.
%\end{thm}

In \cite{fagella-godillon-jarque15}, Fagella, Godillon and Jarque showed that the function from \cite[Theoprem~17.1]{bishop15} has precisely two grand orbits of oscillating wandering domains, and no other wandering domain. We can modify our construction to obtain functions of finite order in the class $\mathcal B$ with any number of grand orbits of oscillating wandering domains, including functions with infinitely many grand orbits of wandering domains. 

\begin{thm}
\label{thm:infinitely-many}
For every $q\in\mathbb{N}\cup\{\infty\}$, there exists an entire function of finite order in the class $\mathcal{B}$ with exactly $2q$ grand orbits of oscillating wandering domains. 
\end{thm}

The functions from Theorem~\ref{thm:infinitely-many} all have an even number of grand orbits of wandering domains and this is due to the fact in the construction we use a slightly different function to that in Theorems~\ref{thm:main} and \ref{thm:oder-half}, which is symmetric with respect to the real line but is not an even function as before. However it is also possible to construct functions with odd numbers of wandering domains (see Remark~\ref{rmk:odd-number-wd}).
%To prove this, we modify the construction from Theorem~\ref{thm:oder-half}. The same arguments as in \cite{fagella-godillon-jarque15} ensure that all the wandering domains of these functions come from the construction, allowing us to prescribe the exact number of grand orbits of wandering domains. 

%%Lazebnik~\cite{lazebnik17} proved that all the wandering domains in Bishop's example are bounded sets. We devote the last part of the paper to show that every wandering domain of the functions from Theorem~\ref{thm:infinitely-many} is also bounded.

%%\begin{thm}
%%\label{thm:bounded-wd}
%%There exists an entire function of finite order in the class $\B$ with a wandering domain $W$ such that all the components of the grand orbit of $W$ are bounded sets.
%%\end{thm}

%%Finally, observe that if we want to obtain a function of infinite order in the class~$\mathcal B$ with a wandering domain, we can take a function $f$ of finite order provided by Theorem~\ref{thm:infinitely-many} and compose it with itself. Since $F(f^2)=F(f)$, if $f$ has $2q$ grand orbits of wandering domains, then $f^2$ has $4q$ grand orbits of wandering domains.% and if the wandering domains of $f$ were all bounded sets, so will be those of $f^2$.

%Although the iteration of transcendental (entire) functions dates back to the times of Fatou \cite{fatou26},\\

%\cite{ahlfors73}

\vspace{5pt}

\noindent 
\textbf{Structure of the paper.} In Section~\ref{sec:setup} we define the main quantities and sets involved with the construction, and also prove some growth estimates of the map $g(z)=2\cosh z$ which is the base for the construction. We use quasiregular interpolation to modify $g$ to a quasiregular function $g_{\mathbf{w}}$ such that $f_{\mathbf{w}}=g_{\mathbf{w}}\circ \phi^{-1}_{\mathbf{w}}\in\mathcal{B}$ has the desired dynamics; see Figure~\ref{fig:setup} in p.~\pageref{fig:setup}. The proof of our main result, Theorem~\ref{thm:main}, is in Section~\ref{sec:proof-main-thm}. In order to keep the outline of the proof of Theorem~\ref{thm:main} as clear as possible, we placed the details of the construction of the quasiregular interpolation from Section~\ref{sec:interpolation} in Appendix~\ref{sec:a1}, and the proofs of the estimates about the integrating quasiconformal map \vseven{$\phi_\mathbf{w}$} from Section~\ref{sec:qc-estimate} are given in Appendix~\ref{sec:proof-of-estimates}. For those familiar with Bishop's construction, the domains $\{U_n\}_n$ are defined in Section~\ref{sec:domains}; see Diagram~1 in p. \pageref{diag:1} \redf{and Figure~\ref{fig:setup2} in p. \pageref{fig:setup2}}. In Section~\ref{sec:derivatives} we estimate the inner radius of $U_n$ and in Section~\ref{sec:shooting} we define $\mathbf{w}$. The modifications to obtain functions of order $p/2$ for $p\in \N$ (Theorem~\ref{thm:oder-half}) and functions with $2q$ grand orbits of wandering domains for $q\in\N\cup\{\infty\}$ (Theorem~\ref{thm:infinitely-many}) are done in Sections~\ref{sec:oder-half} and \ref{sec:infinitely-many}, respectively. %Finally, the proof of Therem~\ref{thm:bounded-wd} regarding the boundedness of these sets is in Section~\ref{sec:bounded-wd}.

\vspace{5pt}

\noindent
\textbf{Notation.} Throughout the paper, $\Z^*:=\Z\setminus\{0\}$ and, for $N\in\N$, we define\linebreak $\N_N:=\N\setminus\{1,\hdots,N-1\}$ and $\Z_N:=\Z\setminus\{-N+1,\hdots,N-1\}=\N_N\cup(-\N_N)$. Let $\mathbb R_+\!:=\!\{x\in \mathbb{R}:x\!>\!0\}$ and \mbox{$\mathbb R_-\!:=\!\{x\in \mathbb{R}:x\!<\!0\}$}. For $x\in\mathbb R$, we use $\left \lfloor{x}\right \rfloor$ to denote the largest integer not exceeding $x$. For $z_0,z_1\in \mathbb C$, we use $[z_0,z_1]$ for the line segment joining $z_0$ to $z_1$. We define the Euclidean discs of centre $z_0\in \C$ and radius $R>0$ by
%$$
%A(r_1,r_2):=\{z\in \mathbb C\ :\ r_1<|z|<r_2\}  ~\mbox{ and }~ \overline{A}(r_1,r_2):=\{z\in \mathbb C\ :\ r_1\leqslant |z|\leqslant r_2\},
%$$
$$
\mathbb{D}(z_0,R):=\{z\in\C\, :\, |z-z_0|<R\} 
$$
and put $\barD(z_0,R):=\overline{\D(z_0,R)}$, $\mathbb D:=\D(0,1)$ and $\mathbb D_R:=\D(0,R)$. \blue{Suppose that} $D=\D(z_0,R)$ and $\lambda\in\R_+$, then $\lambda D:=\D(z_0,\lambda R)$. %For $a,b>0$, we define the region bounded by the ellipse of semi-axes $a,b$ by
%$$
%\mathbb{E}(a,b):=\left\{z\in\C\ :\ \frac{(\text{Re}\,z)^2}{a^2}+\frac{(\text{Im}\,z)^2}{b^2}< 1\right\},
%$$
%and it will be convenient to use $\mathbb{E}_R:=\mathbb{E}(2\cosh R,2\sinh R)$. 
We denote the upper half-plane by $\mathbb H:=\{z\in\C\ :\ \textup{Im}\,z>0\}$ and also use $\mathbb H_l:=\{z\in\C\, :\, \textup{Re}\,z<0\}$ and $\mathbb H_r:=\{z\in\C\, :\, \textup{Re}\,z>0\}$ for the left and right half-planes, respectively. 

If $\Omega\subseteq \C$ is a hyperbolic domain, then $\textup{dist}_\Omega(z,w)$ is the hyperbolic distance between two points $z,w\in \Omega$. We define the hyperbolic disc of centre $z_0\in \Omega$ and radius $R>0$ by
$$
\D_\Omega(z_0,R):=\{z\in\Omega\, :\, \textup{dist}_\Omega(z,z_0)<R\}.
$$
We will also consider the cylinder $\mathcal C:=\mathbb{C}/2\pi i\mathbb{Z}$ and the distance on $\mathcal C$ given by 
$$
|w|_\mathcal C:=\inf_{n\in\Z} |w+2\pi n i|,\quad \text{ for } w\in \mathcal C.
$$

\vspace{10pt}

\noindent
\textbf{Acknowledgments.} \blue{The authors would like to thank Christopher Bishop, N\'uria Fagella, Xavier Jarque, Masashi Kisaka and Kirill Lazebnik for helpful discussions on this topic.}

\section{Setup and properties of the base map $g(z)=2\cosh z$}

\label{sec:setup}

In order to prove Theorem \ref{thm:main}, we first construct a function $f\in \mathcal B$ of order $1$, that is, the case $p=2$, and then we will modify it to obtain functions of order $p/2$ for any $p\in\mathbb N$. For the convenience of the reader who may be familiar with the construction of the oscillating wandering domain from \cite[Theorem~17.1]{bishop15} (and also \cite{fagella-godillon-jarque15}), we use a similar notation wherever possible.

The function that is the base of our construction is $g(z):=2\cosh z$. Note that $g(z)=e^z+e^{-z}\sim e^z$ for $\text{Re}\,z\gg 1$. Define the horizontal band
$$
S_+:=\{z\in\mathbb{C}\, :\, \text{Re}\,z>0,\ |\text{Im}\,z|<\pi\}
$$
and, for $x>1$, consider the rectangle
$$
Q(x):=\{z\in \C\, :\, |\text{Re}\,z-x|<1,\ |\text{Im}\,z|<\pi\}\subseteq S_+.
$$
Later we will modify the function $g$ on some squares centered along the imaginary axis, but it will be important that this modification takes place outside the horizontal band $S_+$ where the function will equal $g(z)=2\cosh z$. In the next lemma we collect a few basic properties of the function $g$.

\begin{lem}
\label{lem:cosh} 
Let $g(z):=2\cosh z$ and, for $R>0$, define 
$$
\mathcal{E}_R:=\mathcal{E}(2\cosh R,2\sinh R)=\left\{z\in\C\, :\, \frac{(\textup{Re}\,z)^2}{(2\cosh R)^2}+\frac{(\textup{Im}\,z)^2}{(2\sinh R)^2}< 1\right\}.
$$ 
\begin{enumerate}
\item[(i)] The critical points of $g$ are $\pi i \mathbb Z$, and the critical values of $g$ are $\{-2,2\}$.
\item[(ii)] The restriction $g\vert _{S_+} :S_+\to \C\setminus (-\infty,-2]$ is a conformal isomorphism.
\item[(iii)] For $x>1$, the function $g$ maps the rectangle $Q(x)$ conformally to the region $\mathcal{E}_{x+1}\setminus (\overline{\mathcal{E}}_{x-1}\cup \mathbb R_-)$.
\item[(iv)] Let $x_*:=\tfrac{5}{3}$. If $x\geqslant x_*$, then $g(Q(x))\supseteq \{z\in\C\ :\ \tfrac{1}{2}e^x<|z|<2e^x\}\setminus \mathbb R_-$.
\end{enumerate}
\end{lem}
\begin{proof}
The statements (i) to (iii) follow immediately from the fact that the map $g$ is the composition of the exponential function and twice the Joukowsky transform $\zeta\mapsto \zeta+\zeta^{-1}$. For (iv), since $\cosh x>\sinh x$ for all $x>0$, we have to check that, for $x>x_*$,
$$
2e^x<2\sinh(x+1)=e^{x+1}-e^{-(x+1)}\quad\text{and}\quad \tfrac{1}{2}e^x>2\cosh(x-1)=e^{x-1}+e^{-(x-1)}
$$
(see Figure~\ref{fig:ellipses}). Thus, it suffices that
$$
x_*>\tfrac{1}{2}\bigl(-\log(e-2)+\log(2)+2\bigr)
$$
to obtain the desired result.
\end{proof}

\begin{figure}[h!]
%\vspace{5pt}
\def\svgwidth{.9\linewidth}
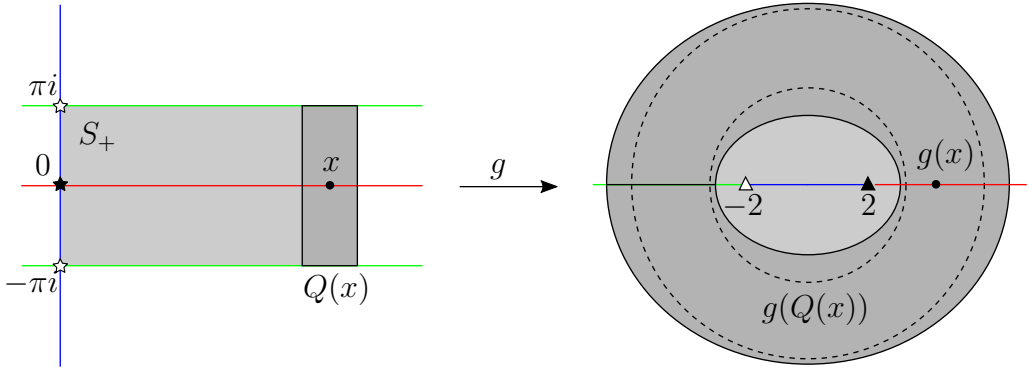
%\vspace{5pt}
\caption{The base map $g(z)=2\cosh z$ and the set $g(Q(x))$ from Lemma~\ref{lem:cosh} for $x>x_*$.}% Note that this is only a sketch as for such values of $x$ it would be difficult to distinguish the ellipses from circles.}
\label{fig:ellipses}
\end{figure}

%\begin{figure}[h]
%\vspace{5pt}
%\def\svgwidth{.5\linewidth}
%\input{ellipses.pdf_tex}
%\vspace{5pt}
%\caption{The set $g(Q(x))$ from Lemma~\ref{lem:ellipses} for $x>x_*$. Note that this is only a sketch as for such values of $x$ it would be difficult to distinguish the ellipses from circles.}
%\label{fig:ellipses}
%\end{figure}

From this point, to simplify the notation, whenever we write $g^{-1}$, we refer to the inverse branch $(g\vert _{S_+})^{-1} :\C\setminus (-\infty,2]\to S_+$.

\begin{dfn}[Reference orbit $(x_n)_n$, critical orbit $(v_n)_n$]
\!Let $x_0\!:=\!\tfrac{1}{2}$ and define \mbox{$x_n\!:=\!g(x_{n-1})$} for all $n\in\N$; we call $(x_n)_n$ the \textit{reference orbit}. We also consider the orbit of the critical point $v_0=0$ and define $v_n=g(v_{n-1})$ for all $n\in \N$; we refer to $(v_n)_n$ as the \textit{critical orbit}, although $g$ has many other critical points.
\end{dfn}

In the next lemma we show that indeed the orbits $(x_n)_n$ and $(v_n)_n$ escape and the distance between their points grows exponentially fast.

\begin{lem}
\label{lem:escaping-real-orbits}%\margin{$\phi(\R)\subseteq \R$ if $\text{supp}\,\mu_g$ is symmetric\\ wrt $\R$?}
Let $g(z)=2\cosh z$. Then $\R\subseteq I(g)\subseteq J(g)$ and so do all the $2\pi i\Z$ translates of $\R$. The reference orbit $(x_n)_n$ is a strictly increasing sequence and we have $v_n<x_n<v_{n+1}$ with \blue{$v_{n+1}-x_n>1$ and $x_n-v_n>1$} for all $n\geqslant 3$. 
\end{lem}
\begin{proof}
Since $g$ has only two critical values at $\pm 2$ and no asymptotic value, we have $g\in\mathcal B$ and $I(g)\subseteq J(g)$ by \cite[Theorem~1]{eremenko-lyubich92}. We show that there exists $A>0$ such that $g(x)-x>A$ for all $x\geqslant 1$, and this implies that $g^n(x)>x+An\to +\infty$ as $n\to\infty$. Since $g'(x)-1=2\sinh x-1>2x-1>0$ for all $x>\tfrac{1}{2}$, the previous statement holds taking $A:=g(1)-1>1$. For $x<1$, observe that $g(x)>2$, and then the previous case applies. So $\R\subseteq I(g)$.

Since $v_0<x_0<v_1$ and $g$ is a strictly increasing function, we have $v_n<x_n<v_{n+1}$ for all $n\in\N$. It is easy to check that $v_2-x_1>x_1-v_1>\tfrac{1}{4}$ and $g'(x)>2$ for all $x>1$, so \blue{$v_{n+1}-x_n > 1$ and $x_n-v_n>1$} for all $n\geqslant 3$.
\end{proof}

Next we define some quantities and sets associated to the reference orbit. 

\begin{dfn}[Collections of sets $\{D_n\}_n$, $\{E_n\}_n$ and $\{Q_n\}_n$]
\label{dfn:hn-dn-Rn-En-Dn-Qn}
For $n\in\N$, let
\begin{equation}
\label{eq:dn-Rn-hn}
d_n:=\left\lfloor \frac{x_{n+1}}{x_n}\right\rfloor,\quad R_n:=\left(d_n-\tfrac{1}{3}\right)\pi,\quad h_n:=2\pi\left\lfloor\frac{x_{n+1}+\pi}{2\pi}\right\rfloor.
\end{equation}
For $n\geqslant 3$, we define $Q_n:=Q(x_n)\subseteq S_+$, 
$$
E_{\pm n}:=\{z\in\C\ :\ |\textup{Re}\,z|<2d_n\pi,\quad |\textup{Im}\,z\mp h_n|<2d_n\pi\}\subseteq \C\setminus S_+,
$$
and $D_{\pm n}:=\D(\pm ih_n,R_n)\subseteq E_{\pm n}$ (see Figure~\ref{fig:setup}).
\end{dfn}

In the following lemma we will provide some further growth estimates about the sequences that we have just introduced. 

\begin{lem}
\label{lem:growth-estimates}
The reference orbit $(x_n)_n$ satisfies that %$\log x_{n+1}=x_n+o(1)$ as $n\to \infty$, and 
\red{$x_{n+1}>(n!)^2x_n$ for all $n\in\N$}. Thus,
\begin{equation}
\label{eq:growth-estimate-inequalities}
\red{\frac{R_n}{h_n}<\frac{2\pi}{(n!)^2}\quad \text{ and } \quad \frac{h_{n+1}}{h_n}>(n!)^2,\quad \text{ for all } n\geqslant 3.}
\end{equation}

\pagebreak

\noindent In particular,
\begin{equation}
\label{eq:growth-estimate-infinitely-many}
\red{h_n+3R_n+6nR_n<\tfrac{6}{(n!)^2}(h_{n+1}-3R_{n+1}),\quad \text{ for all } n\geqslant 3,}
\end{equation}
and
\begin{equation}
\label{eq:growth-estimate-finitesum}
\sum_{n=1}^\infty \frac{R_n}{h_n}<+\infty.
\end{equation}
%$h_{n+1}-h_n>2\pi(d_n+d_{n+1})$ for all $n\geqslant 1$,\vspace{5pt}
\end{lem}
\begin{proof}
%Observe that 
%$$
%\log x_{n+1}=\log (e^{x_n}+e^{-x_n})=x_n+\log(1+e^{-2x_n})=x_n+o(1)\quad \text{ as } n\to \infty,
%$$
%and using that $x_n>n^2$ for $n\geqslant 3$, we obtain that
%$$
%\frac{x_{n+1}}{x_n}>\frac{e^{x_n}}{x_n}>\frac{e^{n^2}}{n^2}>(n+1)^2,\quad \text{ for all }n\geqslant 3.
%$$
\red{First, observe that $x_n>(n!)^2$ for all $n\in\N$. Indeed, it is easy to check that this holds for $n\leqslant 3$ directly. For $n\geqslant 4$, by induction, we have
$$
x_{n}>e^{x_{n-1}}>x_{n-1}^2>((n-1)!)^4>(n!)^2,
$$
and our claim follows. Thus,
$$
\frac{x_{n+1}}{x_n}>\frac{e^{x_n}}{x_n}>x_n>(n!)^2,\quad \text{ for all } n\in\N.
$$
One can also show easily that}
$$
\frac{R_n}{h_n}\leqslant \frac{\left(\frac{x_{n+1}}{x_n}-\frac{1}{3}\right)\pi}{x_{n+1}-\pi}< \frac{\pi x_{n+1}}{x_n(x_{n+1}-\pi)}<\frac{2\pi}{x_n}<\red{\frac{2\pi}{(n!)^2}},\quad \text{ for all } n\geqslant 3,
$$
and 
$$
\frac{h_{n+1}}{h_n}\geqslant \frac{x_{n+2}-\pi}{x_{n+1}+\pi}>\red{\frac{((n+1)!)^2x_{n+1}-\pi}{x_{n+1}+\pi}>(n!)^2},\quad \text{ for all } n\geqslant 3.
$$
Finally, it follows from these two inequalities that
$$
\red{\frac{h_n+3(1+2n)R_n}{h_{n+1}-3R_{n+1}}<\frac{h_n\left(1+\frac{6\pi(1+2n)}{(n!)^2}\right)}{h_{n+1}\left(1-\frac{3}{(n!)^2}\right)}<\frac{1}{(n!)^2}\frac{12}{11}\left(1+\frac{7\pi}{6}\right)<\frac{6}{(n!)^2}, \text{ for } n\geqslant 3,}
$$
and the series in \eqref{eq:growth-estimate-finitesum} is finite.
%
%To prove the first inequality, we will show that
%$$ 
%h_{n+1}-h_n\geqslant x_{n+2}-x_{n+1}-2\pi > 2\pi \left(\frac{x_{n+1}}{x_n}+\frac{x_{n+2}}{x_{n+1}}\right) \geqslant 2\pi(d_n+d_{n+1}),\quad \text{for all } n\in\N.
%$$
%The first and last inequalities are straightforward, and the middle inequality can be rewritten as
%$$
%x_{n+2}>2\pi x_{n+1}^2>\frac{2\pi x_{n+1}^2}{x_{n+1}-2\pi}\left(\frac{1}{x_n}+\frac{1}{2\pi}+\frac{1}{x_{n+1}}\right),\quad \text{ for all } n\in\N,
%$$
%and follows from the fact that
%$$
%x_{n+1}=2\cosh x_n=e^{x_n}(1+e^{-2x_n})> e^{x_n}, \quad \text{for all } n\in\N.
%$$
%To prove the first inequality, it is easy to check that
%$$ 
%3R_n<3\pi d_n\leqslant 3\pi\frac{x_{n+1}}{x_n}\leqslant\tfrac{1}{2}(x_{n+1}-\pi)\leqslant \tfrac{1}{2}h_n,\quad \text{for all } n\geqslant 3.
%$$
%Then, since $h_{n-1}<h_n/6$ for $n\geqslant 3$, we have
%$$
%\frac{h_n-(h_{n-1}+3R_n+3R_{n-1})}{h_n+3R_n}\geqslant \frac{\tfrac{1}{2}h_n-\tfrac{3}{2}h_{n-1}}{\tfrac{3}{2}h_n}
%$$
\end{proof}

Observe that inequality \eqref{eq:growth-estimate-infinitely-many} in Lemma~\ref{lem:growth-estimates} implies that $h_{n+1}-2d_{n+1}\pi> h_n+2d_n\pi$ for all $n\geqslant 3$ and thus, the squares $\{E_n\}_{n}$ are disjoint. The next lemma concerns the first $M$ iterates of the map $g$.

\begin{lem} 
\label{lem:initial-segment} 
For $n \geqslant 1$, let $g^{-n}$ be the $n$-th iterate of $g^{-1} = (g|_{S_+})^{-1}$, which is a univalent function that maps $\C \sminus (-\infty, v_n]$ into 
$S_+$. For $M\geqslant 3$, the function $g^{-M}$  is defined on~$\overline{Q}_M$ and $g^{-M}(\overline{Q}_M)\subseteq \D(\tfrac{1}{2}, \tfrac{1}{16})$.  
\end{lem}
\begin{proof} 
Since $g:[0, \infty) \to [v_1, \infty)$ is homeomorphic and monotone, the function
$g^{-1}:\C\setminus (-\infty,2]\to S_+$ maps $\C \setminus (-\infty, v_j]$ onto $S_+ \setminus (-\infty, v_{j-1}]$ for $j \ge 2$ 
and onto $S_+$ for $j=1$. Therefore the function $g^{-n}:\C \setminus (-\infty, v_n] \to S_+$ is defined and univalent for all $n\in\N$. 

Let $M\geqslant 3$, and define
$$
R:=10^4\leqslant x_3-v_3\leqslant x_M-v_M.
$$
By the above, $\psi = g^{-M}$ is defined and univalent on $\D(x_M, R)$, with 
$g^{-M}(x_M)=x_0$ 
and its image does not contain the critical point $0$.  
By Koebe 1/4-Theorem \cite[Corollary~1.4]{pommerenke75}, we have 
$$
\frac{1}{4} |\psi'(x_M)| R \le |0-x_0| = \frac{1}{2}.
$$  
Hence, it follows from the Koebe Distortion Theorem~\cite[Theorem~1.6]{pommerenke75} that if $0< R'< R$, then
$\psi(\D(x_M, R')) \subset \D(x_0, R'')$, 
where 
$$
R'' = |\psi'(x_M)| \frac{R'}{1- (R'/R)^2}.
$$
Now taking $R'=R/32>\sqrt{1+\pi}$, we obtain that 
$$
R''\leqslant \frac{2R'/R}{1- (R'/R)^2}\leqslant \frac{1}{16}
$$ 
and thus $g^{-M}(\overline{Q}_M) \subset \D(\frac{1}{2}, \frac{1}{16})$.  
%Note that for $M \ge 1$, the set of critical values of $g^M$ is 
%$\{v_1, \dots, v_M\} \cup \{-1\}$.   
%Suppose $\D(x_M, R) \subset \C \sminus \{v_1, \dots, v_M\} \cup \{-1\}$.  
%Then there exists and inverse branch $\psi$ 
%of $g^M$ defined in $\D(x_M, R)$ such that $g \circ \psi = id$, $\psi$ is holomorphic and univalent and 
%$\psi(x_M)=x_0$.  As $g^{-1}$ is the inverse of $g|_{S_+}: S_+ \to \C \sminus (-\infty,2]$, 
%$\psi$ must coincide with $g^{-M}$.  
%Its image does not contain the critical point $0$.  
%By Koebe 1/4-theorem \cite{}, we have 
%$\frac{1}{4} |\psi(x_M)| R \le |0-x_0| = \frac{1}{2}$.  
%Hence by Koebe distortion theorem, 
%$\psi(\D(x_M, R_1)) \subset \D(x_0, R_2)$, 
%where $R_2 = |\psi(x_M)| R_1/(1- (R_1/R)^2)$.
%Now with $M=3$, we have $x_M=..$, $v_M=..$ and we can take $R=..$ and $R_1=..$.  
%This implies that $R_2<$ and $g^{-M}(\overline{Q}_M) \subset \D(x_0, ..)$.  
\end{proof}

\section{Quasiregular interpolation and definition of the function $f_{\mathbf{w}}$}

\label{sec:qr-interpolation}
\label{sec:interpolation}

In this section we state the results about how to modify the map \mbox{$g(z)=2\cosh z$} to obtain an entire function $f_{\mathbf{w}}\in\mathcal B$ such that, for an appropriate choice of the sequence $\mathbf{w}\in\D_{3/4}^{\N_N}$, has an oscillating wandering domain. In order to make clear the outline of the proof of the main theorem, we provide all the details in the Appendix~\ref{sec:a1}. 

We start with some basic definitions about quasiconformal mappings, see \cite{branner-fagella14} for a reference on quasiconformal mappings with an emphasis on their applications on holomorphic dynamics.

\begin{dfn}[Quasiconformal map]
Let $\phi:\mathbb{C}\to\mathbb{C}$ be a $\mathcal C^1$ homeomorphism that preserves orientation. We define the \textit{complex dilatation} (or the \textit{Beltrami coefficient}) of $\phi$ at a point $z$ by
$$
\mu_\phi(z):=\frac{\partial_{\overline{z}}\phi(z)}{\partial_{z}\phi(z)}\in\mathbb{D}
$$
and then, the \textit{dilatation} of $\phi$ at a point $z$ is given by
$$
K_\phi(z):=\frac{1+|\mu_\phi(z)|}{1-|\mu_\phi(z)|}.
$$
We say that $\phi$ is a $K$-\textit{quasiconformal} map, $1\leqslant K<+\infty$, if 
$$
K=K(\phi):=\mathop{\mbox{ess~sup}}_{z\in\C} K_\phi(z).
$$
A map $g:\C\to\C$ is $K$-\textit{quasiregular} if and only if $g$ can be expressed as
$$
g=f\circ \phi,
$$
where $\phi:\C\to \C$ is a $K$-quasiconformal map and $f:\phi(\C)\to\C$ is a holomorphic function.
\end{dfn}

Let $\supp \mu_\phi := \{z \in \C : \mu_\phi(z) \neq 0\}$. Observe that $\mu_\phi(z)=0$ (or, equivalently, $K_\phi(z)=1$) if and only if the map~$\phi$ is conformal at $z$, and hence a $K$-quasiconformal map~$\phi$ is conformal in the complement of $\mathop{\textup{supp}} \mu_\phi$. Geometrically, a conformal map sends infinitesimal circles to infinitesimal circles, while a $K$-quasiconformal map sends infinitesimal circles to infinitesimal ellipses with excentricity bounded by $K$. Note that $K$-quasiregular maps also satisfy this property but are not required to be global homeomorphisms. 

We will construct a $K$-quasiregular map that equals the base map $g(z)=2\cosh z$ in the complement of the squares $\{E_n\}_n$ and is a different power map inside each disc $D_n$, followed by a further quasiconformal map. In the doubly connected sets \mbox{$\{E_n\setminus D_n\}_n$} we will interpolate the two holomorphic functions defined on the boundaries $\partial E_n$ and $\partial D_n$ by quasiregular maps. In the following lemma we describe these quasiregular interpolating maps, which have a bounded dilatation that is independent of $d$ (see Figure~\ref{fig:interpolation-annulus}).

\begin{lem}[Cosh-power Interpolation Lemma]
\label{lem:interpolation}
\label{lem:key-interpolation}
Let $d\in\N$ and consider the square 
$$
E:=\{z\in\C\, :\, |\textup{Re}\,z|\leqslant 2d\pi,\ |\textup{Im}\,z|\leqslant 2d\pi\}.
$$
Define $R:=(d-\tfrac{1}{3})\pi$ and let $D:=\D_R$. There exists $K_1\geqslant 1$ independent~of~$d$ and a $K_1$-quasiregular map $G:E\to \overline{\mathcal{E}}_{2d\pi}$ with $\textup{supp}\,\mu_G\subseteq E\setminus D$ satisfying that $G(-z)=G(z)$, $G(\overline{z})=\overline{G(z)}$ and
\begin{equation}
G(z)=\left\{
\begin{array}{ll}
2\cosh z, & \textup{ if } z\in \partial E\cup ((E\cap i\mathbb{R})\setminus D),\vspace{5pt}\\
\left(\dfrac{z}{R}\right)^{2d}, & \textup{ if } z\in \overline{D}.
\end{array}\right.
\end{equation}
\end{lem}

Observe that two definitions of $G$ match at the points $z=\pm Ri\in \partial D\cap i\R$ as
$$
2\cosh(\pm Ri)=e^{\pm d\pi i}e^{\mp \tfrac{\pi}{3}i}+e^{\mp d\pi i}e^{\pm \tfrac{\pi}{3}i}=(-1)^d=\left(\frac{\pm Ri}{R}\right)^{2d}.
$$
The proof of Lemma~\ref{lem:interpolation} will be given in the Appendix~\ref{sec:a1} (see p. \pageref{proof:lem:interpolation}). 

\begin{figure}[h!]
%\vspace{30pt}
\def\svgwidth{.9\linewidth}
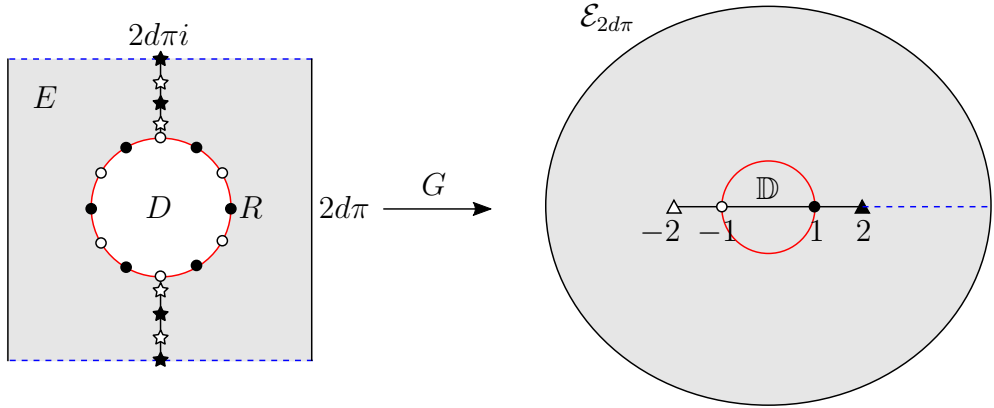
\caption{The quasiregular map $G:E\setminus D\to \overline{\mathcal{E}}_{2d\pi}$ from Lemma~\ref{lem:interpolation} that interpolates $g(z)=2\cosh z$ on $\partial E$ and $h(z)=(z/R)^{2d}$ on $\partial D$ with $d=3$.}
\label{fig:interpolation-annulus}
\end{figure}

\begin{rmk}
Note that Lemma~\ref{lem:interpolation} is similar to Bishop's $\exp$-$\cosh$ Interpolation Lemma \cite[Lemma~7.1]{bishop15}\vseven{, which claims that given a partition of $i\R$ into two collections of intervals $\mathcal{J}_1$ and $\mathcal{J}_2$ with endpoints in $i\pi\Z$, there exists a quasiregular mapping $\nu:\mathbb{H}_r\to\C\setminus[-1,1]$ such that $\nu(z) = \cosh z$ on $\mathcal{J}_1$ and $\nu(z) = \exp z$ on $\mathcal{J}_2$. Unfortunately, 
such a quasiregular mapping cannot exist because of the criticality/non-criticality of the maps. In fact, suppose, for example, $\nu(iy) = \exp (iy)$ for $y\ge 0$ and $\nu(iy) = \cosh (iy)$ for $y <0$ and $\nu_j$ is quasiregular in $\mathbb{H}_r$. If the local map $\nu_j$ near $0$ covers the upper half unit disk $n$ times and its complement $(n+1)$ times for some $n \ge 0$, let
$$
\alpha = \frac{\pi}{\frac{3\pi}{2}(n+1) + \frac{\pi}{2}n} = \frac{2}{4n+3},
$$
then, by composing the branched covering map $w \mapsto (w-1)^\alpha$, $h(z) = (\nu_j(z)-1)^{\alpha}$ induces a quasiconformal map onto a smooth Jordan domain near $0$,  
hence it is quasisymmetric.  Therefore the H\"older exponents at $0$ from both sides must coincide, 
and this contradicts the fact that $\cosh$ has a critical point at $0$, but $\exp$ does not.  
However, this can be amended by \redf{replacing $\cosh$ by a piecewise linear function on $i\R$ that maps $i\pi\Z$ to $\{\pm 1\}$ in the same way that $\cosh$ does} as described in Bishop's note\footnotemark[1].}
%Unfortunately that result is false as stated because the points $z\in i\pi\Z$ are critical points for $\cosh$, but regular points for $\exp$. Therefore the map $\nu_j:\mathbb H_r\to \C\setminus [-1,1]$ cannot be quasiregular as its boundary map would not be quasisymmetric at a point $z\in \overline{J}_1\cap \overline{J}_2\subseteq \partial\mathbb H_r$ with $J_1\in \mathcal J_1^j$ and $J_2\in \mathcal J_2^j$, where the function takes the values $\cosh$ on one side and $\exp$ on the other side. However, this can be amended by doing the interpolation away from the critical points.
\end{rmk}

Following the strategy of Bishop, each critical point $ih_n \in D_n$ \vseven{of $g_\mathbf{w}$} will produce a critical value $w_n\in \overline{\D}_{3/4}$, the position of which will be adjusted to create the oscillating wandering domain. In the following result we show that the position of $w=\rho_w(0)$ can be chosen using a map $\rho_w$ with bounded dilatation. This fact was already used in Bishop's construction; see the discussion on the map~$\rho$ in \cite[p.~28]{bishop15} or the map~$\rho_n$ in \cite[p.~485]{fagella-godillon-jarque15}.

\begin{lem} 
\label{lem:shift-interpolation} 
\label{lem:interpolation2}
There exists $K_2 >1$ such that for all $w \in \D_{3/4}$,  
there exists a \mbox{$K_2$-quasiconformal} mapping $\rho_w: \barD \to \barD$ such that 
\begin{align} 
\rho_w(z) = 
\begin{cases} 
z, \ &\text{ if } z\in\partial \D, \\
z+w,  \ &\text{ if }  z\in\overline{\D}_{1/8},
\end{cases}
\end{align}
and $\supp \rho_w\subseteq \overline{\D}\setminus \D_{1/8}$. Moreover the Beltrami coefficient $\mu_{\rho_{w}}$ depends continuously and holomorphically on $w \in \D_{3/4}$ in the $L^\infty$ sense.
\end{lem}

The proof of Lemma~\ref{lem:shift-interpolation} will be given in the Appendix~\ref{sec:a1} (see p. \pageref{proof:lem:shift-interpolation}). \vfive{One} can find a sketch of the function $\rho_w$ in Figure~\ref{fig:interpolation-rho}. Note that $\rho_{\overline{w}}(\overline{z})=\overline{\rho_w(z)}$ \vfive{for $z\in\D$}.

From now on we fix $K := \max\{K_1,K_2\}$. We define $\N_N:=\N\setminus\{1,\hdots,N-1\}$ and $\Z_N:=\Z\setminus\{-N+1,\hdots,N-1\}=\N_N\cup(-\N_N)$, where the constant $N\in\N$ will be determined in Section~\ref{sec:domains}. Combining the two previous lemmas we can define the $K$-quasiregular map $g_{\mathbf{w}}$ (see Figure~\ref{fig:setup}). 

\begin{dfn}[$K$-quasiregular map $g_{\mathbf{w}}$]
\label{dfn:gw}
For $n\in \N$, let $d_n\in\N$ and $R_n,h_n\in\R_+$ be the quantities given in Definition~\ref{dfn:hn-dn-Rn-En-Dn-Qn}. Recall that, for $n\geqslant 3$, we defined 
$$
\begin{array}{c}
E_{\pm n} := \{ z \in \C\ : \ | \textup{Re}\, z| \leqslant 2d_n \pi, \ | \textup{Im}\, z\mp h_n| \leqslant 2 d_n \pi \},\vspace{5pt}\\
D_{\pm n} := \{z \in \C\ : \ |z\mp ih_n| < R_n\}\subseteq E_{\pm n},
\end{array}
$$ 
and 
let $G_n:E\to \overline{\mathcal{E}}_{2d\pi}$ be the quasiregular mapping from Lemma \ref{lem:key-interpolation} for $d=d_n$ and $R=R_n$. Then we have $E = E_n - ih_n$ and $D = D_n - ih_n$.    
For every sequence $\mathbf{w} = (w_{N}, w_{N+1}, w_{N+2}, \dots) \in \D_{3/4}^{\N_N}$, define the function $g_{\mathbf{w}}: \C \to \C$ as follows:  
\begin{align}
g_{\mathbf{w}}(z): = 
\begin{cases} 
G_n(z \mp ih_n), \ &\text{ if } z\in E_{\pm n} \sminus D_{\pm n}   \text{ with } n \geqslant N, \\
\rho_{w_n} \circ G_n(z \mp ih_n), \ &\text{ if } z\in  D_{\pm n}  \text{ with } n \geqslant N, \\
%\rho_{\overline{w_n}} \circ G_n(z + ih_n), \ &\text{ if } z\in  D_{n}  \text{ with } n \leqslant -N, \\
2 \cosh z, \ &\text{ otherwise. }
\end{cases}
\end{align}%\margin{problem with\\ holomorphic\\ dependence?}
Then $g_{\mathbf{w}}$ is a $K$-quasiregular map such that 
$$
\supp \mu_{g_{\mathbf{w}}} \subseteq \bigcup_{n \in \Z_N} E_n \setminus \D\left(ih_n, \vseven{\left(\tfrac{1}{8}\right)^{1/(2d_n)}R_n}\right),
$$ 
and $g_{\mathbf{w}}(-z)=g_{\mathbf{w}}(z)$ for all $z\in\C$.%$g_{\mathbf{w}}(z) = g(z) = 2 \cosh z$ for all $z\in\C \setminus \bigcup_{n \in \Z_N} E_n$.
\end{dfn}

\begin{figure}[h!]
\vspace*{-25pt}
\def\svgwidth{\linewidth}
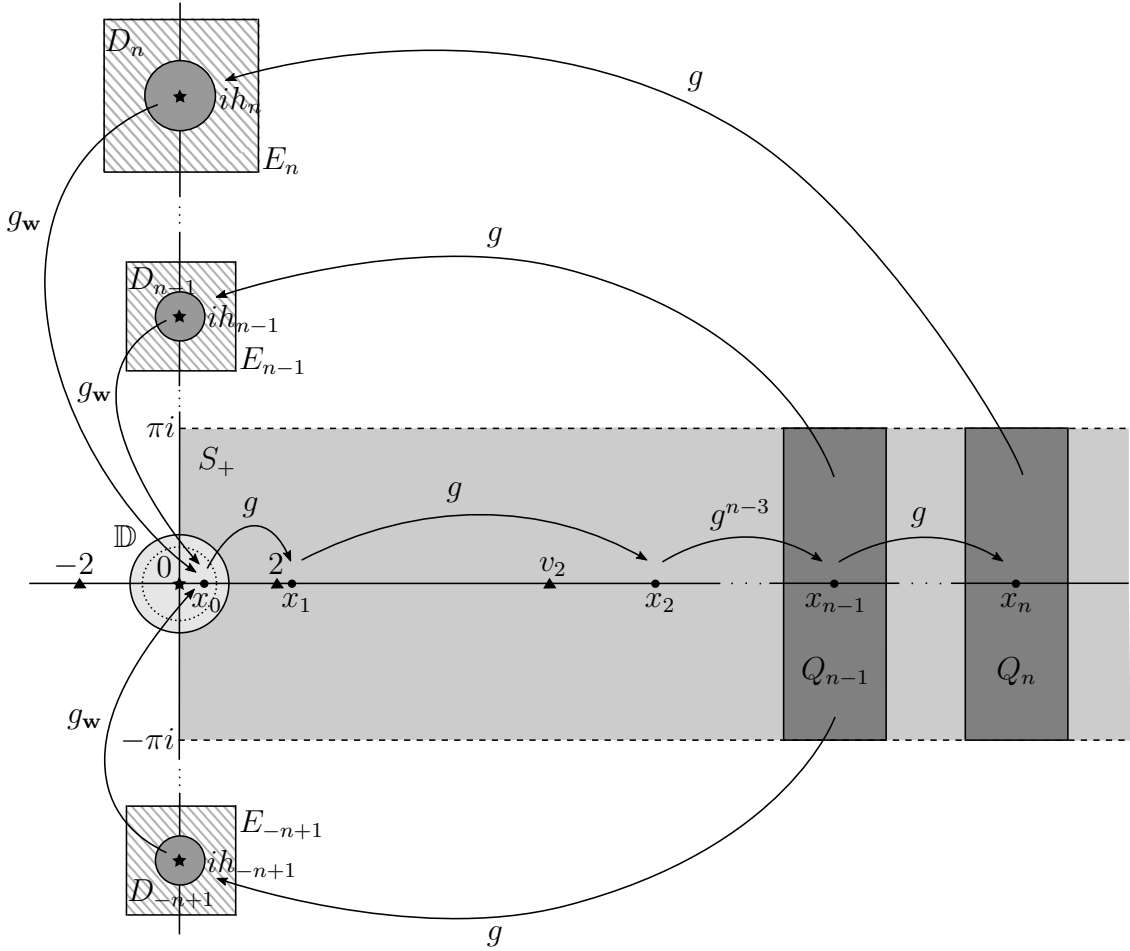
\caption{Sketch of the $K$-quasiregular map $g_{\mathbf{w}}:\C\to\C$ from Definition~\ref{dfn:gw}.}
\label{fig:setup}
\end{figure}

%The base function $g(z)=2\cosh z$ is both even and symmetric with respect to the real line. However, it is not possible to make the modified function~$g_\mathbf{w}$ have these two properties at the same time. In Definition~\ref{dfn:gw}, we chose the function~$g_\mathbf{w}$ to be even, but not symmetric with respect to $\R$. See Definition~\ref{dfn:gw-2} for an alternative modification $\tilde{g}_\mathbf{w}$ that is symmetric with respect to $\R$ but not an even function.

%Observe that for $z\in D_n$,
%$$
%g_\mathbf{w}(\overline{z})=\rho_{\overline{w_n}} ( G_n(\overline{z} + ih_n))=\rho_{\overline{w_n}} ( G_n(\overline{z - ih_n}))=\rho_{\overline{w_n}} ( \overline{G_n(z - ih_n)})=\overline{g_\mathbf{w}(z)},
%$$
%and therefore the map $g_\mathbf{w}$ is symmetric with respect to the real line.

%\begin{figure}[h]
%\vspace{10pt}
%\def\svgwidth{.9\linewidth}
%\input{setup2.pdf_tex}
%\vspace{5pt}
%\caption{Sketch of the $K$-quasiregular map $g_{\mathbf{w}}:\C\to\C$ from Definition~\ref{dfn:gw}.}
%\label{fig:setup}
%\end{figure}

%\pagebreak

\begin{dfn}[Entire function $f_{\mathbf{w}}$ and $K$-quasiconformal map $\phi_{\mathbf{w}}$]
\label{dfn:fw-phiw}
Let $g_{\mathbf{w}}$ be the $K$-quasiregular map from Definition~\ref{dfn:gw}. By the Measurable Riemann Mapping Theorem \cite[Theorem~3~\redf{in Chapter V}]{ahlfors06}, 
there exists a unique $K$-quasiconformal mapping $\phi_{\mathbf{w}} : \C \to \C$ such that 
$\phi_{\mathbf{w}}(0) =0$, $\phi_{\mathbf{w}}(1)=1$ and $\mu_{\phi_{\mathbf{w}}} = \mu_{g_{\mathbf{w}}}$.  
Then define 
\begin{equation}
f_{\mathbf{w}} := g_{\mathbf{w}} \circ \phi_{\mathbf{w}}^{-1}, 
\end{equation}
which is an entire function \vfive{in the class $\mathcal{B}$ with $S(f_\mathbf{w})=\{\pm 2\}\cup
\overline{\{w_n\}}_n$}.  
\end{dfn}

\begin{rmk}
The base function $g(z)=2\cosh z$ is both even and symmetric with respect to the real line. However, it is not possible to make the modified function~$g_\mathbf{w}$ have these two properties at the same time. In Definition~\ref{dfn:gw}, we chose the function~$g_\mathbf{w}$ to be even, but not symmetric with respect to $\R$ and therefore $\phi_\mathbf{w}$ (and, consequently, also $f_\mathbf{w}$) does not preserve the real line. See Definition~\ref{dfn:gw-2} for an alternative modification $\tilde{g}_\mathbf{w}$ that is symmetric with respect to $\R$ but not even. 
\end{rmk}

%Recall that a function $f:\C^\N\to \C$ is said to be \textit{separately holomorphic} if for each $j\in\N$ and every choice of $z_1,\hdots,z_{j-1},z_{j+1},\hdots$, the function
%$$
%z\mapsto f(z_1,\hdots,z_{j-1},z,z_{j+1},\hdots)
%$$ 
%is holomorphic. Hartog's Theorem \cite[Theorem~1.2.5]{krantz92} tells us that a function that is separately holomorphic on $\C^k$, for $k\in\N$, is holomorphic.
%
%\begin{lem} \label{lem:holomorphic-dependence} 
%The $K$-quasiconformal map $\phi_{\mathbf{w}}$ from Definition~\ref{dfn:fw-phiw} is a separately holomorphic function of the sequence $\mathbf{w}\in \D_{3/4}^{\N_N}$, that is, for every $z\in\C$ and $k\in\N$, the function 
%$$
%(w_1,w_2,\hdots,w_k)\mapsto \phi_{\mathbf{w}}(z)
%$$
%with $\mathbf{w}=(w_1,w_2,\hdots,w_{k-1},w_{k},w_{k+1},\hdots)$ and $w_j\in \D_{3/4}$ fixed for all $j>k$, is a holomorphic map from $\C^k$ to $\C$.
%\end{lem}
%\begin{proof} 
%By Lemma~\ref{lem:interpolation2}, for every $z\in\C$ and $n\in\N$, if we fix all $w_j$ for $j\neq n$, the function $w_n\mapsto \mu_{g_{\mathbf{w}}}(z)$ with $\mathbf{w}=(w_1,w_2,\hdots,w_n,\hdots)$ is holomorphic. In fact, if we put $\mathbf{w}_0=(\tfrac{1}{2},\tfrac{1}{2},\hdots)$, then for every $n\in \N$,
%$$
%\mu_{g_{\mathbf{w}}}=\mu_{g_{\mathbf{w}_0}}+(w_n-\tfrac{1}{2})\nu_n+o(w_n-\tfrac{1}{2})
%$$
%for some function $\nu_n\in L^\infty(\C)$. Thus, it follows from Hartogs Theorem \cite[Theorem~1.2.5]{krantz92} that $\mu_{g_{\mathbf{w}}}$ depends holomorphically on $\mathbf{w}\in\D_{3/4}^{\N_N}$ if we fix all except finitely many of the components of $\mathbf{w}$.
%\end{proof}

\section{Estimates on quasiconformal maps}

\label{sec:qc-estimate}

%In order to proof Theorem \ref{thm:main}, we will first construct a quasiregular map $g$ such that $g(z)=2\cosh z$ on $\C\setminus E$, where $E$ is a disjoint union of squares $E_n$ centered along the imaginary axis, and $g(z)=(z-z_0)^{d_n}$ on discs $D_n\subseteq E_n$ for all $n\in \Z^*$. We will do a quasiconformal interpolation on the sets $E_n\setminus D_n$ that is independent of the powers $d_n$ that we choose on each disc $D_n$. Then using the Riemann mapping theorem we will obtain an entire function $f\in\mathcal{B}$ such that $f=g\circ  \phi^{-1}$. We know the function $g$ explicitly, but we need to show that applying the quasiconformal map $\phi^{-1}$ preserves the dynamics that we prescribed. 

In the previous section we constructed a $K$-quasiregular map $g_{\mathbf{w}}$ and then, using the measurable Riemann mapping theorem, we obtained an entire function $f_{\mathbf{w}}\in\mathcal{B}$ such that $f_{\mathbf{w}}=g_{\mathbf{w}}\circ  \phi_{\mathbf{w}}^{-1}$. Despite we know the function $g_{\mathbf{w}}$ explicitly, in order to be able to prescribe the dynamics of $f_{\mathbf{w}}$, we need to estimate the quasiconformal map $\phi_{\mathbf{w}}$, which is close to the identity. 

Define the cylinder $\calC=\C/2\pi i \Z$ and its distance $|w|_{\calC}=\inf \{ |w+2\pi i n|\ : \ n \in \Z\}$. Important estimates on the function $\phi_{\mathbf{w}}$ will be derived from the following result on quasiconformal maps from \cite{shishikura}.

\begin{thm}[{Key Inequality \cite[Lemma~9]{shishikura}}] 
\label{thm:KeyInequality} 
Given $K >1$, there exist $0<\delta_1<1$ and $C>0$ such that  
if $\redf{f}:\C \to \C$ is a $K$-quasiconformal map with $\redf{f}(0)=0$ and $0<|z_2| \le \delta_1|z_1|$, 
then 
$$
\begin{aligned} \label{eq:KeyIneq}
\left| \log \frac{\redf{f}(z_1)}{z_1}\! -\! \log \frac{\redf{f}(z_2)}{z_2} \right|_{\calC}\! \!
&\leqslant 
2C\!\left( \left| \iint_\C \frac{\mu_\redf{f}(z)\vphi_{z_1,z_2 \vphantom{|}}(z)}{1-|\mu_\redf{f}(z)|^2}dx \tinyhsp dy\right|\!
+\!\! \iint_\C \frac{|\mu_\redf{f}(z)|^2 |\vphi_{z_1,z_2 \vphantom{|}}(z)|}{1-|\mu_\redf{f}(z)|^2}dx \tinyhsp dy\right),  
\end{aligned}
$$
where $\vphi_{z_1,z_2 \vphantom{|}}(z) = \frac {z_1}{z(z-z_1)(z-z_2)}$. 
\end{thm} 

The constants $\delta_1$, $C$ depend on $K$ and are related to the hyperbolic metric on the twice punctured plane $\C\setminus\{0,1\}$ (see \cite[Lemma~5]{shishikura}). In the next corollary we adapt the Key Inequality to the form in which we will use it in this paper.

\begin{cor} \label{cor:KeyInequality}
Let the constants $K\!>\!1$, $0\!<\!\delta_1\!<\!1$ and $C\!>\!0$ be as in Theorem~\ref{thm:KeyInequality}. If $\phi:\C \to \C$ is a $K$-quasiconformal map and $\alpha, \beta, \gamma\in \C$ are distinct points with
\begin{equation}
\label{eq:key-ineq-condition-abc}
0<|\gamma-\alpha| \le \delta_1|\beta-\alpha|,
\end{equation}
then
\begin{equation}
\begin{aligned} \label{eq:KeyIneq-cor}
\left| \log \frac{\phi(\beta)-\phi(\alpha)}{\beta-\alpha} - \log \frac{\phi(\gamma)-\phi(\alpha)}{\gamma-\alpha} \right|_{\calC} 
&\le  C(K -1) \iint_{\supp \mu_\phi} \frac{|\beta- \alpha| \tinyhsp dx\tinyhsp dy}{| (z-\alpha)(z-\beta)(z-\gamma)|}.
\end{aligned}
\end{equation}
\end{cor}
\begin{proof} 
Applying Theorem \ref{thm:KeyInequality} to 
$f(z)=\phi(z+\alpha)-\phi(\alpha)$ with $z_1=\beta-\alpha$, $z_2=\gamma-\alpha$:
$$
\begin{array}{c}
\left| \log \dfrac{\phi(\beta)-\phi(\alpha)}{\beta-\alpha} - \log \dfrac{\phi(\gamma)-\phi(\alpha)}{\gamma-\alpha} \right|_{\calC}  \hspace*{\fill}\vspace{5pt}\\
\hspace*{\fill}\displaystyle\leqslant 2C\!\left( \left| \iint_\C \frac{\mu_\phi(z)\vphi_{z_1,z_2 \vphantom{|}}(z-\alpha)}{1-|\mu_\phi(z)|^2}dx \tinyhsp dy\right|\!
+\!\! \iint_\C \frac{|\mu_\phi(z)|^2 |\vphi_{z_1,z_2 \vphantom{|}}(z-\alpha)|}{1-|\mu_\phi(z)|^2}dx \tinyhsp dy\right)\vspace{5pt}\\
\displaystyle\leqslant 2C \iint_{\supp \mu_\phi(z)} \dfrac{|\mu_\phi(z)|+|\mu_\phi(z)|^2}{1-|\mu_\phi(z)|^2}\frac{|\beta- \alpha|\tinyhsp dx\tinyhsp dy}{| (z-\alpha)(z-\beta)(z-\gamma)|}.\hspace*{\fill}
\end{array}
$$ 
Then, the result follows from the fact that 
$$
2 \dfrac{|\mu_\phi(z)|+|\mu_\phi(z)|^2}{1-|\mu_\phi(z)|^2}=2 \dfrac{|\mu_\phi(z)|}{1-|\mu_\phi(z)|}=K_\phi(z)-1\leqslant K-1
$$
because $\phi$ is a $K$-quasiconformal map.
\end{proof}

In the rest of this section, we make the following standing assumption. 

\begin{assump}
\label{ass:qu-estimates}
Let $K>1$ be a fixed constant and, for $m\in\N$, consider the disc $B_m := \D(\zeta_m, r_m)$ with $\zeta_m \in\C$ and $r_m>0$. Assume that 
\begin{enumerate}
\item[(i)] $\sum_{m=1}^\infty r_m/|\zeta_m| < +\infty$;\vspace{5pt}
\item[(ii)] $|\zeta_m| \geqslant 4$ and $r_m/|\zeta_m| \leqslant \min\{\frac{1}{4},\delta_1\}$ for all $m\in\N$, where $0<\delta_1<1$ is the constant from Theorem~\ref{thm:KeyInequality};\vspace{5pt}
\item[(iii)] there is a $K$-quasiconformal map $\phi:\C \to \C$ such that $\phi(0)=0$, $\phi(1)=1$ and $\supp \mu_\phi \subseteq  \bigcup_{m =1}^\infty B_m$.  
\end{enumerate}
\end{assump}

%\blueb{Observe that here we state the results with large generality, but later we will use them by putting $\zeta_{2k-1}=-\zeta_{2k}=ih_{N-1+k}$ and $r_{2k-1}=r_{2k}=3R_{N-1+k}$ for all $k\in\N$, where for $n\geqslant 3$, $h_n,d_n$ are as in Definition~\ref{dfn:hn-dn-Rn-En-Dn-Qn} and $N\geqslant 3$ is sufficiently large, so $E_{N-1+k}\subseteq B_{2k-1}$ and $E_{-N+1-k}=-E_{N-1+k}\subseteq B_{2k}$ for all $k\in\N$, and $\phi$ will be the quasiconformal map $\phi_{\mathbf{w}}$ from Definition~\ref{dfn:fw-phiw} (see Lemma~\ref{lem:assumption}).}

The following three lemmas are consequences of Corollary \ref{cor:KeyInequality}, and their proofs 
will be given in the Appendix~\ref{sec:proof-of-estimates}. \blueb{We state the results with large generality (see Figures~\ref{fig:qc-estimate2} and \ref{fig:qc-estimate3}). Later we will apply them by reindexing the discs $\{3D_n\}_{n\in\Z_N}$ to $\{B_m\}_{m\in\N}$ and $\phi$ will be the quasiconformal map $\phi_{\mathbf{w}}$ from Definition~\ref{dfn:fw-phiw} (see Lemma~\ref{lem:assumption}). Note that $3D_{\pm n}=\D(\pm ih_n,3R_n)\supseteq E_{\pm n}$ for $n\geqslant N$.} 

\begin{lem}
\label{lem:qc-estimate1} 
Suppose that Assumption~\ref{ass:qu-estimates} holds. For every $\eps>0$, 
there exists $M_1=M_1(\eps)\in\N$ such that if $\supp \mu_\phi \subseteq \bigcup_{m = M_1}^{\infty} B_m$, then 
\begin{equation}
\label{eq:qu-estimate1}
\left| \log \frac{\phi(\zeta)}{\zeta} \right|_{\calC} < \eps,\quad \text{ for } \zeta \in \C \sminus \{0\},
\end{equation}
and, in particular,
\begin{equation}
\label{eq:qc-estimate1-2}
e^{-\varepsilon}|\zeta|<|\phi(\zeta)|<e^{\varepsilon}|\zeta|\quad \text{ and } \quad |\textup{arg}\,\phi(\zeta)-\textup{arg}\,\zeta\!\!\!\pmod{2\pi}|<\varepsilon,
\end{equation}
for all $\zeta\in\C\setminus\{0\}$.
\end{lem}

\redf{Observe that it is also possible to show that under the hypothesis of Lemma~\ref{lem:qc-estimate1}, the $K$-quasiconformal map $\phi$ is conformal at $\infty$, that is, the limit $\lim_{\zeta\to\infty}\phi(\zeta)/\zeta$ exists.}

\begin{lem} \label{lem:qc-estimate2} 
Suppose that Assumption~\ref{ass:qu-estimates} holds and suppose also that there exists $C_1>0$ such that 
if $z \in B_m$ and $z' \in B_{m'}$ with $m \neq m'$, 
then $|z-z'| \ge C_1 \sqrt{|z z'|}$.
For any $0<\kappa \leqslant 1$, there exists $C_2\,\blue{=C_2(\kappa)}>1$ such that for any $m\in\N$,
if   $|\zeta - \zeta_m| = \kappa r_m$, then
\begin{equation}
\label{eq:qc-estimate2}
\frac{1}{C_2} \kappa r_m \le |\phi(\zeta) - \phi(\zeta_m)| \le C_2 \kappa r_m.
\end{equation}
\end{lem}

\begin{figure}[h!]
%\vspace{10pt}
\def\svgwidth{.8\linewidth}
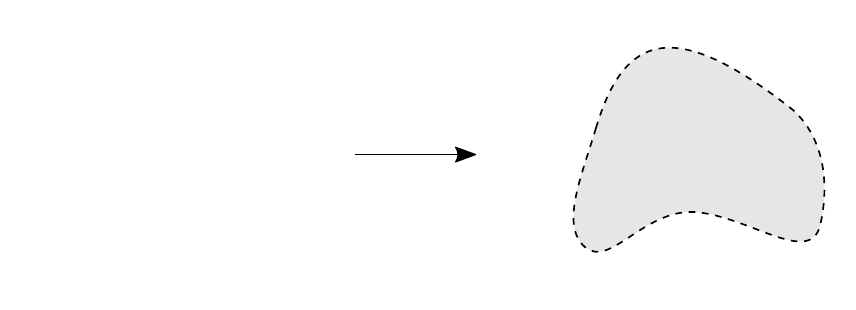
\caption{Sketch of Lemma~\ref{lem:qc-estimate2}.}
\label{fig:qc-estimate2}
\end{figure}

\begin{lem} 
\label{lem:qc-estimate3} 
Suppose that Assumption~\ref{ass:qu-estimates} holds. For every $0<\theta<\redf{\pi}$, there exists $C_3>1$ such that if $\zeta\in\C$ satisfies that 
$$
%B_m \subseteq \{z \in \C\ : \ \arg \zeta + \theta < \arg z < \arg \zeta+2\pi - \theta\} \quad \text{ for all } m\in\N,
\vfive{B_m\cap \{z \in \C\, : \, |\textup{arg}\, z - \textup{arg}\, \zeta|<\theta\}=\emptyset, \quad \text{ for all } m\in\N,}
$$
then
\begin{equation}
\label{eq:qc-estimate3}
\frac{1}{C_3} \le |\phi'(\zeta)| \le C_3.
\end{equation}
\blue{For every $\eta>0$, there exists $M_2=M_2(\eta)\in\N$ such that if $\supp\mu_\phi\subseteq \sum_{m=M_2}^\infty B_m$, then $1<C_3<1+\eta$.}
\end{lem}

\begin{figure}[h!]
%\vspace{10pt}
\def\svgwidth{.5\linewidth}
\input{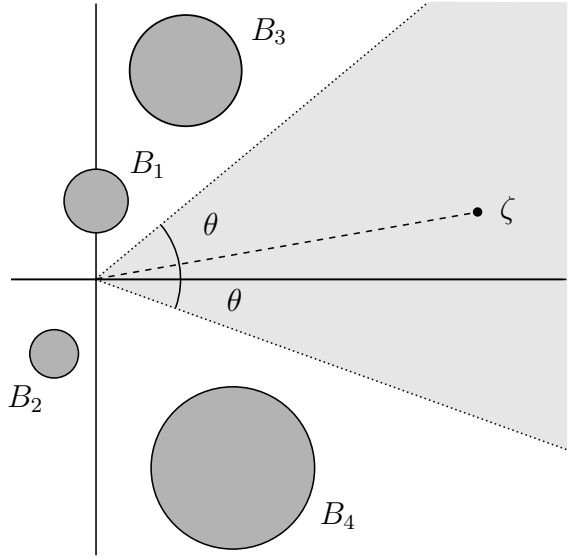}
\caption{Sketch of Lemma~\ref{lem:qc-estimate3}.}
\label{fig:qc-estimate3}
\end{figure}

%We also use the following lemma.  
%\begin{lem} 
%\label{lem:Koebe-estimate} \margin{remove\\ completely?}
%For any $R>1$ and $\eps>0$, there exists $R'>R$ such that if $\phi$ is conformal in $\D_{R'}$ 
%and $\phi(0)=0$, $\phi(1)=1$,  then 
%\begin{equation}
%|\phi(z)-z| < \eps \quad \text{ and } \quad \bigl| \log |\phi'(z)| \bigr| < \eps\quad \text{for all } z\in\D_R.
%\end{equation}  
%\end{lem}
%\begin{proof} 
%Consider the conformal map $\tilde{\phi}:\D \to \C$ given by
%$$
%\tilde{\phi}(z):=\frac{1}{R'|\phi'(0)|}\phi(R'z)
%$$
%that satisfies that $\tilde{\phi}(0)=0$ and $\tilde{\phi}'(0)=1$. It follows from the Koebe distortion theorem~\cite[Theorem~1.6]{pommerenke75} that for $z\in\D_{R'}$,
%$$
%|\phi(z)-z|=\left|\tilde{\phi}\left(\frac{z}{R'}\right)-z\right|
%$$
%\red{This follows from Koebe distortion theorem.}
%\end{proof}

\section{Definition of the domains $\{U_n\}_n$ and the centers $\{c_n\}_n$}
\label{sec:domains}

Let $g_{\mathbf{w}}$, $\phi_{\mathbf{w}}$, and $f_{\mathbf{w}} = g_{\mathbf{w}} \circ \phi_{\mathbf{w}}^{-1}$ be as in Section \ref{sec:interpolation}, 
where $N\in\N$ is to be determined in this section. For simplicity, we may omit $\mathbf{w}$ in the notation, that is, we write $g,\phi,f$ instead of $g_{\mathbf{w}},\phi_{\mathbf{w}},f_{\mathbf{w}}$, respectively. 

We are going to introduce two collections of domains $U_{n,j}$ and $\widehat{U}_{n,j}=\phi^{-1}(U_{n,j})$, for $n \geqslant N$ and $0<j\leqslant n$, as in the following diagram. We shall think about the quasiconformal map $\phi$ as a different coordinate. The goal of this section is to define the sets $U_n=U_{n,0}$, for $n\in \N$.

\newcommand{\upin}{\mathrel{\rotatebox[origin=c]{90}{$\in$}}}
\newcommand{\upsubseteq}{\mathrel{\rotatebox[origin=c]{90}{$\subseteq$}}}
\newcommand{\upeq}{\mathrel{\rotatebox[origin=c]{90}{$=$}}}

%: the diagram
%%%%%%%%%%%%%%%%%%%%
\begin{center}
\label{diag:1}
\hspace*{-43pt}
\begin{tikzcd}[row sep=1.2em, column sep=1.4em]
\D(\tfrac{1}{2}, \tfrac{1}{8}) &[3em] Q_M \lar[swap]{(\phi \,\circ\, g^{-1})^{M} \circ\, \phi} \rar{g} 
& g(Q_M) 
& Q_{M+1} \lar[swap]{\phi} \rar{g} & \cdots \rar{g} %long version
%& \cdots \lar[swap]{\phi} \rar{g} %short version
& g(Q_{j-1}) & Q_j \lar[swap]{\phi} \rar{g} & g(Q_j) & \cdots \lar[swap]{\phi} 
\\
U_n \rar{\phi^{-1} \circ \,f^M} \arrow[u, phantom, "\upsubseteq" description]  
& \widehat{U}_{n,M}  \rar{g} \arrow[u, phantom, "\upsubseteq" description]  
& U_{n,M+1} \arrow[u, phantom, "\upsubseteq" description] 
& \widehat{U}_{n,M+1} \lar[swap]{\phi} \rar{g} \arrow[u, phantom, "\upsubseteq" description]  & \cdots  \rar{g} %long version
%& \cdots \lar[swap]{\phi} \rar{g} %short version
& U_{n,j} \arrow[u, phantom, "\upsubseteq" description]  
& \widehat{U}_{n,j} \lar[swap]{\phi} \rar{g} \arrow[u, phantom, "\upsubseteq" description]  
& U_{n,j+1} \arrow[u, phantom, "\upsubseteq" description]  
& \cdots \lar[swap]{\phi} 
\\
c_n \arrow[r, mapsto] \arrow[u, phantom, "\upin" description]  
& \hat{c}_{n,M} \ar[r,mapsto] \arrow[u, phantom, "\upin" description]  
& c_{n,M+1} \arrow[u, phantom, "\upin" description] \ar[r,mapsto]
& \hat{c}_{n,M+1}  \ar[r,mapsto] \arrow[u, phantom, "\upin" description]  & \cdots  \ar[r,mapsto]%long version
%& \cdots \lar[swap]{\phi} \rar{g} %short version
& c_{n,j} \arrow[u, phantom, "\upin" description]  \ar[r,mapsto]
& \hat{c}_{n,j}  \ar[r,mapsto] \arrow[u, phantom, "\upin" description]  
& c_{n,j+1} \arrow[u, phantom, "\upin" description]  \ar[r,mapsto]
& \cdots  
\end{tikzcd}
%\]
\vspace{3pt}
\vskip 1em
%\begin{equation}
\hspace*{-27pt}
\begin{tikzcd}[row sep=1.2em, column sep=1.4em]
  \cdots & Q_{n-1} \lar[swap]{\phi} \rar{g} 
& g(Q_{n-1}) & Q_n \lar[swap]{\phi} \rar{g} & g(Q_n) \arrow[r, phantom, "\supset" description]  
&[-2em] \phi(\tfrac{1}{2}D_n) & \tfrac{1}{2}D_n  \lar[swap]{\phi} \rar{g} 
&\D(w_n, (\tfrac{1}{2})^{2d_n})
\\
 \cdots & \widehat{U}_{n,n-1} \lar[swap]{\phi} \rar{g}  \arrow[u, phantom, "\upsubseteq" description] 
& U_{n,n} \arrow[u, phantom, "\upsubseteq" description]  
& \widehat{U}_{n,n} \lar[swap]{\phi} \arrow[rr,"g"] \arrow[u, phantom, "\upsubseteq" description]  
& & \D(\phi(ih_n), R_n') \arrow[u, phantom, "\upsubseteq" description] \arrow[rr, "f"] & & \D(w_n, (\tfrac{1}{2})^{2d_n}) \arrow[u, phantom, "\upeq" description]  
\\
 \cdots\ar[r,mapsto] & \hat{c}_{n,n-1}  \ar[r,mapsto]  \arrow[u, phantom, "\upin" description] 
& c_{n,n} \arrow[u, phantom, "\upin" description]  \ar[r,mapsto]
& \hat{c}_{n,n} \arrow[rr,mapsto] \arrow[u, phantom, "\upin" description]  
& &\phi(ih_n) \arrow[u, phantom, "\upin" description] \arrow[rr, mapsto] & & w_n \arrow[u, phantom, "\upin" description]  
\end{tikzcd}
%\end{equation}
\vspace{10pt}\\
\textsc{Diagram 1.} The domains $U_{n,j}$ and $\widehat{U}_{n,j}$, for $n \geqslant N$ and $0\leqslant j\leqslant n$. The right hand side of the top part continues in the left hand side of the bottom part. 
\vspace{10pt}
\end{center}

In this section we use the results from Section~\ref{sec:qc-estimate} with a particular choice of discs $\{B_m\}_m$ and quasiconformal map $\phi$. So we need to check that Assumption~\ref{ass:qu-estimates} holds. The next lemma follows from Definition~\ref{dfn:fw-phiw} and the property \eqref{eq:growth-estimate-finitesum} that we gave in Lemma~\ref{lem:growth-estimates}.

\begin{lem}
\label{lem:assumption}
\blue{There is $N\geqslant 3$ sufficiently large so that if we define $\{B_m\}_{m\in\N}$ by reindexing the discs $\{3D_n\}_{n\in\Z_N}$ and let $\phi=\phi_{\mathbf{w}}$ be the $K$-quasiconformal map from Definition~\ref{dfn:fw-phiw}, then the Assumption~\ref{ass:qu-estimates} holds.}
\end{lem}
%\begin{proof}
%Since $h_n\geqslant x_{n+1}-\pi\geqslant x_2-\pi >6.5$ for all $n\in\N$, for any $N_0\geqslant 0$, we have $|\zeta_m|=h_{N_0+m}>4$ for all $m\in\N$. On the other hand, $d_n\leqslant x_{n+1}/x_n$, so
%$$
%\frac{r_m}{|\zeta_m|}=\frac{3\pi d_{N_0+m}}{h_{N_0+m}}\leqslant \frac{x_{N_0+m+1}}{(x_{N_0+m+1}-\pi)x_{N_0+m}}<\min\{\tfrac{1}{4},\delta_1\}
%$$
%provided that $N_0\in\N$ is sufficiently large. Part (ii) follows from Lemma~\ref{lem:growth-estimates}, and part (iii) is satisfied by definition.
%\end{proof}

\blue{Recall that in Section~\ref{sec:setup} we considered $M\geqslant 3$, from this point we fix $M=3$.} The next two lemmas bound the distortion of the map $\phi$ on some sets by supposing that the constant $N\in\N$ in Definition~\ref{dfn:gw} is sufficiently large. The first lemma concerns the initial iterates~$f^n$ of $f=g\circ \phi^{-1}$, namely for $n\leqslant M$, while the second lemma deals with the case $n>M$.

\begin{lem}
\label{lem:initial-iterates} 
If \blue{$N\geqslant 3$} in Definition~\ref{dfn:gw} is large enough, 
then %$\phi_{\mathbf{w}}$ is close to the identity up to $Q_M$ and 
$(\phi \circ g^{-1})^M \circ \phi$ maps the rectangle $Q_M$ into the disc $\D(\tfrac{1}{2}, \tfrac{1}{8})$.  
\end{lem}
\begin{proof} 
When $\phi=id$, $((\phi \circ g^{-1})^M \circ \phi)(Q_M) = g^{-M}(Q_M) \subset\D(\tfrac{1}{2}, \tfrac{1}{16})$ by Lemma \ref{lem:initial-segment}. %The map $\phi$ is conformal in $\C\setminus \bigcup_{|n|\geqslant N}E_n$, so if $N\in\N$ is large, then $\phi$ is conformal in a large domain and, by Lemma~\ref{lem:Koebe-estimate}, 
By Lemma~\ref{lem:qc-estimate1}, we can choose $N\in\N$ large enough so that the map $\phi$ is sufficiently close to the identity on the finite collection of sets $(g^{-1} \circ \phi)^j (Q_M)$ for $0\leqslant j\leqslant M$.  
So the assertion follows.  
\end{proof}

\begin{lem} 
\label{lem:phi(Q)-in-g(Q)} 
If \blue{$N\ge 3$} in Definition~\ref{dfn:gw} is large enough, 
then $\phi(Q_{j+1}) \subset g(Q_j)$ for $j\ge M$, and $\phi(D_{\pm n}) \subset g(Q_n)$ for $n\geqslant N$.  
\end{lem}
\begin{proof} 
Since $Q_{j+1}\subseteq\{z\in\C\, :\, x_{j+1}-1\leqslant |z|\leqslant x_{j+1}+1+\pi\}$, it follows from Lemma~\ref{lem:qc-estimate1} that, for every $\varepsilon>0$,  there exists $M_1=M_1(\eps)\in\N$ such that if $\supp\mu_\phi\subseteq \bigcup_{m\geqslant M_1} B_m$, then 
$$
\phi(Q_{j+1})\subseteq \{z\in\C\, :\, e^{-\eps}(x_{j+1}-1)\leqslant |z|\leqslant e^\eps (x_{j+1}+1+\pi) \},
$$
for all $j\geqslant M$. On the other hand, since $x_j\geqslant x_M>x_*=\tfrac{5}{3}$ for $j\geqslant M$ by Lemma~\ref{lem:escaping-real-orbits}, Lemma \ref{lem:cosh} (iv) implies that
$$
g(Q_j)\supseteq \{z\in\C\, :\, \tfrac{1}{2}e^{x_j}<|z|<2e^{x_j}\}\setminus \R_-
$$
for $j\geqslant M$. Setting $\eps:=\log\tfrac{3}{2}<\tfrac{1}{2}$ we obtain that
$$
e^\eps< \frac{2(e^{x_j}+e^{-x_j}-1)}{e^{x_j}} \quad\text{ and }\quad e^\eps <\frac{2e^{x_j}}{e^{x_{j}}+e^{-x_j}+1+\pi}
$$
and, \blue{for $N$ sufficiently large}, we have \blue{$\phi(Q_{j+1}) \subset g(Q_j) \cup \R_-$} for \blue{$j\geqslant M$}. \blue{To show that $\phi(Q_{j+1})\cap \R_-=\emptyset$, observe that \mbox{$Q_{j+1}\subseteq \{z\in\C\,:\, |\text{arg}\,z|<\tfrac{\pi}{6}\}$}, and hence by Lemma~\ref{lem:qc-estimate1}, \red{$\phi(Q_{j+1})\subseteq  \{z\in\C\,:\, |\text{arg}\,z|<\tfrac{\pi}{2}\}$} for all $j\geqslant M$.}

Similarly, we have $D_{\pm n}=\D(\pm ih_n,R_n)\subseteq \{z\in\C\, :\, h_n-R_n\leqslant |z|\leqslant h_n+R_n \}$ for $n\geqslant N$, and, by Lemma~\ref{lem:qc-estimate1},
$$
\phi(D_{\pm n})\subseteq \left\{z\in\C\, :\, e^{-\eps}\!\left(x_{n+1}\!\left(1-\frac{\pi}{x_n}\right)\!-\frac{2\pi}{3}\right)\!\leqslant\! |z|\!\leqslant\! e^\eps\! \left(x_{n+1}\!\left(1+\frac{\pi}{x_n}\right)\!+\frac{2\pi}{3}\right)\! \right\},
$$
for $n\geqslant N$. It is easy to check that this implies that $\phi(D_{\pm n})\subseteq g(Q_n)\cup \R_-$ with $\eps=\log\tfrac{3}{2}$. To show that $\phi(D_{\pm n})\cap \R_-=\emptyset$, observe that 
$$
D_{\pm n}\subseteq \left\{z\in\C\ :\ |\textup{arg}\,z\mp\tfrac{\pi}{2}|<\arcsin\frac{R_n}{h_n}\right\},
$$
and by Lemma~\ref{lem:growth-estimates}, \red{$R_n/h_n<1/(n!)^2\ge\tfrac{\pi}{18}$} for $n\geqslant 3$ and so \red{$\arcsin(R_n/h_n)<\tfrac{\pi}{12}$} for $n\geqslant 3$. Then, by Lemma~\ref{lem:qc-estimate1}, for \blue{$N$ large enough}, we have
$$
\phi(D_{\pm n})\subseteq \{z\in\C\, :\, |\textup{arg}\,z\mp\tfrac{\pi}{2}|<\red{\tfrac{\pi}{12}}+\varepsilon<\tfrac{\pi}{4}\}\subseteq \C\setminus \R_-,
$$
and therefore $\phi(D_{\pm n})\subseteq g(Q_n)$ for all $n\geqslant N$ as we wanted to prove.
\end{proof}

We now fix \blue{$N\geqslant 3$} sufficiently large so that Lemmas \ref{lem:initial-iterates} and \ref{lem:phi(Q)-in-g(Q)} are satisfied. In particular, from here onwards the set $\text{supp}\, \mu_{g_{\mathbf w}}$ is also fixed and the map $g_{\mathbf w}$ (and also $\phi_{\mathbf w},f_{\mathbf w}$) depends only on the choice of the sequence $\mathbf{w}\in\D^{\N_N}_{3/4} $, where $\N_N:=\N\setminus\{1,\hdots,N-1\}$.

Recall that if $\lambda\in\R_+$ and $D=\D(z_0,R)$ for some $z_0\in\C$ and $R>0$, then we use the notation $\lambda D:=\D(z_0,\lambda R)$.

\begin{lem}
\label{lem:inside-half-disk} 
There exists $C_4>0$ such that if we set $R_n' := C_4 R_n$ for all $n\geqslant N$, then $\D(\phi(\pm ih_n), R_n') \subseteq \phi(\frac{1}{2}D_{\pm n})$ for all $n\ge N$.   
\end{lem}
\begin{proof} 
Define $C_1:=\tfrac{1}{2}$, \blue{we want to show that if $z_1 \in B_{m_1}$ and $z_2 \in B_{m_2}$ with $1\leqslant m_1 < m_2$, then $|z_2-z_1|\geqslant C_1 \sqrt{|z_1 z_2|}$. If $m_1$ and $m_2$ have different parity, this is clear. Therefore, we may assume without loss of generality that $m_1,m_2$ are both odd.} Lemma~\ref{lem:growth-estimates} implies that \red{if $N$ is large enough}
$$
|z_2-z_1|\geqslant (h_{n_2}-3R_{n_2})-(h_{n_2-1}+3R_{n_2-1})\geqslant  C_1 (h_{n_2}+3R_{n_2})\geqslant C_1 \max\{|z_1|, |z_2|\}
$$
\blue{as required, where $n_2\in\Z_N$ is the value that corresponded to $m_2\in \N$ before reindexing.} Thus, by Lemma \ref{lem:qc-estimate2}, for $\kappa:=\tfrac{1}{6}$, there exists \blue{$C_2=C_2(\kappa)>0$} such that
$$
\D(\phi(\zeta_m),\kappa r_m/C_2)\subseteq \phi(\kappa B_m),\quad \text{ for all } m\in\N.
$$
Since \blue{$\kappa r_m=3\kappa R_n=\tfrac{1}{2}R_n$ for all $m\in\N$}, putting $C_4:=3\kappa/C_2$ we obtain that
$\D(\phi(\pm ih_n), R_n') \subseteq \phi(\frac{1}{2}D_{\pm n})$ for all $n\ge N$.
\end{proof}

We are now ready to define the domains in the middle line of the Diagram~1.  
\begin{dfn}[Domains $\{U_n\}_n$] 
\label{dfn:domains-Un}
For every $n\geqslant N$, define the set 
$$
\widehat{U}_{n,n} := g^{-1}(\D(\phi(ih_n), R_n'))\subseteq Q_n,
$$
which is well defined by Lemmas \ref{lem:cosh}, \ref{lem:phi(Q)-in-g(Q)} and \ref{lem:inside-half-disk}.  
Then, for $n\geqslant j\geqslant M+1$, we can define recursively 
$$
U_{n,j} := \phi(\widehat{U}_{n,j}) \quad \text{ and }  \quad \widehat{U}_{n,j-1} := g^{-1}(U_{n,j})\subseteq Q_{j-1},  
$$
where the inclusion is guaranteed by Lemma~\ref{lem:phi(Q)-in-g(Q)}. Finally, let 
$$
U_n := (\phi \circ g^{-1})^M \circ \phi(\widehat{U}_{n,M})\subseteq \D\left(\tfrac{1}{2},\tfrac{1}{8}\right),\quad \text{ for } n\ge N,
$$
by Lemma~\ref{lem:initial-iterates}.
\end{dfn}

Keeping a similar notation as in \cite{bishop15}, the sets $\{U_n\}_n$ will be contained in the grand orbit of the oscillating wandering domain (see Figure~\ref{fig:setup2}). \red{Note that the set $f^{n+1}(U_n)$ in Bishop's notation corresponds to the set $f^{n+2}(U_n)$ in our notation.}

\begin{figure}[h]
\vspace{-104pt}
\def\svgwidth{\linewidth}
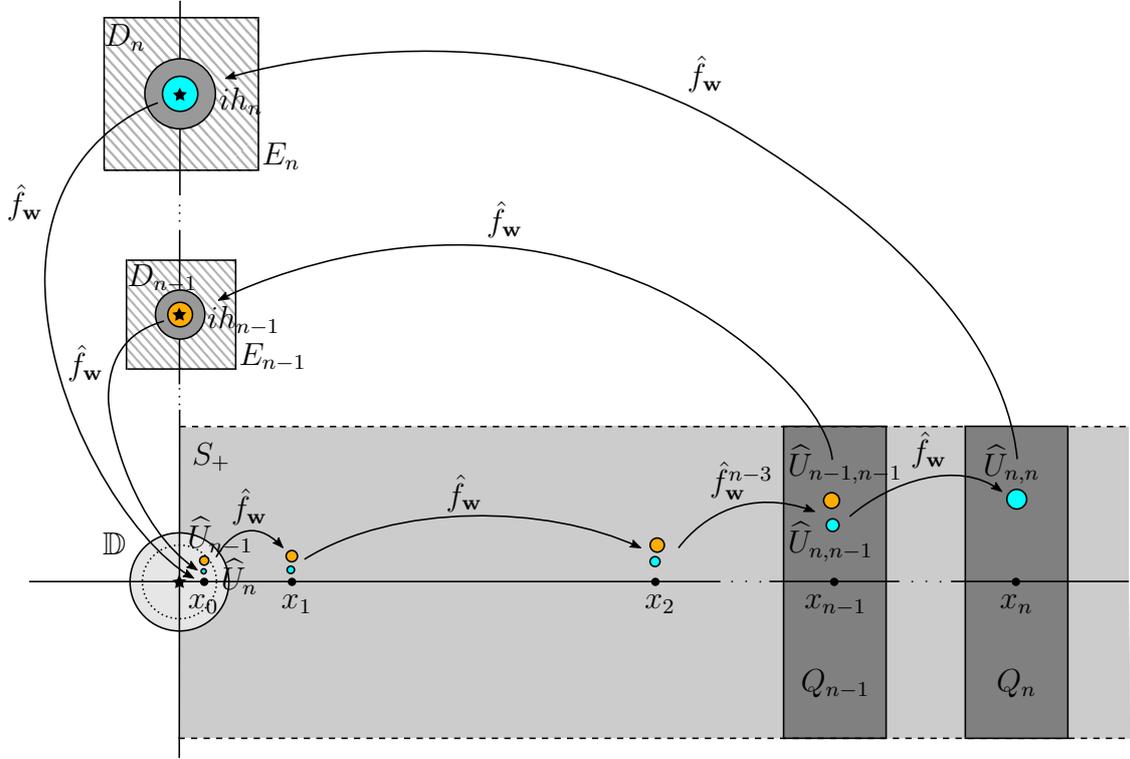
%\vspace{5pt}
\caption{The iterates of the domains $\widehat{U}_n$ from Definition~\ref{dfn:domains-Un} by the function $\hat{f}_\mathbf{w}=\phi_\mathbf{w}^{-1}\circ g_\mathbf{w}$ which is conjugated to $f_\mathbf{w}$ by the map~$\phi_\mathbf{w}$.}
\label{fig:setup2}
\end{figure}

We now need to show that every domain $U_n$ contains a disc $\D(c_n,\rho_n)$ with \redf{not too small} radius $\rho_n>0$. We first define the points $\{c_n\}_n$ that we call the centers of $\{U_n\}_n$.

\begin{dfn}[Centers $\{c_n\}_n$]
\label{dfn:centers-cn}
For every $n\in\N$, define the point
$$
\hat{c}_{n,n}(\mathbf{w}) := g^{-1}(\phi(ih_n)) \in \widehat{U}_{n,n}
$$
and, for $n\geqslant j\geqslant M+1$, define recursively
$$
c_{n,j}(\mathbf{w}) := \phi(\hat{c}_{n,j}(\mathbf{w}))\in U_{n,j}\quad \text{ and } \quad  
\hat{c}_{n,j-1}(\mathbf{w}) := g^{-1}(c_{n,j}(\mathbf{w}))\in Q_{j-1}.
$$% ($j=n, \dots, 1$) 
Finally, we put
$$
c_n=c_n(\mathbf{w}) :=  (\phi \circ g^{-1})^M \circ \phi(\hat{c}_{n,M}(\mathbf{w}))\in U_n, \quad \text{ for } n\ge N.
$$  
%Then we have  
%$c_{n,j}(\mathbf{w}) \in U_{n,j}$, $\hat{c}_{n,j-1}(\mathbf{w}) \in V_{n,j-1} \subset Q_{j-1}$ ($j=n, \dots, M+1$) 
%and $c_0(\mathbf{w}) \in U_n$.  
\end{dfn}

Note that the maps in the middle row of Diagram~1, such as $g:\widehat{U}_{n,j} \to U_{n,j+1}$ and $\phi: \widehat{U}_{n,j} \to U_{n,j}$, are all conformal isomorphisms.  
In the top row, the maps $g:Q_j\to g(Q_j)$ are also conformal isomorphisms and idependent of $\mathbf{w}\in\D_{3/4}^{\N_N}$ 
except the last one, from $\frac{1}{2} D_n$, on which $g$ is analytic and depends on $w_n$. The maps of the form $\phi:Q_j\to g(Q_{j-1})$ in the top row are also holomorphic, as $Q_j$ is disjoint from $\supp \mu_g$, but not surjective. We would like to emphasize that all the domains depend on $\mathbf{w}$, and we may write $U_n(\mathbf{w})$. 

\begin{rmk}
\redf{It follows from Lemma~\ref{lem:qc-estimate3} and the expansivity of the function $g$ on the rectangles $\{Q_n\}_n$ that $|\hat{c}_{n+1,n}-\hat{c}_{n,n}|$ is bounded and hence $|c_{n+1}-c_n|\to 0$ exponentially as $n\to\infty$, so the sequence of critical values $\{w_n\}_n$ converges to a point $w_\infty\in \overline{\D}(\tfrac{1}{2},\tfrac{1}{8})$.}
\label{rmk:w-infty}
\end{rmk}

Although the function that maps each sequence $\mathbf{w}\in \overline{\D}(\frac{1}{2}, \frac{1}{8})^{\N_N}$ to the Beltrami coefficient $\mu_{\phi_{\mathbf{w}}} \in L^{\infty}(\C)$ is not continuous, we can still prove the following.

\begin{lem}
\label{lem:holom-dependence-of-centers} 
Let $C(\C, \C)$ denote the space of all continuous maps from $\C$ to $\C$. The mapping 
$$
\begin{array}{ccc}
\overline{\D}(\frac{1}{2}, \frac{1}{8})^{\N_N} & \longrightarrow & C(\C, \C)\vspace{5pt}\\
\mathbf{w} & \mapsto & \phi_{\mathbf{w}} 
\end{array}
$$
is continuous, where 
we use the product topology in $\overline{\D}(\frac{1}{2}, \frac{1}{8})^{\N_N}$ and we use the compact open topology in $C(\C, \C)$.  As a consequence, for each $n \ge N$, the center $c_n(\mathbf{w})$ depends continuously on the sequence $\mathbf{w}$, and the mapping
$$
\begin{array}{ccc}
\overline{\D}(\frac{1}{2}, \frac{1}{8})^{\N_N} &\longrightarrow& \overline{\D}(\frac{1}{2}, \frac{1}{8})^{\N_N}\vspace{5pt}\\
 \mathbf{w} &\mapsto& (c_{N+1}(\mathbf{w}), c_{N+2}(\mathbf{w}), \dots) 
\end{array}
$$
is continuous with respect to the product topology.  
\end{lem}

We give the proof at the end of Appendix~\ref{sec:proof-of-estimates} (see p. \pageref{proof:lem:holom-dependence-of-centers}).

\section{Inner radius of $U_n$ and derivatives along the orbit of $c_n$}
\label{sec:derivatives}

In order to create the oscillating wandering domain, we will prove that for some choice of $\mathbf{w}\in\D_{3/4}^{\N_N}$, there exists $N_1\geqslant N$ so that 
$$
f^{n+2}(U_n)\subseteq U_{n+1},\quad \text{ for all } n\geqslant N_1,
$$
and hence the domains $\{U_n\}_n$ are contained in the orbit of the wandering domain for $n\geqslant N_1$. To that end, first we need to make sure that $U_n$ contains a disc with center $c_n(\mathbf{w})$ and sufficiently large radius $\rho_n>0$ so that 
$$
f\left(\phi(\tfrac{1}{2}D_n)\right)= g\left(\tfrac{1}{2}D_n\right)= \D\left(w_n,\left(\tfrac{1}{2}\right)^{2d_n}\right)\subseteq \D(c_{n+1}(\mathbf{w}),\rho_{n+1})\subseteq U_{n+1},
$$
for $n\geqslant N_1$. Then, in the next section we will do the second part of this process, that consists on choosing the sequence $\mathbf{w}\in\D_{3/4}^{\N_N}$.

%Note that the approach that we use here is completely different to that of \cite[Theorem~17.1]{bishop15}. Bishop uses a recursive argument to find integers $\{d_n\}_n$ sufficiently large so that the size of the domains $f^{n+1}(U_n)$ is small enough to fit inside $U_{n+1}$, while in our case the values $\{d_n\}_n$ are given in the beginning with explicit formulas, and here we only check that indeed they satisfy the required properties. 

We start by giving positive numbers $\{\rho_n\}_n$ with the property that $\D(c_n(\mathbf{w}),\rho_n)\subseteq U_{n}$ for all $n\geqslant N$. \vfour{In other words, $\rho_n$ is a lower bound of the inner radius of the domain $U_n$ with center at $c_n$ for $n\geqslant N$.}

\begin{lem}
\label{lem:inner-radius} 
There exists a constant $C_5>0$ such that if we define 
\begin{equation}
 \rho_n := \exp\left(-nC_5 - \sum_{j=0}^{n-1} x_j - x_{n-1}\right),\quad \text{ for } n\geqslant N, 
\end{equation}
then $\D(c_n(\mathbf{w}), \rho_n) \subset U_n$ for all $n\geqslant N$.  
\end{lem}
\begin{proof} 
Let $\Psi := (f^{n+1}|_{U_n})^{-1}$ which is a univalent mapping from $\D(\phi(ih_n), R_n')$ onto $U_n$, sending the center $\phi(ih_n)$ to $c_n(\mathbf{w})$.  
By the Koebe $1/4$-Theorem \cite[Corollary~1.4]{pommerenke75}, $U_n$ contains a disc with center $c_n(\mathbf{w})$ and radius $\frac{1}{4} |\Psi'(\phi(ih_n))| R_n'$.    
So we want to show that $\rho_n \le \frac{1}{4} |\Psi'(\phi(ih_n))| R_n'$ for all $n\ge N$.  
\par
By the chain rule, we have 
\begin{equation}
\label{eq:estimate-inverse-der}
\frac{1}{ \Psi'(\phi(ih_n)) } = (f^{n+1})'(c_n(\mathbf{w})) =\left((g\circ \phi^{-1})^M\right)'(c_n(\mathbf{w}))\cdot \prod_{j=M}^{n} g'(\hat{c}_{n,j}(\mathbf{w})) \cdot \frac{1}{\phi'(\hat{c}_{n,j}(\mathbf{w}))}. 
\end{equation}
Since $\log g'(z) = z + \log (1- e^{-2z})$ for the principal branch of $\log$, we have, for $\Re z>0$, 
\[
\left| \log |g'(z)| - \Re z \right| \le \left| \log g'(z) - z \right| = \left| \log (1- e^{-2z}) \right| \le \sum_{m=1}^{\infty} \frac{|e^{-2z}|^m}{m} \le \frac{e^{-2\Re z}}{1- e^{-2\Re z}}. 
\]
So for $j \ge M$, $\hat{c}_{n,j}(\mathbf{w}) \in Q_j$ and $\Re \hat{c}_{n,j}(\mathbf{w}) \ge 1$, hence $\log |g'(\hat{c}_{n,j}(\mathbf{w}))| \le x_j +2$. On the other hand, for $m\in\N$, 
$$
\red{B_m\subseteq \left\{z\in\C\ :\ |\textup{arg}\,z-\tfrac{\pi}{2}|<\arcsin\frac{r_m}{|\zeta_m|}\right\}}
$$
and by inequality~\eqref{eq:growth-estimate-inequalities} in Lemma~\ref{lem:growth-estimates}, we have \red{$r_m/|\zeta_m|<\tfrac{\pi}{6}$} for all $m\in\N$, so $\arcsin(r_m/|\zeta_m|)<\red{\tfrac{\pi}{4}}$ for all $m\in\N$. Since $|\textup{arg}\,z|<\blueb{\tfrac{\pi}{6}}$ for all $z\in Q_j$ with $j\geqslant M$, Lemma~\ref{lem:qc-estimate3} with $\theta=\tfrac{\pi}{12}$ implies that for $j\geqslant M$, $\log |\phi'(\hat{c}_{n,j}(\mathbf{w}))|$ is bounded by some constant. Finally, since $c_n(\mathbf{w})$ stays in a bounded region, $\log |\left((g\circ \phi^{-1})^M\right)'(\hat{c}_n(\mathbf{w}))|$ is bounded by some other constant.  

As $R_n'=\lambda R_n$ and $\log R_n$ was $\log \frac{x_{n+1}}{x_{n}} \sim x_{n} -x_{n-1}$ 
up to a constant. We obtain that there is $C_5>0$ such that 
\[
\log\! \rho_n\! =\! -nC_5 - \sum_{j=0}^{n-1} x_j - x_{n-1} 
\le \!\!\sum_{j=M}^{n} \bigl( \log\! |\phi'(\hat{c}_{n,j}(\mathbf{w}))| - \log\! | g'(\hat{c}_{n,j}(\mathbf{w})) | \bigr) +\log \tfrac{1}{4} + \log\! R_n', 
\]
which implies our assertion. Note that the crucial point here is that the term $x_n$ from $\log R_n'$ cancels with the term $x_n$ coming from $\log |g'(\hat{c}_{n,n}(\mathbf{w}))|$ in the sum.
\end{proof}

Next, we check that there is a constant $N_1\geqslant N$ so that \vfive{the degrees $\{2d_n\}_n$ given in Definition~\ref{dfn:hn-dn-Rn-En-Dn-Qn} are sufficiently large that $\text{diam}(f_\mathbf{w}(\tfrac{1}{2}D_n))<2\rho_{n+1}$ for all $n\geqslant N_1$.}

\pagebreak

\begin{lem} \label{lem:inclusion} 
There exists $N_1\ge N$ such that 
\begin{equation}
\left(\frac{1}{2}\right)^{2d_n} < \rho_{n+1},\quad \text{ for } n \ge N_1.  
\end{equation}
In particular, if $w_n = c_{n+1}(\mathbf{w})$ for \blueb{$n\ge N_1$}, then $f_{\mathbf{w}}^{n+2}(U_n) \subset U_{n+1}$ for \blueb{$n\ge N_1$}.  
\end{lem}
\begin{proof} 
The left hand side of the inequality satisfies
$$
\log\left(\left(\tfrac{1}{2}\right)^{2d_n}\right)=-2d_n\log 2=-2\frac{x_{n+1}}{x_n}\log2+O(1),\quad \text{ for } n\geqslant N,
$$
while the right hand side satisfies
$$
\log \rho_{n+1}=-(n+1)C_5-\sum_{j=0}^n x_j-x_n,\quad \text{ for } n\geqslant N,
$$
where $C_5>0$ is the constant from Lemma~\ref{lem:inner-radius}. \red{Since $x_{n+1}>e^{x_n}>6x_n^2$ for all $n\geqslant 3$}, and therefore there exists $N_1\geqslant N$ such that 
$$
\log\left(\left(\tfrac{1}{2}\right)^{2d_n}\right)<-\frac{x_{n+1}}{2x_n}<-3x_n<\log \rho_{n+1},\quad \text{ for } n\geqslant N_1,
$$
as required. If $w_n = c_{n+1}(\mathbf{w})$ for \blueb{$n\geqslant N_1$}, then 
$$
\begin{array}{rl}
f^{n+2}_{\mathbf{w}}(U_n) \hspace{-6pt}&\ds= f_{\mathbf{w}}( \D(\phi_{\mathbf{w}}(ih_n), R_n') ) \subset f_{\mathbf{w}}( \phi_{\mathbf{w}}(\tfrac{1}{2}D_n) ) 
= g_{\mathbf{w}}(\tfrac{1}{2}D_n) = \D(w_n, \left(\tfrac{1}{2}\right)^{2d_n}) \vspace{10pt}\\
&\ds\subseteq \D(c_{n+1}(\mathbf{w}), \rho_{n+1}) \subset U_{n+1},
\end{array}
$$
for all $n\geqslant N_1$.
\end{proof}

\section{Simultaneous shooting problem}
\label{sec:shooting}

It only remains to solve the equation $w_n = c_{n+1}(\mathbf{w})$. Observe that as we vary the parameters $\mathbf{w}=(w_n)_n$, the sequence $(c_n(\mathbf{w}))_n$ also moves, but more slowly. Recall that throughout the paper we use the notation $\N_N=\N \setminus\{1,\hdots,N-1\}$, where the constant $N\in\N$ was defined in Section \ref{sec:domains}.

%\pagebreak

%%%%%
\begin{lem} \label{lem:shooting} 
There exists $\mathbf{w}=(w_{N}, w_{N+1}, \dots) \in \overline{\D}(\frac{1}{2}, \frac{1}{8})^{\N_N}$ such that 
\begin{equation} \label{eq:shooting}
w_n = c_{n+1}(\mathbf{w}),\quad \text{ for } n \ge N.  
\end{equation}
\end{lem}

%%%%%
\begin{proof} 
Fix $T\ge N$ and take $\mathbf{w}'' = (w_{T+1}, w_{T+2}, \dots) \in \barD(\frac{1}{2}, \frac{1}{8})^{\N_{T+1}}$. For example, we can just set $\mathbf{w}''$ to be the constant sequence $w_j=\tfrac{1}{2}$ for all $j> T$. We first try to solve the finite shooting problem 
for 
$$
\mathbf{w}'=(w_N, w_{N+1}, \dots, w_T) \in \barD(\tfrac{1}{2}, \tfrac{1}{8})^{\N_N\setminus\N_{T+1}}.
$$  
Writing $\mathbf{w} = (\mathbf{w}', \mathbf{w}'')=(w_N, \dots) \in \barD(\frac{1}{2}, \frac{1}{8})^{\N_N}$, by Lemma~\ref{lem:holom-dependence-of-centers}, the function
$$
\begin{array}{ccc}
\barD(\frac{1}{2}, \frac{1}{8})^{\N_N\setminus \N_{T+1}} &\longrightarrow &  \barD(\frac{1}{2}, \frac{1}{8})^{\N_N\setminus \N_{T+1}}\vspace{5pt}\\
 \mathbf{w}' & \mapsto & (c_{N+1}(\mathbf{w}), \dots, c_{T+1}(\mathbf{w})) 
 \end{array}
$$
is a continuous mapping from a $(T-N+1)$-dimensional closed polydisc to itself.  By Brouwer's Fixed Point Theorem, this map has a fixed point, which is a solution of the equation
$$
w_n = c_{n+1}(\mathbf{w}),\quad \text{ for } N\leqslant n\leqslant \blue{T}.
$$  
Denote one of these solutions by $\mathbf{w}_{T}$.  
By the compactness of $\overline{\D}(\frac{1}{2}, \frac{1}{8})^{\N_N}$, we can find a convergent subsequence 
$\mathbf{w}_{T_k} \to \mathbf{w}_{\infty}$ as $k \to \infty$.  
By Lemma~\ref{lem:holom-dependence-of-centers}, $\mathbf{w}=\mathbf{w}_{\infty}$ is a solution for $w_n = c_{n+1}(\mathbf{w}) \text{ for } n \ge N.$ 
\end{proof}

\begin{rmk}
%\red{In the proof of Lemma~\ref{lem:shooting}, we use Rouch\'e's Theorem to solve the simultaneous shooting problem. At the end of Section~\ref{sec:proof-main-thm} we outline a modification of the construction in which $g_{\mathbf{w}}$ does not depend holomorphically on $\mathbf{w}$ where we replace Rouch\'e's Theorem by Brouwer's Fixed Point Theorem that only requires continuity. }
When $\phi_{\mathbf{w}}$ is constructed so that $\mu_{\phi_{\mathbf{w}}}$ is holomorphic with respect to~$w_n$ for $n\geqslant N$ (as in Definition~\ref{dfn:gw}), 
one can also use Rouch\'e's Theorem \redf{repeatedly, one coordinate at a time}. In that case we obtain that the solution is unique.
\end{rmk}

%\begin{rmk}
%It has been pointed out to us by Bishop that a similar approach has been developed independently by Lazebnik in a work that is not yet available.
%\end{rmk}

\section{Proof of the main theorem}
\label{sec:proof-of-main-thm}
\label{sec:proof-main-thm}

We are now ready to prove Theorem~\ref{thm:main}. Most of the work has already been done in the previous sections, but it remains to check that indeed the sets $\{U_n\}_n$ are contained in the grand orbit of a wandering domain $U$ of $f$ \redf{for all $n\geqslant N_1$}.

\begin{proof}[Proof of Theorem \ref{thm:main}]
Take the sequence $\mathbf{w}\in \overline{\D}(\frac{1}{2}, \frac{1}{8})^{\N_N}$ as in Lemma~\ref{lem:shooting}, and consider $f= f_{\mathbf{w}}$.  
By Lemma \ref{lem:inclusion}, $f^{n+2}(U_n) \subset U_{n+1}$ for $n \ge N_1$.  
So it follows by induction that 
\begin{equation}
\label{eq:main-thm-eq1}
f^{kN_1 + \frac{k(k+3)}{2}}(U_{N_1}) \subset U_{N_1 + k} \subset  \D\left(\tfrac{1}{2},\tfrac{1}{8}\right),\quad \text{ for all } k\in\N.
\end{equation}
This means that any point in $U_{N_1}$ has bounded derivatives for $f^{kN_1 + \frac{k(k+3)}{2}}$, 
hence cannot be a repelling periodic point.  Therefore $U_{N_1}$ is contained in the Fatou set, let $U$ be the Fatou component that contains $U_{N_1}$.  

Let us show that $U$ is a wandering domain.  
Assume to the contrary that $U$ is an eventually periodic Fatou component.  
Then \eqref{eq:main-thm-eq1} implies that there exists a finite limit function $l=\lim_{k \to \infty} f^{n_k}$ on $U_{N_1}$.  
\blue{Since functions in the class $\mathcal B$ have no escaping Fatou components \cite{eremenko-lyubich92}, $\infty$ cannot be a limit function, and by the classification of periodic Fatou components (see \cite[Section~4.2]{bergweiler93}), all other possible limit functions are bounded on $U_{N_1}\subseteq\D(\tfrac{1}{2},\tfrac{1}{8})$.}  
On the other hand,
\begin{equation}
\label{eq:main-thm-eq2}
f^{kN_1 + \frac{k(k+3)}{2}+N_1+k+1}(U_{N_1}) \subset f^{N_1+k+1}(U_{N_1 + k}) 
\subset \phi(D_{N_1+k}),\quad \text{ for all } k\in\N,
\end{equation}
and, by Lemma \ref{lem:phi(Q)-in-g(Q)}, $\phi(D_{N_1+k})\subseteq g(Q_{N_1+k})\subseteq \C\setminus \D(0,e^{x_{N_1+k}-1}/2)$ and we have $x_{N_1+k}\to +\infty$ as $k\to\infty$. This is a contradiction. Therefore $U$ is a wandering domain, which must be oscillating by \eqref{eq:main-thm-eq1} and \eqref{eq:main-thm-eq2}.  

It remains to check that the order of $f$ is $1$. By Lemma \ref{lem:qc-estimate1}, it suffices to check that $g=g_{\mathbf{w}}$ has order $1$. Suppose that the circle of radius $r>0$ intersects the set~$E_n$, so $h_n-2d_n\pi<r<h_n+3R_n$. The image of the arc that is contained in $E_n$ is mapped inside the ellipse $\mathcal E_{2d_n\pi}\subseteq \D(0,2\cosh(2d_n\pi))$ by $g$. But we have $g(r)>2\cosh(h_n-2d_n\pi)>2\cosh(2d_n\pi)$ as $h_n>4d_n\pi$ for all $n\geqslant 3$ by Lemma~\ref{lem:growth-estimates}. So the function $M(r,g_{\mathbf{w}})=\max_{|z|=r}|g_{\mathbf{w}}(z)|=M(r,2\cosh)$ for all sufficiently large $r>0$. Hence, $\rho(g_{\mathbf{w}})=\rho(\cosh)=1$, and this ends the proof of the main theorem.
\end{proof}

\section{Generalisation to entire functions of order $p/2$}

\label{sec:oder-half}

In this section we modify the previous construction to obtain functions of order $p/2$ for any $p\in\N$. It is of particular interest the case $p=1$ which give a function of order equal to one half. Recall that functions in the class $\mathcal B$ have lower order (and hence also order) greater or equal to one half \cite{heins48}. Therefore, the example that we provide here with $p=1$ has the lowest possible order in the class $\mathcal B$.

\begin{proof}[Proof of Theorem~\ref{thm:oder-half}]
Let $p \in \N$.  
Our model map is now 
\begin{equation}
\breve{g}_p(z) := 2 \cosh \bigl( z^{\frac{p}{2}} \bigr), \quad \vfive{\text{ for } z\in\C}.
\end{equation}  
Note that even if $p$ is an odd integer, $g_p$ is well-defined because $\cosh z$ is an 
even function.  
The reference orbit is now
\begin{equation}
\xi_0 := \tfrac{1}{2},\quad  \xi_n := \breve{g}_p^n(\xi_0),\quad \text{ for } n\in\N, 
\end{equation}
and the new critical orbit is given by $v_0:=0$ and $v_n := \breve{g}_p^{n}(v_0)$, for $n\in\N$.
Set 
\begin{equation}
x_n := \xi_n^{\frac{p}{2}},\quad \text{ for }n\in\N. 
\end{equation}
Then,
\red{$\xi_{n+1} = 2 \cosh x_{n} =e^{x_n}(1+o(1))$ with $\xi_{n+1}\ge e^{x_n}$} and 
$x_{n+1} = \xi_{n+1}^{\frac{p}{2}} \ge e^{\frac{p}{2} x_n}$ for $n\in\N$ and thus, the sequence $(x_n)_n$ 
also has a fast growth as in Lemma \ref{lem:growth-estimates}. We use $(x_n)_n$ to define the quantities $d_n$, $R_n$ and $h_n$ for $n\in\N$ as in Section~\ref{sec:setup}.  
We use exactly the same function $g_{\mathbf{w}}$ as before, which is an even function, and 
define 
\begin{equation}
g_{p, \mathbf{w}} (z) := g_{\mathbf{w}} \bigl( z^{\frac{p}{2}} \bigr),\quad \vfive{\text{ for } z\in \C}.
\end{equation}  
Let $\phi_{\mathbf{w}}$ be \vfive{the} $K$-quasiconformal mapping with $\mu_{\phi_\mathbf{w}} = \mu_{g_{p, \mathbf{w}}}$ \vfive{normalised with $\phi_\mathbf{w}(0)=0$ and $\phi_\mathbf{w}(1)=1$,} and define the entire function $f_{p, \mathbf{w}} := g_{p, \mathbf{w}} \circ \phi_{\mathbf{w}}^{-1}$. \blue{At this point one could redo the computations that we did for the function $g_\mathbf{w}$ for the function $g_{p,\mathbf{w}}$, but we choose a different approach in order to use as much as possible of what we know about $g_\mathbf{w}$.} 
The \vfive{Beltrami coefficient} $\mu_{\phi_{\mathbf{w}}}$ \vfive{is} the pull-back of $\mu_{g_{\mathbf{w}}}$ by the map 
$z \mapsto z^{\frac{p}{2}}$. In Diagram 1, we replace the quasiconformal map $\phi_\mathbf{w}$ by 
\begin{equation}
\psi_{\mathbf{w}}(z) := \phi_{\mathbf{w}}\bigl (z ^{\frac{2}{p}} \bigr),
\end{equation}
which is defined for $z \in \C \sminus \R_-$ and, \vfive{for $p>2$}, we use the branch of \vfive{the multivalued map} $\vfive{z\mapsto\,} z ^{\frac{2}{p}}$ \blue{with image $A_p:=\{z\in\C\ :\ |\text{arg}\, z|<\frac{2\pi}{p}\}$} \vseven{(see Diagram 2).}
$$
\xymatrix{
\C & \ar[l]_{\phi_{\mathbf{w}}} A_p\ar@/^2.0pc/[rr]^{g_{p,\mathbf{w}}} \ar[r]^{z^{\frac{p}{2}}} & \C\setminus \R_-\ar@/^2.0pc/[ll]^{\psi_{\mathbf{w}}} \ar[r]^{g_{\mathbf{w}}} & \C
}
$$
\begin{center}
{\sc Diagram 2.} The maps from the construction in Theorem~\ref{thm:oder-half}.
\end{center}
The entire function \vseven{$f_{p,\mathbf{w}}$} extends the composition of the three maps in the central row, namely
\begin{equation}
\vseven{f_{p,\mathbf{w}}=g_\mathbf{w}\circ \psi_\mathbf{w}^{-1}=g_{p,\mathbf{w}}\circ \phi_\mathbf{w}^{-1}.}
\end{equation}

\red{A derivative estimate shows that if $|\text{arg}\, \zeta| \le \pi/2$ and $r/|\zeta|<1/4$, then the aforementioned branch of $z^{2/p}$ maps $\D(\zeta, r)$ into 
$\D(\zeta^{2/p}, c |\zeta|^{2/p-1} r)$ with $c=\frac{2}{p}(3/4)^{2/p-1}$ (if $p \ge 2$) and $c=5/2$ (if $p=1$). The inverse images of $\pm ih_n$ by the map $z \mapsto z^{p/2}$ are the $p$-th roots of $-h_n^2$, which can be
 written as $ \zeta_{n,\ell} = \omega^{2\ell + 1}(h_n)^{2/p} = \omega^{2\ell+1} \xi_{n+1}(1+o(1))$, 
 where $\omega=e^{\pi i /p}$, for $0\leqslant \ell\leqslant p-1$. The support of $\mu_{\phi_{\mathbf{w}}}$ is contained in discs $B_m=\D(\zeta_{n, \ell}, r_n)$, 
where }
$$
\red{r_n = c  (h_n)^{2/p-1} 3R_n = 3c \frac{x_{n+1}^{2/p}}{x_n}(1+o(1))  = 3c \frac{\xi_{n+1}}{x_n}(1+o(1))\quad \text{ as } n\to\infty. }
$$ 
\blue{From Lemma~\ref{lem:growth-estimates}, the discs $\{B_m\}_m$ also satisfy the estimates required for Lemmas \ref{lem:qc-estimate1}, \ref{lem:qc-estimate2} and \ref{lem:qc-estimate3}.  
In particular, \vseven{$\psi_\mathbf{w}(\frac{1}{2}D_{n})$} contains a disc centered at \vseven{$\psi_\mathbf{w}(ih_{n})=\phi_\mathbf{w}(\zeta_{n,0})$} of radius  $R_n''=C_4'\frac{\xi_{n+1}}{x_n}$ for all $n\geqslant N$.  

\par
Let us suppose $p \ge 2$.  Then 
we take $Q_j=Q(x_j)$ as before, then $g(Q_j)$ contains the set $\{z\in\C\, :\, \tfrac{1}{2}e^{x_j}<|z|<2e^{x_j}\}\setminus \R_-$, while the image of $Q_{j+1}$ by 
the map $z\mapsto z^{2/p}$ is contained in the disc \vseven{$\D(x_{j+1}^{2/p}, cx_{j+1}^{2/p-1}(1+\pi)) 
= \D(\xi_{j+1}, c\frac{\xi_{j+1}}{x_{j+1}}(1+\pi))$. Since $\xi_{j+1}=e^{x_j}(1+o(1))$, we have 
$\psi_{\mathbf{w}}(Q_{j+1})\subseteq g(Q_j)$ for $j\geqslant M$ by Lemma~\ref{lem:qc-estimate1}. Similarly, from $\zeta_{n,0}=\omega e^{x_n}(1+o(1))$, we have $\psi_\mathbf{w}(D_n) \subset g(Q_n)$ for $n \ge N$.}  \par
For $p=1$, we make a special modification: for $n\geqslant M$, we replace the rectangle $Q_n$ by 
$Q'_n=Q'(x_n):= \{z \in \C: \, |\Re z-x_n| < 1, \ - \frac{\pi}{2} < \Im z < \frac{3\pi}{2}\}$ so that $g(Q'_n)$ covers 
$\psi(D_n)$, which in this case is near the negative real axis. It is immediate to check that in this case we still have that $\vseven{\psi_\mathbf{w}}(Q_{j+1}') \subset g(Q_j')$ for $j\geqslant M$ and $\vseven{\psi_\mathbf{w}}(D_n) \subset g(Q_n')$ for $n\geqslant N$.  
\par 
Let $\Psi = (f^{n+1}|_{U_n})^{-1}:\D(\psi(ih_n), R_n'') \to U_n$ for any $n\geqslant N$. Since 
$$
\psi_{\mathbf{w}}'(z)=\phi_{\mathbf{w}}'(z^{2/p}) \frac{2}{p} z^{\frac{2}{p}-1}
$$
and, by Lemma~\ref{lem:qc-estimate3}, $\log|\vseven{\phi'_\mathbf{w}}(\hat{c}_{n,j}(\mathbf{w})^{2/p})|$ is bounded,
we have the following estimate:
\begin{align}
\log \left( \tfrac{1}{4} |\Psi'(\psi(ih_n))| R_n'' \right) 
&\!=\! \sum_{j=0}^n \left( \log |\psi_\mathbf{w}'(\hat{c}_{n,j}(\mathbf{w}))| 
\!-\! \log |g_{\mathbf{w}}'(\hat{c}_{n,j}(\mathbf{w}))| \right) 
\!+\! \log R_n'' \!\nonumber\\ 
&\!=\! \sum_{j=0}^n \left(\!(\tfrac{2}{p}\!-\!1) \log |\hat{c}_{n,j}(\mathbf{w})| 
\!-\! \log |g_{\mathbf{w}}'(\hat{c}_{n,j}(\mathbf{w}))| \right) 
\!+\! \log R_n'' \!+\! O(n).\nonumber
\end{align}
As $\hat{c}_{n,j}(\mathbf{w}) \in Q(x_j)$ and $\log x_j = \frac{p}{2} \log \xi_j$ and $\log \xi_j = x_{j-1}+o(1)$, 
we have 
\begin{align}
\log \left( \tfrac{1}{4} |\Psi'(\psi(ih_n))| R_n'' \right) 
&\!=\! \sum_{j=0}^n \left(\! (\tfrac{2}{p}\!-\!1) \log x_j 
\!-\! x_j \right) \!+\! 
\log \left(C'' \frac{\xi_{n+1}}{x_n}\right)  
\!+\!O(n)\nonumber\\
&=  \sum_{j=1}^n \left( (1-\tfrac{p}{2}) x_{j-1}
- x_j \right) + x_{n} - \tfrac{p}{2}x_{n-1}+O(n) \nonumber\\
&= -\tfrac{p}{2}\sum_{j=0}^{n-1}  x_j - \tfrac{p}{2} x_{n-1} + O(n).
\end{align}
In Lemma \ref{lem:inner-radius}, instead of the inner radius $\{\rho_n\}_n$, we can take  
\begin{equation}
\rho_n' :=\exp\left( -nC_5' - \tfrac{p}{2}\sum_{j=0}^{n-2} x_j - px_{n-1}\right),\quad \text{ for } n\geqslant N,   
\end{equation}
for some constant $C_5'$, and then the rest of argument goes similarly with $d_n =\lfloor \frac{x_{n+1}}{x_n}\rfloor$ for $n\geqslant N$.}
\end{proof}

%\section{Entire functions with no zeros}
%
%\label{sec:nozeros}
%
%%Entire functions with no zeros are all of the form $f=\exp \circ\,g$ for some non-constant entire function $g$. P\'olya \cite{polya} showed that if an entire function with no zeros has finite order, then necessarily it is of the form $f=\exp\circ\,P$ with $P$ being a polynomial. It is easy to see that such functions have at most as many critical values as the degree of $P$ and $0$ is the only finite asymptotic value. Thus, entire functions of finite order with no zeros are all in the Speiser class $\mathcal S$ and hence, they have no wandering domains.
%
%In this section we prove Theorem~\ref{thm:nozeros}, which says that there exists a function $g\in\mathcal{B}$ such that $f=\exp\circ g\in\B$ has an oscillating wandering domain. To do so, we use the fact that $|2\cosh z-e^z|=|e^{-z}|=e^{-\text{Re}\,z}$ is small on the sets $Q_n$, for $n\geqslant 3$.
%
%\red{Change the goal in the shooting part. Check that the estimates still work, with the error introduced by the fact that $g(z)-e^z=e^{-z}$. Make sure that we can use the iterates with the same parity to go to $D_n$}

\begin{rmk}
\redf{Mihaljevi\'{c}-Brandt and Rempe-Gillen \cite[Theorem~6.1]{mihaljevicbrandt-rempegillen13} proved that if $f\in\mathcal B$ has order $\rho(f)<1$, $f(\R)\subseteq \R$ and all the zeros of $f$ lie in the negative real line, then $f$ does not have wandering domains. When $p=1$, our construction provides a function $f_{1,\mathbf{w}}\in\mathcal B$ that is real and has order $\rho(f_{1,\mathbf{w}})=1/2$. However, despite the zeros of the base function $\breve{g}_1(z) := 2 \cosh \bigl( z^{\frac{1}{2}} \bigr)$ lie in the negative real line, this does not contradict our result as the zeros of $f_{1,\mathbf{w}}$, which are the image by $\phi_\mathbf{w}$ of the zeros of $g_{1,\mathbf{w}}$, are not in the negative real line.}
\end{rmk}

\section{Real-symmetric version of the construction}

\label{sec:real-symmetric-version}

As we mentioned in the introduction, the base function $g(z)=2\cosh z$ is both even and symmetric with the real line. However, the modified function $g_\mathbf{w}$ is even but not real-symmetric. To construct functions with an arbitrary number of grand orbits of wandering domains, it is more convenient to use a slightly different modification $\tilde{g}_\mathbf{w}$ of $g$ that is symmetric with respect to $\R$ (but not even).

\begin{dfn}[$K$-quasiregular map $\tilde{g}_{\mathbf{w}}$]
\label{dfn:gw-2}
For $n\in \N$, let $d_n\in\N$ and $R_n,h_n\in\R_+$ be the numbers given in Definition~\ref{dfn:hn-dn-Rn-En-Dn-Qn} and, for $n\geqslant 3$, let the sets $E_{\pm n}$, $D_{\pm n}$ and the function $G_n:E\to \overline{\mathcal{E}}_{2d\pi}$ be as in Definition~\ref{dfn:gw}. For every sequence $\mathbf{w} = (w_{N}, w_{N+1}, w_{N+2}, \dots) \in \D_{3/4}^{\N_N}$, define the function $\tilde{g}_{\mathbf{w}}: \C \to \C$ as follows:  
\begin{align}
\tilde{g}_{\mathbf{w}}(z): = 
\begin{cases} 
G_n(z \mp ih_n), \ &\text{ if } z\in E_{\pm n} \sminus D_{\pm n}   \text{ with } n \geqslant N, \\
\rho_{w_n} \circ G_n(z - ih_n), \ &\text{ if } z\in  D_{n}  \text{ with } n \geqslant N, \\
\rho_{\overline{w_n}} \circ G_n(z + ih_n), \ &\text{ if } z\in  D_{n}  \text{ with } n \leqslant -N, \\
2 \cosh z, \ &\text{ otherwise. }
\end{cases}
\end{align}%\margin{problem with\\ holomorphic\\ dependence?}
Then $\tilde{g}_{\mathbf{w}}$ is a $K$-quasiregular map such that 
$$
\supp \mu_{\tilde{g}_{\mathbf{w}}} \subseteq \bigcup_{n \in \Z_N} E_n \setminus \D\left(ih_n, \vseven{\left(\tfrac{1}{8}\right)^{1/(2d_n)}R_n}\right),
$$ 
and $\tilde{g}_{\mathbf{w}}(\overline{z})=\overline{\tilde{g}_{\mathbf{w}}(z)}$ for all $z\in\C$.
\end{dfn}

\begin{dfn}[Entire function $\tilde{f}_{\mathbf{w}}$ and $K$-quasiconformal map $\tilde{\phi}_{\mathbf{w}}$]
\label{dfn:tilde-fw-phiw}
Let $\tilde{g}_{\mathbf{w}}$ be the $K$-quasiregular map from Definition~\ref{dfn:gw-2}. By the Measurable Riemann Mapping Theorem \cite[Theorem~3]{ahlfors06}, 
there exists a unique $K$-quasiconformal mapping $\tilde{\phi}_{\mathbf{w}} : \C \to \C$ such that 
$\tilde{\phi}_{\mathbf{w}}(0) =0$, $\tilde{\phi}_{\mathbf{w}}(1)=1$ and $\mu_{\tilde{\phi}_{\mathbf{w}}} = \mu_{\tilde{g}_{\mathbf{w}}}$.  
Then define 
\begin{equation}
\tilde{f}_{\mathbf{w}} := \tilde{g}_{\mathbf{w}} \circ \tilde{\phi}_{\mathbf{w}}^{-1}, 
\end{equation}
which is an entire function \vfive{in the class $\mathcal{B}$ with $S(f_\mathbf{w})=\{\pm 2\}\cup
\overline{\{w_n,\overline{w_n}\}}_n$}. Since $\tilde{g}_\mathbf{w}$ (and, in particular, $\mu_{\tilde{g}_{\mathbf{w}}}$) is symmetric with respect to the real line, it follows that $\tilde{\phi}_\mathbf{w}$ preserves the real line (see \cite[Exercise~1.4.1]{branner-fagella14}), and so does $\tilde{f}_\mathbf{w}$. 
\end{dfn}

Recall that in Lemma~\ref{lem:escaping-real-orbits} we showed that the base map $g(z)=2\cosh z$ satisfies $\R\subseteq I(g)\subseteq J(g)$. We now prove that the same holds for the function $\tilde{f}_\mathbf{w}$ for any $\mathbf{w}\in\D_{3/4}^{\N_N}$ provided that $N\in\N$ is large enough.

\begin{lem}
\label{lem:real-symmetry}
The functions $\tilde{\phi}_\mathbf{w}$ and $\tilde{f}_{\mathbf{w}}$ are symmetric with respect to the real line. If $N\in\N$ in Definition~\ref{dfn:gw-2} is sufficiently large, then $\R\subseteq I(\tilde{f}_{\mathbf{w}})\subseteq J(\tilde{f}_{\mathbf{w}})$.
\end{lem}
\begin{proof}
To show that $\R\subseteq I(\tilde{f}_\mathbf{w})$, recall that, by Lemma~\ref{lem:qc-estimate3}, we know there exists $C_3>1$ such that
$$
(\tilde{\phi}_\mathbf{w}^{-1})'(x)\geqslant \frac{1}{C_3},\quad \text{ for all } x\in\R,
$$
and we can make this constant be arbitrarily close to $1$ by choosing \blueb{$N$} sufficiently large. In particular, \blueb{we can put $\eta=1$ in Lemma~\ref{lem:qc-estimate3} and choose $N\geqslant M_2(\eta)$ so that $1<C_3<2$.} Then, 
$$
\tilde{f}_\mathbf{w}'(x)-1=2\sinh(\tilde{\phi}_\mathbf{w}^{-1}(x))(\tilde{\phi}_\mathbf{w}^{-1})'(x)-1\geqslant \frac{2x}{C_3^2}-1>0,\quad \text{ for } x>2,
$$
and since
$$
\tilde{f}_\mathbf{w}(2)-2\geqslant 2\cosh(1)-2>1,
$$
arguing as in Lemma~\ref{lem:escaping-real-orbits}, we obtain that $\R\subseteq I(\tilde{f}_\mathbf{w})\subseteq J(\tilde{f}_\mathbf{w})$.
\end{proof}

\begin{lem}
\label{lem:separation}
The curves $L_k:=\tilde{\phi}_\mathbf{w}(\{z\in\C\, :\, \textup{Im}\,z=k\pi\})$, $k\in\Z$, are a subset of $J(\tilde{f}_\mathbf{w})$. For every different $n_1,n_2\in\Z_N$, there is $k\in\Z$ such that the discs $\D(\tilde{\phi}_\mathbf{w}(ih_{n_1}),R_{n_1}')$ and $\D(\tilde{\phi}_\mathbf{w}(ih_{n_2}),R_{n_2}')$ lie in different complementary components of $L_k$.
\end{lem}
\begin{proof}
Recall that by Lemma~\ref{lem:inside-half-disk}, $\D(\tilde{\phi}_\mathbf{w}(ih_{n}),R_{n}')\subseteq \tilde{\phi}_\mathbf{w}(\tfrac{1}{2}D_n)$ for all $n\in\Z_N$. The result follows from the fact that $\tilde{\phi}_\mathbf{w}$ is a plane homeomorphism and that there is a line of the form $\{z\in\C\, :\, \textup{Im}\,z=k\pi\}$ for some $k\in\Z_N$ that separates $\tfrac{1}{2}D_{n_1}$ and $\tfrac{1}{2}D_{n_2}$.
\end{proof}

%\begin{rmk}
%We want to remark that although it simplifies some of the proofs (for example, it is useful when showing that the wandering domains are bounded), our %construction does not require that $\phi_\mathbf{w}(\R)=\R$. Thus, in order to construct functions with an odd number of wandering domains, we may use a %different definition of $g_\mathbf{w}$ that is not symmetric with respect the real line (see Remark~\ref{rmk:odd-number-wd}).
%\end{rmk}

\section{Functions with any number of wandering domains}

\label{sec:infinitely-many}

Now we prove Theorem~\ref{thm:infinitely-many}, which says that we can obtain functions of finite order in the class $\B$ with $2q$ grand orbits of oscillating wandering domains for any $q\in\N\cup\{\infty\}$. \red{We use the $\R$-symmetric version of the construction introduced in the previous section.} The strategy consists of introducing a sequence of sets $\{E_n^j\}_{n\geqslant 3+j}$ for each $0\leqslant j<q$ that creates a new grand orbit of wandering domains (see Figure~\ref{fig:setup3}). We can do this because the distance between the \blueb{original} sets $E_n$ and $E_{n+1}$ tends to infinity sufficiently fast and also because $g(Q_n)$ is much larger than $E_n$. To prove that the resulting function has no other wandering domains apart from the ones given by the construction we follow \cite[Theorem~A]{fagella-godillon-jarque15}.

%In \cite{fagella-godillon-jarque15} it was shown that the function from \cite[Theorem~17.1]{bishop15} has exactly two grand orbits of wandering domains, namely one that is eventually contained in the upper half plane, and its complex conjugate as $f(\overline{z})=\overline{f(z)}$. And so do all modifications of this function in \cite{fagella-godillon-jarque15,kirill, }. Since.... it is not easy to see how to modify Bishop's construction to obtain infinitely many grand orbits of wandering domains. In contrast, it is very easy to modify our construction to produce functions with any number of grand orbits of wandering domains, including functions with infinitely many grand orbits of wandering domains. 

\begin{proof}[Proof of Theorem~\ref{thm:infinitely-many}]
Fix $q\in\N\cup\{\infty\}$. Let $d_n\in \N$, $h_n\in 2\pi\N$ and $R_n>0$ for $n\in\N$ be as in Definition~\ref{dfn:hn-dn-Rn-En-Dn-Qn}. \blue{Then, for $n\geqslant 3$ and $0\leqslant j< n-2$, put
\begin{equation}
h_{n}^j:=h_n+j6R_n
\end{equation}
so that $h_{n}^0=h_n$ and $h_{n}^j-h_n=j6R_n=j(6d_n\pi-2\pi)\in 2\pi\N$. For $n\geqslant 3$ and $0\leqslant j< \min\{n-2,q\}$, define the sets
\begin{equation}
\begin{array}{c}
E_{\pm n}^j:=\{z\in \C\, :\, |\textup{Re}\,z|<2d_n\pi,\ |\textup{Im}\,z\mp h_{n}^j|<2d_n\pi\}=E_{\pm n}\pm j6R_n,\vspace{5pt}\\
D_{\pm n}^j:=\D(\pm ih_{n}^j,R_n)=D_{\pm n}\pm j6R_n\subseteq E_{\pm n}^j.
\end{array}
\end{equation}
Observe that at the level $n=3$ we only have the square $E_3^0=E_3$, when $n=4$ we have $E_4^0=E_4$ and $E_{4}^1=E_4+6R_4$, and we keep adding one square at a time as we increase $n$. Since $h_{n+1}-3R_{n+1}>h_n+3R_n+6nR_n$ by the inequality (3) in Lemma~\ref{lem:growth-estimates}, the squares $\{E_{\pm n}^j\}_{n,j}$ are all disjoint (see Figure~\ref{fig:setup3}).

\blue{For $0\leqslant j< q$, define $N_j:=\max\{N,3+j\}$ and let 
\begin{equation}
\mathbf{w}^j=(w_{N_j}^j,w_{N_j+1}^j,\hdots)\in \D_{3/4}^{\mathbb N_{N_j}}.
\end{equation}
Then, we write $\mathbf{w}=(\mathbf{w}^0, \mathbf{w}^1,\hdots)$. In this case we define the function $\tilde{g}_{\mathbf{w}}$ as in Definition~\ref{dfn:gw-2} but with more squares. For $N\geqslant 3$, define
\begin{equation}
\tilde{g}_{\mathbf{w}}(z): = 
\begin{cases} 
G_n(z \mp ih_{n}^j),  &\!\!\!\text{if } z\in E_{\pm n}^j \sminus D_{\pm n}^j \text{ with } n \geqslant N, \ 0\leqslant j<\min\{n-2,q\}, \\
\rho_{w_{n}^j} \circ G_n(z - ih_{n}^j),  &\!\!\!\text{if } z\in  D_n^j  \text{ with } n \geqslant N,\ 0\leqslant j<\min\{n-2,q\}, \\
\rho_{\overline{w_{n}^j}} \circ G_n(z + ih_{n}^j),  &\!\!\!\text{if } z\in D_n^j  \text{ with } n \leqslant -N,\ 0\leqslant j<\min\{n-2,q\}, \\
2 \cosh z,  &\!\!\!\text{otherwise,}
\end{cases}
\end{equation}
}where $G_n:E\to \overline{\mathcal{E}}_{2d\pi}$ is the quasiregular mapping from Lemma~\ref{lem:key-interpolation} for $d=d_n$ and $R=R_n$. Then, the Measurable Riemann Mapping Theorem gives an entire function $\tilde{f}_{\mathbf{w}}$ and quasiregular map $\tilde{\phi}_{\mathbf{w}}$ such that $\tilde{f}_{\mathbf{w}}=\tilde{g}_{\mathbf{w}}\circ \tilde{\phi}_{\mathbf{w}}^{-1}$.

\vfour{There is $N\geqslant 3$ sufficiently large such that if we define
\begin{equation}
B_{m}:=\D(ih_{n}, h_{n}^{n-3}-h_{n}+3R_n),\quad \text{ for } m\in\N,
\end{equation}
where, after reindexing, $m\in\N$ corresponds to $n\in\Z_N$ as before. For $n\in\Z_N$, we have $E_{n}^j\subseteq B_{m}$  for all $0\leqslant j<\min\{n-2,q\}$ and,} by inequality \eqref{eq:growth-estimate-infinitely-many} in Lemma~\ref{lem:growth-estimates}, the discs~\blueb{$\{B_m\}_m$} \vfive{and the $K$-quasiconformal map $\tilde{\phi}_\mathbf{w}$} satisfy the Assumption~\ref{ass:qu-estimates}.} %Moreover, any two such discs are at least a distance $2\pi$ apart. We may define $B_{2m}$ and $B_{2m+1}$ to be the $m$th discs in the sequences $\{B_{2m}^k\}_{m,k}$ and $\{B_{2m+1}^k\}_{m,k}$, respectively, ordered by the imaginary part of their points.%Then, any two of such sets $B_m$ and $B_{m'}$ are at least a distance $2\pi$ apart. \margin{is there a\\  line in $J(f)$ separating\\ them?\\ ~\\ remove the\\ negative $h_n$?}

\begin{figure}[h]
\def\svgwidth{.75\linewidth}
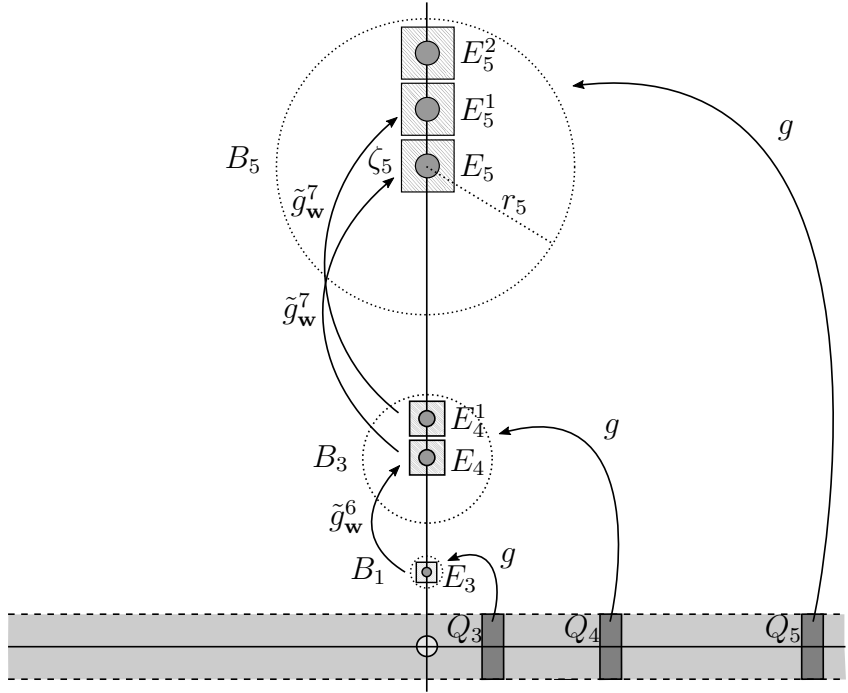
\caption{Sketch of the squares $\{E_n^j\}_{n,j}$ and one possible indexing of the discs $\{B_m\}_m$ for $N=3$ in the construction of Theorem~\ref{thm:infinitely-many}.}
\label{fig:setup3}
\end{figure}

\blue{Since the size of the discs $\{B_{m}\}_m$ is larger now,} we may need to increase the value of $N$. But it is clear that after that, $\phi(Q_{j+1})\subseteq g(Q_j)$ for all $j\geqslant M$ and also \blue{$\phi(\tfrac{1}{2}D_{\pm n}^j)\supseteq \D(\phi(\pm ih_{n}^j),R_n')$ for all $n\geqslant N$ and $0\leqslant j<\blueb{\min\{n-2,q\}}$} (the values of $R_n'$ may be different from before). We need to check that \blue{$\phi(D_{\pm n}^j)\subseteq g(Q_n)$ for all $n\geqslant N$ and $0\leqslant j<\blueb{\min\{n-2,q\}}$}. This follows from the fact that by Lemma~\ref{lem:growth-estimates}, 
$$
\red{e^\varepsilon(h_{n}^{n-3}+3R_n)=\tfrac{3}{2}(h_n+6(n-3)R_n+3R_n)< 2x_{n+1}, \mbox{ for } n\geqslant \vfour{N},}
$$
where $\varepsilon=\log\tfrac{3}{2}$ as in the proof of Lemma~\ref{lem:phi(Q)-in-g(Q)}. It follows that, for \blueb{$n\geqslant N$ and $0\leqslant j< \min\{n-2,q\}$}, we can define a \blueb{domain $U_n^j$ in the same way that we defined the domain $U_n$ in Definition~\ref{dfn:domains-Un} but replacing the disc $\D(\phi(ih_n),R_n')$ by $\D(\phi(ih_{n}^j),R_n')$. Similarly, we can define the centers $c_n^j(\mathbf{w})\in U_n^j$ for $n\geqslant N$ and $0\leqslant j< \min\{n-2,q\}$.} \vfour{The estimates from Lemma~\ref{lem:inner-radius} show that $\D(c_n^j(\mathbf{w}),\rho_n)\subseteq U_n^j$ for all $n\geqslant N$ and $0\leqslant j< \min\{n-2,q\}$. Finally, we can solve the simultaneous shooting problem
\begin{equation}
w_n^j=c_{n+1}^j(\mathbf{w}),\quad \text{ for } 0\leqslant j<q \text{ and } n\geqslant N_j:=\max\{N,3+j\},
\end{equation}
as in Lemma~\ref{lem:shooting}. Hence, there exists $N_1\geqslant N$ as in Lemma~\ref{lem:inclusion} such that
\begin{equation}
\tilde{f}_\mathbf{w}^{n+2}(U_n^j)\subseteq U_{n+1}^j,\quad \text{ for } 0\leqslant j<q \text{ and } n\geqslant N_j':=\max\{N_1,3+j\}.
\end{equation}
To show that the sets $U_n^j$ are contained in different Fatou components, we use the fact that $\tilde{f}_\mathbf{w}^{n+1}(U_n^j)\subseteq \tilde{\phi}_\mathbf{w}(\tfrac{1}{2}D_n^j)$ for all $0\leqslant j<q$ and $n\geqslant N_j'$ and the sets \red{$\{\tilde{\phi}_\mathbf{w}(D_n^j)\}_{n,j}$} are separated by curves of the form $L_k=\tilde{\phi}_\mathbf{w}(\{z\in\C\ :\ \text{Im}\,z=k\pi\})$ for $k\in\Z$ that lie in the Julia set as in Lemma~\ref{lem:separation}. Therefore the sets $\{U_n^j\}_{n,j}$ are contained in wandering domains, and for different values of $j$, these wandering domains belong to different grand orbits.

Up to now, we have shown that the entire function $\tilde{f}_\mathbf{w}\in\mathcal B$ has at least \red{$2q$} grand orbits of wandering domains. To conclude that $\tilde{f}_\mathbf{w}$ has exactly \red{$2q$} grand orbits of wandering domains, we will follow closely the results of \cite[Section~4]{fagella-godillon-jarque15}. 

Since $w_\infty$ is the only accumulation point of the singular values of $\tilde{f}_\mathbf{w}$ (see Remark~\ref{rmk:w-infty}), for every wandering domain $U$ of $\tilde{f}_\mathbf{w}$, a classical argument using normal families implies that there is a sequence $(n_k)_k$ such that $\tilde{f}_\mathbf{w}^{n_k}\vert_U\to w_\infty$ as $k\to\infty$ and \red{$\tilde{f}_\mathbf{w}^{n_k-1}(U)\subseteq \bigcup_{n,j}\tilde{\phi}_\mathbf{w}(D_{\pm n}^j)$} for all $k\in\N$ (see \red{\cite[Lemma~4.1]{fagella-godillon-jarque15}}). Then, it follows from \cite[Theorem~4.1]{mihaljevicbrandt-rempegillen13} that the forward invariant closed set \vspace{-5pt}
\begin{equation}
A:=\red{\R}\cup \bigcup_{j=0}^{q-1}\bigcup_{n= N_j'}^\infty \bigcup_{k=0}^{n+1} \overline{\tilde{f}_\mathbf{w}^k\left(U_n^j\cup \red{\overline{U_n^j}}\right)}\vspace{-5pt}
\end{equation}
containing all the singular values of $\tilde{f}_\mathbf{w}$ (the points $\pm 2$, the sequences $\{w_n^j\}_{n,j}$ and the limit point $w_\infty$) satisfies that if $\tilde{f}_\mathbf{w}^n(U)\cap A=\emptyset$ for all $n\geqslant 0$, then 
\begin{equation}
\label{eq:helena}
\text{dist}_V(\tilde{f}_\mathbf{w}^{n_k-1}(U),V\setminus \tilde{f}_\mathbf{w}^{-1}(\D(\tfrac{1}{2},\tfrac{1}{8})))\to +\infty\quad \text{ as } k\to \infty,
\end{equation}
where $V=\C\setminus A$ (see \cite[Lemma~4.2]{fagella-godillon-jarque15}). Finally, suppose to the contrary that $\tilde{f}_\mathbf{w}^k(U)\cap \red{(U_n^j\cup\overline{U_n^j})}=\emptyset $ for all $n,j,k$, then $\tilde{f}_\mathbf{w}^k(U)\cap A=\emptyset$ and since $\tilde{f}_\mathbf{w}^{n_k-1}(U)\subseteq \red{\tilde{\phi}_\mathbf{w}(D_{\red{\pm n}}^j)}$ for some $n,j$, the \red{domain $\tilde{\phi}_\mathbf{w}(D_{\red{\pm n}}^j)$} must contain some of the connected components of $A$, and it follows from \eqref{eq:helena} that $\tilde{f}_\mathbf{w}^{n_k-1}(U)\cap \tilde{\phi}_\mathbf{w}(\tfrac{1}{2}D_{\red{\pm n}}^{j})\neq\emptyset$ for some $n,j$, which is a contradiction (see \cite[Lemma~4.3]{fagella-godillon-jarque15}). This finishes the proof of Theorem~\ref{thm:infinitely-many}.}
\end{proof}

\begin{rmk}
\label{rmk:odd-number-wd}
\vfour{If one wished to construct entire functions in the class $\mathcal B$ of finite order with an odd number $q$ of grand orbits of wandering domains, then one may consider $q$ different reference orbits near the orbit of $x_0=\tfrac{1}{2}$ and for each of them define a sequence of rectangles $Q_n$, and carry on the same procedure as in the proof of Theorem~\ref{thm:main}.}
\end{rmk}

%\section{Boundedness of the wandering domains as sets}

%\label{sec:bounded-wd}

%Finally, we show that all the wandering domains of functions from Theorem~\ref{thm:infinitely-many} are bounded sets. For that we use the existence of horizontal lines that are in the Julia set and obtain a contradiction with the expansivity of the function a similar way that was done in \cite[\red{Theorem~XX}]{lazebnik17}.

%\begin{proof}[Proof of Theorem~\ref{thm:bounded-wd}]
%x

%\begin{figure}[h!]
%\def\svgwidth{.8\linewidth}
%\input{separation.pdf_tex}
%%\vspace{5pt}
%\caption{Sketch of how an unbounded wandering domain $W$ may look in the $\phi_\mathbf{w}^{-1}$ coordinate if it existed, as considered in the proof of Theorem~\ref{thm:bounded-wd}.}
%\label{fig:setup}
%\end{figure}

%x
%\end{proof}

\appendix

\section{Detailed construction of the interpolation}

\label{sec:a1}
\label{sec:construction-of-interpolation}

Here we provide the details of the construction of the quasiregular interpolation that we described in Section~\ref{sec:qr-interpolation}. Our main tool will be the following result on quasiregular interpolation.

\begin{thm}[Linear Interpolation Theorem]
%\label{lem:linear-interpolation}
\label{thm:linear-interpolation}
Let \mbox{$\gamma_j:[0,t_0]\to\C$}, $j\in\{1,2\}$, be $\mathcal C^1$ curves for some $t_0>0$. Suppose that $\gamma_1(t)\neq \gamma_2(t)$ and $\gamma_j'(t)\neq 0$ for all $t\in[0,t_0]$ and $j\in\{1,2\}$. Furthermore, suppose that there exist $s_0>0$ and $r>0$ such that
\begin{equation}
\frac{s_0\gamma_j'(t)}{\gamma_2(t)-\gamma_1(t)}\in \D_{\mathbb H}\left(i,r\right), \quad \text{ for all } t\in [0,t_0] \text{ and } j\in\{1,2\}.
\end{equation}
Then, the function $\Phi:[0,s_0]\times [0,t_0]\to \C$ defined by 
\begin{equation}
\Phi(s,t):= \left(1-\tfrac{s}{s_0}\right)\gamma_1(t)+\tfrac{s}{s_0}\gamma_2(t)
\end{equation}
is locally quasiconformal with Beltrami coefficient $\mu_\phi$ satisfying that
%is a $K$-quasiconformal map such that 
\begin{equation}
|\mu_\Phi(s,t)|\leqslant \tanh\frac{r}{2}, \quad  \text{ for all } (s,t)\in [0,s_0]\times [0,t_0],
\end{equation}
and hence $K_\Phi\leqslant e^r$. If, in addition, the segments $[\gamma_1(t),\gamma_2(t)]$ and $[\gamma_1(t'),\gamma_2(t')]$ are disjoint for $0\leqslant t<t'\leqslant t_0$, then the map $\Phi$ is $e^r$-quasiconformal onto its image.

In particular, if there exist constants $0<\theta<\tfrac{\pi}{2}$ and $v_\textup{min}, v_\textup{max},l_\textup{min},l_\textup{max}>0$ such that
\begin{equation}
\begin{array}{c}
0<\dfrac{\pi}{2}-\theta\leqslant\displaystyle \textup{arg}\left(\frac{\gamma_j'(t)}{\gamma_2(t)-\gamma_1(t)}  \right)\leqslant \frac{\pi}{2}+\theta<\pi,\vspace{10pt}\\
v_\textup{min}\leqslant |\gamma_j'(t)|\leqslant v_\textup{max}\quad \textup{and}\quad l_\textup{min}\leqslant |\gamma_2(t)-\gamma_1(t)|\leqslant l_\textup{max},
\end{array}
\end{equation}
for all $t\in[0,t_0]$ and $j\in\{1,2\}$, then it suffices to choose
$$
s_0=\sqrt{\frac{l_\textup{min}l_\textup{max}}{v_\textup{min}v_\textup{max}}} \quad \text{and} \quad r=\textup{dist}_{\mathbb H}\left(s_0\frac{v_{\text{max}}}{l_\text{min}}e^{i\left(\tfrac{\pi}{2}+\theta\right)},i\right)=2\,\textup{artanh}\,\left|\frac{s_0\frac{v_{\text{max}}}{l_\text{min}}e^{i\theta}-1}{s_0\frac{v_{\text{max}}}{l_\text{min}}e^{i\theta}+1}\right|.
$$
\end{thm}
\begin{proof}
Identifying the point $(s,t)$ with $z=s+it$, we have
$$
\mu_\Phi(s,t)=\frac{\partial_{\overline{z}}\Phi(s,t)}{\partial_{z}\Phi(s,t)}=-\frac{\frac{\partial_t\Phi(s,t)}{\partial_s\Phi(s,t)}-i}{\frac{\partial_t\Phi(s,t)}{\partial_s\Phi(s,t)}+i}=-M\left(\frac{\partial_t\Phi(s,t)}{\partial_s\Phi(s,t)}\right),
$$
where $M(z)=\frac{z-i}{z+i}$ (see Figure~\ref{fig:linear-interpolation}). The partial derivatives of $\Phi$ are
$$
\begin{array}{c}
\partial_s\Phi(s,t)=-\frac{1}{s_0}\gamma_1(t)+\frac{1}{s_0}\gamma_2(t)\quad \text{and}\quad
\partial_t\Phi(s,t)=\left(1-\frac{s}{s_0}\right)\gamma_1'(t)+\frac{s}{s_0}\gamma_2'(t),
\end{array}
$$
and therefore,
$$
\frac{\partial_t\Phi(s,t)}{\partial_s\Phi(s,t)}=s_0\left(\left(1-\frac{s}{s_0}\right)\frac{\gamma_1'(t)}{\gamma_2(t)-\gamma_1(t)}+\frac{s}{s_0}\frac{\gamma_2'(t)}{\gamma_2(t)-\gamma_1(t)}\right)\in \D_{\mathbb H}(i,r).
$$
Finally, the image of $\D_{\mathbb H}(i,r)$ by $M$ is the Euclidean disc $\D_{\tanh(r/2)}$ and hence $|\mu_\Phi(s,t)|\leqslant \tanh\tfrac{r}{2}$ for all $(s,t)\in [0,s_0]\times [0,t_0]$ as required.
\end{proof}

\begin{figure}
\vspace{30pt}
\def\svgwidth{.8\linewidth}
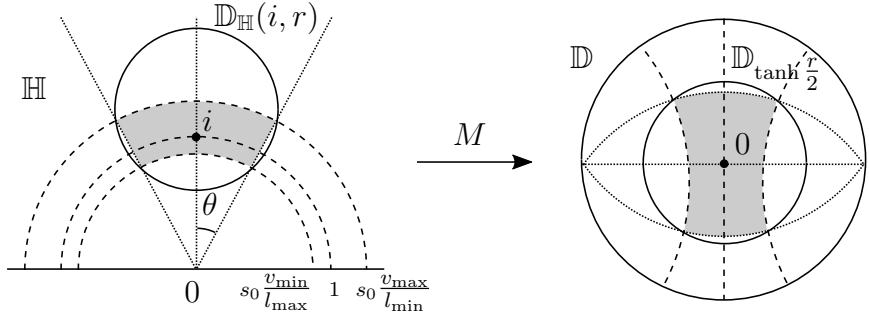
\caption{The M\"obius transformation $M(z)=\frac{z-i}{z+i}$ from Theorem~\ref{thm:linear-interpolation}.}
\label{fig:linear-interpolation}
\end{figure}

We now use linear interpolation to construct a quasiregular map $G$ that interpolates between $g(z)=2\cosh z$ and a $d$-th power map in a doubly connected region. The dilatation $K_G$ should be bounded by a constant independent of $d\in \N$.

\begin{proof}[Proof of Lemma~\ref{lem:interpolation}]
\label{proof:lem:interpolation}
Consider the set
$$
\Omega:=(E\setminus \D_R)\cap \{z\in\C\, :\, \textup{Re}\,z\geqslant 0,\ \textup{Im}\,z\geqslant 0\}.
$$
Since $\cosh(-z)=\cosh(z)$ and $\cosh(\overline{z})=\overline{\cosh(z)}$ for all $z\in \C$ and the map $G$ is an even power on $\partial \D_R$, we have that  $G(-z)=G(z)$ and $G(\overline{z})=\overline{G(z)}$ for all $z\in \partial E\cup  ((E\cap i\mathbb{R})\setminus \D_R)\cup \partial \D_R$. Thus, we only need to do the interpolation in $\Omega$ and then extend it by these relations to $(E\setminus \D_R)\setminus \Omega$.

Let $Q:=\{z\in\C\, :\, 0\leqslant \textup{Re}\,z\leqslant 2d\pi,\ 0\leqslant \textup{Im}\,z\leqslant 2d\pi\}$. First we construct a quasiconformal map $\Phi_1:Q\to \Omega$. For example, we can partition both $Q$ and $\Omega$ into four sets $Q_i$ and $\Omega_i$, $i\in\{1,2,3,4\}$, as shown in Figure~\ref{fig:interpolation-step1} and apply the Linear Interpolation Theorem~\ref{thm:linear-interpolation} to each pair $Q_i,\Omega_i$, so that $\Phi_1(Q_i)=\Omega_i$ for all\linebreak 

\begin{figure}[h]
\vspace{10pt}
\def\svgwidth{.8\linewidth}
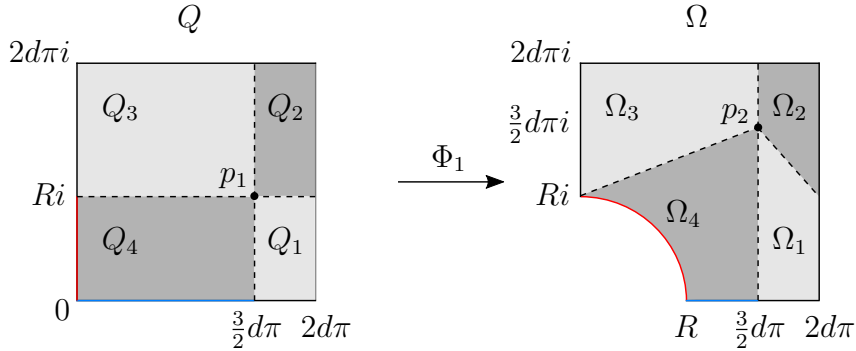
\caption{The quasiconformal map $\Phi_1:Q\to \Omega$ from Lemma~\ref{lem:interpolation}.}
\label{fig:interpolation-step1}
\end{figure}

\noindent
$i\in\{1,2,3,4\}$, $\Phi_1\equiv \text{id}$ on $\partial Q_i\cap \partial Q$ for all $i\in\{1,2,3\}$, $\Phi_1(z)=R\exp\left(i\frac{\pi}{2R}\text{Im}\,z\right)$ on $[0,iR]$ and $\Phi_1(p_1)=p_2$, where $p_1:=\tfrac{3}{2}d\pi+iR\in Q$ and $p_2:=\tfrac{3}{2}d\pi+i\tfrac{3}{2}R\in \Omega$. It is easy to check that the dilatation of $\Phi_1$ is uniformly bounded independently of $d$.

The rest of the proof will be devoted to construct a quasiregular function\linebreak \mbox{$\check{G} : Q\to \mathcal{E}_{2d\pi}$} so that defining $G=\check{G}\circ \Phi_1^{-1}$ on $\Omega$ (and extending the definition of~$G$ to the rest of $E\setminus \D_R$ as described before), then $G$ has the required proper\-ties (see Figure~\ref{fig:interpolation-step4}). Note that $\check{G}$ has to map the vertical segment $[0,iR]\subseteq \partial Q$ to the arc of circle $\{z=e^{i\theta}\, :\, \theta\in [0,d\pi]\}\subseteq \partial \mathbb{D}$ and $\check{G}(iR)=(-1)^{d}$. 

\begin{figure}[h]
%\vspace{10pt}
\def\svgwidth{.9\linewidth}
\input{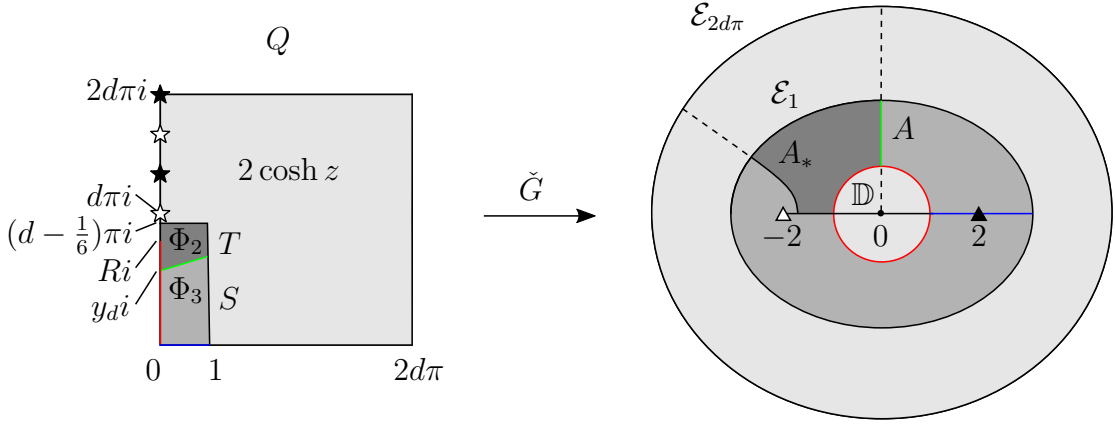}
\caption{The quasiregular map $\check{G}:Q\to \mathcal{E}_{2d\pi}$ from Lemma~\ref{lem:interpolation} with $d$ odd (here we used $d=3$).}
\label{fig:interpolation-step4}
\end{figure}

Define
$$
y_d:=\tfrac{2d-1}{2d}R=\tfrac{R}{d\pi}(d-\tfrac{1}{2})\pi=\left(d-\tfrac{5}{6}+\tfrac{1}{6d}\right)\pi,
$$
then, we have
$$
(d-\tfrac{5}{6})\pi<y_d<(d-\tfrac{1}{2})\pi<R.
$$
Let $T$ be the trapezoid of vertices $(d-\tfrac{1}{6})\pi i, 1+(d-\tfrac{1}{6})\pi i, 1+(d-\tfrac{1}{2})\pi i$ and $y_d i$ (see Figure~\ref{fig:interpolation-step2}). Observe that if we consider the translate $\tilde{T}:=T-(d-1)\pi i$ of $T$, then it becomes clear that the shape of~$T$ is bounded in the sense that the only vertex of $\tilde{T}$ that depends on $d$ is $\tilde{y}_di:=y_di-(d-1)\pi i$ and satisfies 
$$
\tfrac{\pi}{6}<\tilde{y}_d<\tfrac{\pi}{2}.
$$
The point $Ri\in \partial T$ corresponds to $\tfrac{2}{3}\pi i\in\partial \tilde{T}$. Let $A=\mathcal{E}_1\setminus \mathbb D$ and define the region $H:=g(\{z\in\C\, :\, \text{Re}\,z\geqslant 0,\ \tfrac{5}{6}\pi< \text{Im}\,z< \tfrac{7}{6}\pi\})$ that is bounded by a hyperbola. Then, define the set $A_*:=(A\cap  \overline{\mathbb H}\cap \overline{\mathbb H_l})\setminus H$ if $d$ is odd, and let~$A_*$ be the symmetric set with respect to the origin if $d$ is even. We construct a quasiconformal map $\Phi_2:T\to A_*$ such that 
$$
\!\Phi_2(z)\!=\!\left\{\!
\begin{array}{ll}
\!\!2\cosh z,&\! \text{if } z\in \left[R i,\left(d-\tfrac{1}{6}\right)\pi i\right]\cup\vspace{5pt}\\
\!\!&  \! \left[\left(d-\tfrac{1}{6}\right)\pi i,1+\left(d-\tfrac{1}{6}\right)\pi i\right]\cup\vspace{5pt}\\
\!\!& \! \left[1+\tfrac{1}{2}\pi i,1+\tfrac{5}{6}\pi i\right],\vspace{5pt}\\
\!\!\exp\left(i\tfrac{d\pi}{R}\text{Im}\,z\right), &\! \text{if } z\in \left[y_d i,R i\right],\vspace{5pt}\\
\!\!(1-\text{Re}\,z)(-1)^{d-1}i+(\text{Re}\,z)(-1)^{d-1}(e-e^{-1})i, & \!\text{if } z\in \left[y_d i,1+\left(d-\tfrac{1}{2}\right)\pi i\right]\!.
\end{array}
\right.
$$
%To that end, let $T_1\subseteq T$ be a rectangle of height $\pi/3$ and width $\pi/6$ so that after joining each of the vertices of~$T_1$ to the closest vertex of $T$ and also joining $Ri$ to the middle point of the left side of $T_1$ we obtain a partition of $T\setminus T_1$ into quadrilaterals $T_j$, $2\leqslant j\leqslant 6$, with angles at their vertices that are bounded away from $0$ and $\pi$ as shown in Figure~\ref{fig:interpolation-step2}. 
We can partition $T$ and $A$ into quadrilaterals as in Figure~\ref{fig:interpolation-step2} and apply again the Linear Interpolation Theorem~\ref{thm:linear-interpolation} to obtain quasiconformal maps $\Phi_{2,1,i}:R_i\to T_i$ and $\Phi_{2,2,i}:R_i\to A_i$ for $1\leqslant i\leqslant 6$. Then, it suffices to put $\Phi_2(z):=\Phi_{2,2,i}\circ \Phi_{2,1,i}^{-1}(z)$ if $z\in T_i$. This map is well defined since the maps match on the shared boundaries of $T_i$ because they come from linear interpolation.

\begin{figure}
\vspace{30pt}
\def\svgwidth{.8\linewidth}
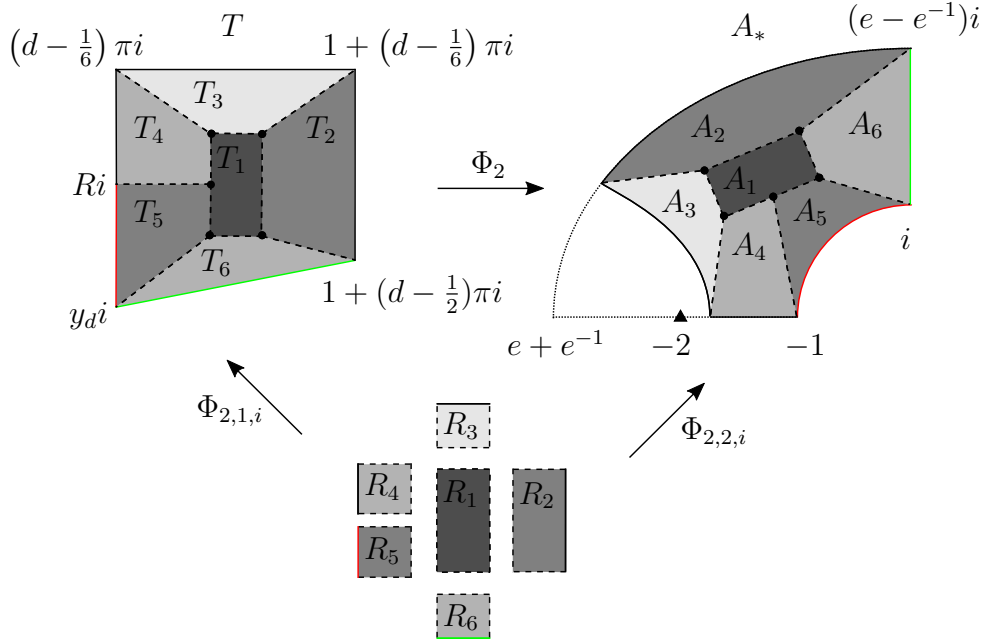
\caption{The quasiconformal map $\Phi_2:T\to A_*$ from Lemma~\ref{lem:interpolation} with $d$ odd.}
\label{fig:interpolation-step2}
\end{figure}

There exists a quasiconformal map $\Phi_{3,1}$ from the rectangle $[0,1]+i[0,(d-\tfrac{1}{2})\pi i]$ onto the trapezoid $S$ with vertices $0$, $1$, $1+(d-\tfrac{1}{2})\pi i$ and $y_d i$ such that $\Phi_{3,1}$ is the identity on the right and bottom edges and is given by linear interpolation on the other two edges (see Figure~\ref{fig:interpolation-step3}). Define a quasiregular map $\Phi_{3,2}$ from the same rectangle $[0,1]+i [0,(d-\tfrac{1}{2})\pi i]$ to the annulus $A$ by
$$
\Phi_{3,2}(s+it):=(1-s)e^{it}+s2\cosh(1+it)=(1-s)e^{it}+s(e^{1+it}+e^{-1-it}).
$$
Observe that $|\text{arg}(2\cosh(1+it))-t|<\arcsin(1/e^2)<\arccos(1/e)$ and therefore for every $t\in [0,(d-\tfrac{1}{2})\pi]$, the line given by $\Phi_{3,2}([0,1]+it)$ is disjoint from $\D$. Thus, the Linear Interpolation Theorem~\ref{thm:linear-interpolation} implies that $\Phi_{3,2}$ is a quasiconformal map with distortion bounded independently of $d$. Then, we put $\Phi_3:=\Phi_{3,2}\circ \Phi_{3,1}^{-1}$.

\begin{figure}[h]
\vspace{30pt}
\def\svgwidth{.75\linewidth}
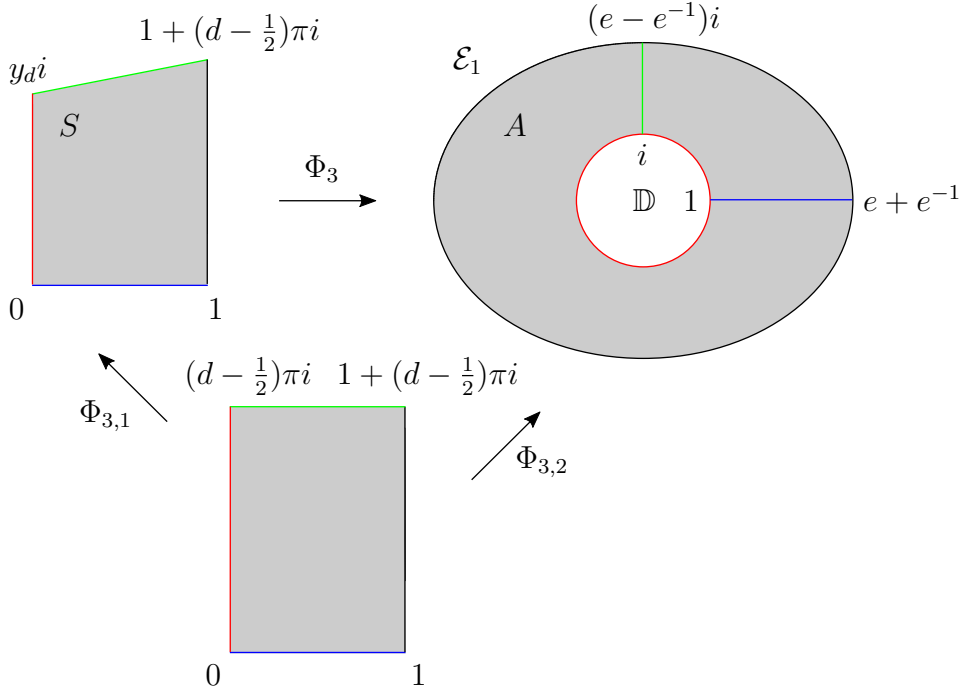
\caption{The quasiregular map $\Phi_3:S\to A$ from Lemma~\ref{lem:interpolation} with $d$ odd.}
\label{fig:interpolation-step3}
\end{figure}

Now we just need to put all the previous maps together to define the quasiregular map $\check{G}:Q\to\mathcal{E}_{2d\pi}$: for $z\in Q$, set
$$
\check{G}(z):=\left\{
\begin{array}{ll}
\Phi_2(z),& \text{ if } z\in T,\vspace{5pt}\\
\Phi_3(z),& \text{ if } z\in S,\vspace{5pt}\\
2\cosh z,& \text{ if } z\in Q\setminus (T\cup S),
\end{array}\right.
$$
and then put $G:=\check{G}\circ \Phi_1^{-1}$ on $\Omega=(E\setminus \D_R)\cap \overline{\mathbb H}\cap \overline{\mathbb H_r}$ (see~Figure~\ref{fig:interpolation-step4}). From the construction, it is clear that the different maps that define $\check{G}$ coincide along the common boundaries of the regions where they are defined. 

Let us check the values of $\tilde{G}$ on the left side of $Q$. On the one hand, we have $\tilde{G}(yi)=2\cosh(yi)$ for $R\leqslant y\leqslant 2d\pi$. On the other hand, for $y_d\leqslant y\leqslant R$, $\check{G}(yi)=\exp\left(i\tfrac{d\pi}{R}y\right)$ and, for $0\leqslant y\leqslant y_d$, 
$$
\check{G}(yi)=\Phi_3(yi)=\Phi_{3,2}\bigl(\Phi_{3,1}^{-1}(yi)\bigr)=\Phi_{3,2}\left(y\frac{\left(d-\tfrac{1}{2}\right)\pi}{y_d}i\right)=\exp\left(i\frac{d\pi y}{R}\right).
$$
Thus, $\check{G}(yi)=\exp(i\tfrac{d\pi y}{R})$ for all $0\leqslant y\leqslant R$. Finally, recall that the two maps match at the point $z=Ri$ as
$$
2\cosh(Ri)=2\cos\left(\left(d-\tfrac{1}{3}\right)\pi\right)=(-1)^d=\exp\left(d\pi i\right)=\exp\left(i\tfrac{d\pi}{R}R\right).
$$
If $z=Re^{i\theta}\in\partial \D_R\cap Q$ with $0\leqslant \theta\leqslant \tfrac{\pi}{2}$ and $yi=\Phi^{-1}(z)$, then $\theta=\tfrac{\pi}{2R}y$. Therefore $0\leqslant y\leqslant R$ and $\check{G}(yi)=\exp\left(i\tfrac{d\pi}{R}y\right)$, so $G(z)=\exp\left(2d\theta i\right)=(z/R)^{2d}$ as required.

Since the dilatation of every function $\Phi_j$ involved is bounded by a uniform constant which is independent of $d$, the dilatation $K_G$ is bounded independently of $d$. This completes the proof of Lemma~\ref{lem:interpolation}.
%We split $\Omega$ into three closed sets $\Omega_+,\Omega_0,\Omega_-$ with pairwise disjoint interior such that $\Omega=\Omega_+\cup\Omega_0\cup\Omega_-$ and we do the interpolation separately in each part (see Figure \ref{fig:interpolation}). Define the sets
%$$
%\begin{array}{c}
%\Omega_+:=\{z\in \Omega\ :\ m\textup{Re}\,z+b\leqslant \textup{Im}\,z\},\vspace{5pt}\\
%\Omega_0:=\{z\in \Omega\ :\ -m\textup{Re}\,z-b\leqslant\textup{Im}\,z\leqslant m\textup{Re}\,z+b\},\vspace{5pt}\\
%\Omega_-:=\{z\in \Omega\ :\ \textup{Im}\,z\leqslant -m\textup{Re}\,z-b\}
%\end{array}
%$$
\end{proof}

In Lemma~\ref{lem:shift-interpolation}, we introduced the map $\rho_w:\overline{\D}\to\overline{\D}$ for $w\in \D_{3/4}$. This map is the identity on $\partial \D$ and $\rho_w(z)=z+w$ on $\overline{\D}_{1/8}$. We have to show that $\rho_w$ is $K_2$-quasiconformal for some constant $K_2\geqslant 1$ that does not depend on the choice of $w\in \D_{3/4}$.

\begin{proof}[Proof of Lemma~\ref{lem:shift-interpolation}]
\label{proof:lem:shift-interpolation}
First, we use the principal branch of the logarithm to map the annulus $\overline{\D}\setminus \D_{1/8}$ to the rectangle 
$$
\{z\in \C\, :\, -\log 8\leqslant \textup{Re}\,z\leqslant 0,\ -\pi\leqslant\textup{Im}\,z\leqslant \pi\}.
$$
Then, we subdivide this rectangle into a certain number of smaller rectangles and apply Theorem~\ref{thm:linear-interpolation} to each of them (see Figure~\ref{fig:interpolation-rho}). Observe that we can arrange them so that the angles of the image partition has angles bounded away from~$0$ and~$\pi$ at the corners, so we deduce that the dilatation $K_{\rho_w}$ is bounded by a constant $K_2\geqslant 1$ for all $w\in \D_{3/4}$.
\end{proof}

\begin{figure}[h]
\def\svgwidth{.8\linewidth}
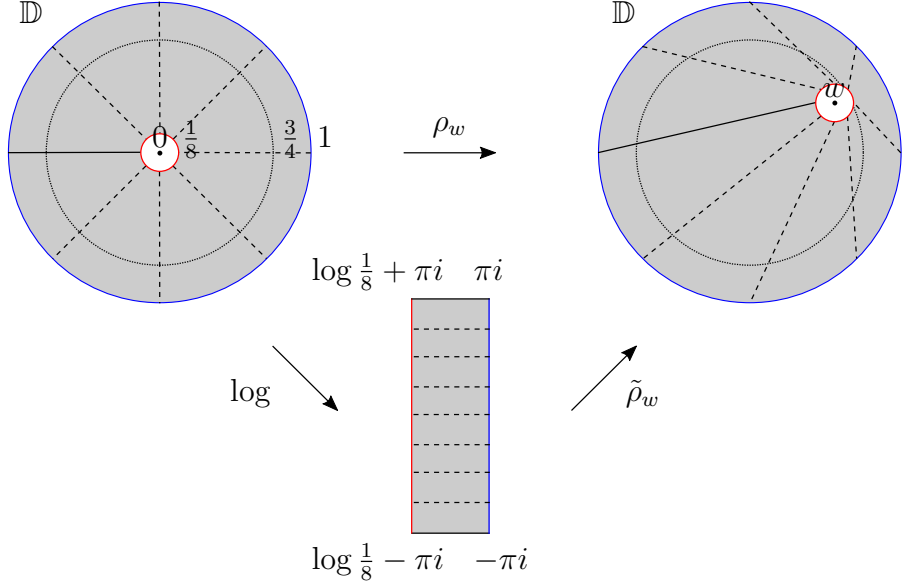
\vspace{15pt}
\caption{The quasiconformal map $\rho_w:\overline{\D}\to \overline{\D}$ from Lemma~\ref{lem:shift-interpolation}.}
\label{fig:interpolation-rho}
\end{figure}

\section{Proofs of the estimates on quasiconformal maps}
\label{sec:proof-of-estimates}

In this appendix we prove the three results that we stated in Section \ref{sec:qc-estimate} and also Lemma~\ref{lem:holom-dependence-of-centers}. Recall that throughout the section we had the standing assumption that $K>1$, the discs $B_m := \D(\zeta_m, r_m)$, $m\in\N$, satisfy that $|\zeta_m| \ge 4$ and $r_m/|\zeta_m| \leqslant \min\{\frac{1}{4},\delta_1\}$ for all $m\in\N$, $\sum_{m=1}^\infty r_m/|\zeta_m| < +\infty$ and there is a $K$-quasiconformal map $\phi:\C \to \C$ such that $\phi(0)=0$, $\phi(1)=1$ and $\supp \mu_\phi \subseteq  \bigcup_{m =1}^\infty B_m$  (see Assumption~\ref{ass:qu-estimates}). 

The following lemma will be used later in the proofs of Lemma~\ref{lem:qc-estimate1} and Lemma~\ref{lem:qc-estimate2}.

\begin{lem} 
\label{lem:simple-pole} 
For any $\alpha, \beta \in \C$ and $r>0$, 
\begin{equation}
\iint_{\D(\alpha, r)} \frac{dx \tinyhsp dy}{|z-\beta|} \le 2\pi r.
\end{equation}
\end{lem}
\begin{proof} We may assume that $\beta=0$.  
Since 
$$
\iint_{\D(\alpha, r) \sminus \D_r} \frac{dx \tinyhsp dy}{|z|}\leqslant\frac{1}{r} |\D(\alpha, r) \sminus \D_r|=\frac{1}{r} |\D_r \sminus \D(\alpha, r)| \leqslant
\iint_{\D_r \sminus \D(\alpha, r)} \frac{dx \tinyhsp dy}{|z|} ,
$$ 
we have 
$$
 \iint_{\D(\alpha, r)}\frac{dx \tinyhsp dy}{|z|}\leqslant  \iint_{\D_r}\frac{dx \tinyhsp dy}{|z|} =2 \pi r, 
$$
as required.  
\end{proof}

%\begin{lem} 
%\label{lem:two-pole} 
%For any $\alpha\in \D^*$, 
%$$
%\iint_{\D} \frac{dx \tinyhsp dy}{|z(z-\alpha)|} < 4\pi\left(1+\frac{2}{|\alpha|^2}\right).
%$$
%\end{lem}
%\begin{proof} Let $r=|\alpha|/2$. We split the disc $\D$ into three domains:
%$$
%\iint_{\D} \frac{dx \tinyhsp dy}{|z(z-\alpha)|} = \iint_{\D_r} \frac{dx \tinyhsp dy}{|z(z-\alpha)|} +\iint_{\D(\alpha,r)} \frac{dx \tinyhsp dy}{|z(z-\alpha)|} +\iint_{\D\setminus(\D_r\cup\D(\alpha,r))} \frac{dx \tinyhsp dy}{|z(z-\alpha)|}.
%$$
%Using that $|z|>r$ on $\D\setminus \D_r$ and $|z-\alpha|>r$ in $\D\setminus \D(\alpha,r)$, we obtain that 
%$$
%\iint_{\D} \frac{dx \tinyhsp dy}{|z(z-\alpha)|} <2\pi+2\pi+\frac{2\pi}{r^2}
%$$
%as we wanted to show.
%\end{proof}

Now we are ready to prove Lemma~\ref{lem:qc-estimate1}, which states that for every $\eps>0$, there exists $M_1=M_1(\eps)\in\N$ such that if $\supp \mu_\phi \subseteq \bigcup_{m = M_1}^{\infty} B_m$, then 
\[
\left| \log \frac{\redf{\phi}(\zeta)}{\zeta} \right|_{\calC} < \eps,\quad \text{ for all } \zeta \in \C \sminus \{0\}.
\]

\begin{proof} [Proof of Lemma \ref{lem:qc-estimate1}] 
Let $\eps>0$ and suppose that $\supp \mu_\phi \subseteq \bigcup_{m=M_1}^{\infty} B_m$. Setting $\alpha =0$, if $0<|\gamma|\leqslant \delta_1 |\beta|$, then Corollary \ref{cor:KeyInequality} gives the inequality
\begin{equation} \label{eq:keyineq-for-estimate1}
\left| \log \frac{\phi(\beta)}{\beta} - \log \frac{\phi(\gamma)}{\gamma} \right|_{\calC} 
\le  C(K -1) \sum_{m=M_1}^{\infty} \iint_{B_m} \frac{|\beta| \tinyhsp dx\tinyhsp dy}{| z(z-\beta)(z-\gamma)|}.   
\end{equation}
Assume $|\gamma| \le \frac{1}{\delta_1}$ and $|\zeta_m| \ge \frac{4}{\delta_1}$, we will give an estimate on the integral on the right hand side.   
Note that if $z \in B_m$, then $|z| \ge |\zeta_m| - r_m = |\zeta_m|(1- r_m/|\zeta_m|) \ge \frac{3}{4}|\zeta_m|$, 
and $|z - \gamma| \ge |\zeta_m|(1- r_m/|\zeta_m| - |\gamma|/|\zeta_m|) \ge \frac{1}{2}|\zeta_m|$.  
Let $H \ge 6$ and $\eta \le \frac{1}{4}$ and assume $r_m/|\zeta_m| \le \eta$.  
Now we split the estimate into two cases.  \vspace{5pt}
\begin{itemize}
\item {Case 1:} $|\beta - \zeta_m| \le \frac{1+H\eta}{H-1} |\zeta_m|$ \vspace{5pt}\\
Note that this includes the case $\beta \in B_m$, since $r_m \le \eta |\zeta_m| \le \frac{1+H\eta}{H-1} |\zeta_m|$.  We~have
$$
\hspace{42pt}\iint_{B_m} \frac{|\beta| \tinyhsp dx\tinyhsp dy}{| z(z-\beta)(z-\gamma)|} 
\le \frac{(1+\frac{1+H\eta}{H-1})|\zeta_m|}{\frac{3}{4}|\zeta_m| \cdot \frac{1}{2}|\zeta_m|} \iint_{B_m} \frac{\tinyhsp dx\tinyhsp dy}{|z-\beta|} 
\le \frac{8H(1+\eta)}{3(H-1)} \frac{2\pi r_m}{|\zeta_m|},
$$
where the last inequality follows from Lemma~\ref{lem:simple-pole}.\vspace{5pt} 
\item {Case 2:} $|\beta - \zeta_m| > \frac{1+H\eta}{H-1} |\zeta_m|$\vspace{5pt}\\
In this case, $H|\beta - \zeta_m| - |\beta - \zeta_m| \ge |\zeta_m| + H r_m$ and
$$
\hspace{42pt}|\beta| \le |\beta-\zeta_m|+|\zeta_m| \le H(|\beta-\zeta_m|-r_m) \le H(|\beta-\zeta_m|-|z - \zeta_m|) 
\le H|z - \beta|,
$$
for $z \in B_m$. Hence, $\frac{|\beta|}{|z - \beta|} \le H$ and
$$
\begin{array}{rl}
\ds\iint_{B_m} \frac{|\beta| \tinyhsp dx\tinyhsp dy}{| z(z-\beta)(z-\gamma)|} 
\hspace{-6pt} & \ds\leqslant H  \iint_{B_m} \frac{\tinyhsp dx\tinyhsp dy}{|z(z-\gamma)|} 
\leqslant \frac{2\pi H r_m^2}{\frac{3}{4}|\zeta_m| \cdot \frac{1}{2}|\zeta_m|} \vspace{10pt}\\
& \ds\leqslant \frac{16\pi H}{3} \frac{r_m^2}{|\zeta_m|^2} 
\leqslant \frac{16\pi H \eta }{3} \frac{r_m}{|\zeta_m|}.   
\end{array}
$$
\end{itemize}
By the assumption that $\sum_{m=1}^\infty r_m/|\zeta_m| < +\infty$, 
there exists $\redf{M_1=M_1(\eps)}\in \N$ such that the right hand side of \eqref{eq:keyineq-for-estimate1} is less than given $\eps/2$.  
By taking $\gamma=1$, if $|\beta| \ge \frac{1}{\delta_1}$, 
$$  
\vspace*{-2pt}  
\left| \log \frac{\phi(\beta)}{\beta} \right|_{\calC} <\frac{\eps}{2}.  
$$  
Now fix a $\beta$ with $|\beta| > \frac{1}{\delta_1^2}$, then for any $\gamma$ with $|\gamma| \le \frac{1}{\delta_1}$, we also have 
$$
\vspace*{-2pt}  
\left| \log \frac{\phi(\gamma)}{\gamma} \right|_{\calC} 
\le \left| \log \frac{\phi(\beta)}{\beta} - \log \frac{\phi(\gamma)}{\gamma}\right|_{\calC} + \left| \log \frac{\phi(\beta)}{\beta} \right|_{\calC} < \eps.
$$
This completes the proof of Lemma~\ref{lem:qc-estimate1}.
\end{proof}

Next we will prove Lemma \ref{lem:qc-estimate2}. On top of the Assumption~\ref{ass:qu-estimates}, we also suppose that there exists a constant $C_1>0$ such that 
if $z \in B_m$ and $z' \in B_{m'}$ with $m \neq m'$, 
then $|z-z'| \ge C_1 \sqrt{|z z'|}$. %Observe that this is the case if $|z-z'| \ge C_1 \max\{|z|, |z'|\}$ for all $z \in B_m$ and $z' \in B_{m'}$ with $m \neq m'$ as
%$\max\{x_1,x_2\}\geqslant \sqrt{x_1x_2}$ for all $x_1,x_2\in\R_+$.
Then, we have to prove that for any $0<\kappa \leqslant1$, there exists $C_2>1$ such that for any $m\in\N$,
if  $|\zeta - \zeta_m| = \kappa r_m$, then
$$
\frac{1}{C_2} \kappa r_m \le |\phi(\zeta) - \phi(\zeta_m)| \le C_2 \kappa r_m.
$$
We want to emphasize that $C_2$ depends on \blue{$\kappa$}.   

\begin{proof} [Proof of Lemma \ref{lem:qc-estimate2}] 
Fix $n\in\N$ and apply Corollary \ref{cor:KeyInequality} with $\alpha =\zeta_n$, $\beta=0$ and $\gamma=\zeta\in B_n$. Then, since  $0<|\zeta - \zeta_n| = \kappa r_n <\delta_1 |\zeta_n|$ by hypothesis, we obtain  
\begin{equation} \label{eq:keyineq-for-estimate2}
\left| \log \frac{\phi(\zeta) - \phi(\zeta_n)}{\zeta - \zeta_n } - \log \frac{\phi(\zeta_n)}{\zeta_n} \right|_{\calC} 
\le  C(K -1) \sum_{m=M_1}^{\infty} \iint_{B_m} \frac{|\zeta_n| \tinyhsp dx\tinyhsp dy}{| z (z-\zeta)(z-\zeta_n)|},
\end{equation}
for some \blue{$M_1\in\N$}. We provide a different estimate of this integral according to $m$.
\begin{itemize}

\item For $m=n$, since $|z|\geqslant |\zeta_n|-r_n\geqslant \tfrac{3}{4}|\zeta_n|$ for $z\in B_n$,
\begin{align*}
\hspace{38pt}\iint_{B_n}\! \frac{|\zeta_n| \tinyhsp dx\tinyhsp dy}{| z(z-\zeta)(z-\zeta_n)|} 
\le \frac{|\zeta_n|}{\frac{3}{4} |\zeta_n|} \iint_{B_n}\! \frac{\tinyhsp dx\tinyhsp dy}{|(z-\zeta)(z-\zeta_n)|}  
= \frac{4}{3} \iint_{\D}\! \frac{\tinyhsp dx'\tinyhsp dy'}{|z'(z'-\kappa)|},  
\end{align*}
where we represent $\zeta = \zeta_m + \kappa r_m e^{i\theta}$ and then set $z = \zeta_m + r_m e^{i\theta}z'$ and 
$z'=x'+iy'$. %By Lemma~\ref{lem:two-pole}, the last integral is bounded by a quantity that only depends on $0<\kappa\leqslant 1$.
\blue{Note that the last integral is finite.}

\item For $m \neq n$, since $z\in B_m$ and $\zeta,\zeta_n\in B_n$,  
$$
|(z-\zeta)(z-\zeta_n)| \ge C_1^2 |z|\sqrt{|\zeta| |\zeta_n|}\geqslant \left(\tfrac{3}{4}\right)^{\frac{3}{2}}C_1^2|\zeta_m| |\zeta_n|
$$ 
as $|z|\geqslant \tfrac{3}{4}|\zeta_m|$ and $|\zeta|\geqslant \tfrac{3}{4}|\zeta_n|$, and, by Lemma~\ref{lem:simple-pole}, we obtain 
\begin{align*}
\iint_{B_m} \frac{|\zeta_n| \tinyhsp dx\tinyhsp dy}{| z(z-\zeta)(z-\zeta_n)|} 
\le \iint_{B_m} \frac{\tinyhsp dx\tinyhsp dy}{|z|\left(\frac{3}{4}\right)^{\frac{3}{2}} C_1^2 |\zeta_m| }  
= \frac{2\pi}{\left(\tfrac{3}{4}\right)^{\frac{3}{2}}C_1^2} \frac{r_m}{|\zeta_m|}.
\end{align*}
\end{itemize}
Therefore the right hand side of \eqref{eq:keyineq-for-estimate2} is convergent as $\sum_{m=1}^\infty r_m/|\zeta_m| < +\infty$, 
and bounded by a constant $\widetilde{C_2}>0$ which is independent of $n\in\N$. Together with the estimate from Lemma~\ref{lem:qc-estimate1}, we obtain the inequality
$$
\left| \log \frac{\phi(\zeta) - \phi(\zeta_n)}{\zeta - \zeta_n }  \right|_{\calC}\leqslant \left| \log \frac{\phi(\zeta) - \phi(\zeta_n)}{\zeta - \zeta_n } - \log \frac{\phi(\zeta_n)}{\zeta_n} \right|_{\calC}+ \left|  \log \frac{\phi(\zeta_n)}{\zeta_n} \right|_{\calC}\leqslant \widetilde{C_2}+\varepsilon,
$$
where $\eps>0$ is sufficiently large so that \blue{the constant $M_1=M_1(\eps)$ from Lemma~\ref{lem:qc-estimate1} equals 1}. Hence, it is suffices to put $C_2:=\exp (\widetilde{C_2}+\eps)>1$ so that
$$
\frac{1}{C_2}\leqslant \frac{|\phi(\zeta) - \phi(\zeta_n)|}{|\zeta - \zeta_n| }  \leqslant C_2,
$$
and use the assumption that $|\zeta - \zeta_n|=\kappa r_n$ to get the result.
\end{proof}

Finally, we prove the third and last result from Section~\ref{sec:qc-estimate}. Let $\zeta\in\C$ and suppose that there exists $0<\theta<\redf{\pi}$ such that for all $m\in\N$,
\begin{equation}
\label{eq:bm-arg}
B_m \subseteq \{z \in \C\ : \ \arg \zeta + \theta < \arg z < \arg \zeta +2\pi- \theta\}.
\end{equation}
Then there exists a constant $C_3>1$ such that 
\begin{equation}
\frac{1}{C_3} \le |\phi'(\zeta)| \le C_3.\nonumber
\end{equation}
Note that $C_3$ depends on $\theta$ but not on $\zeta$, provided that $\zeta$ satisfies \eqref{eq:bm-arg}.

\begin{proof} [Proof of Lemma \ref{lem:qc-estimate3}] 
Applying Corollary \ref{cor:KeyInequality}  
with $\alpha = \zeta$, $\beta=0$ and $\gamma=(1+\delta)\zeta$ with $0<\delta<\delta_1$ so that $0<|\gamma - \zeta|=\delta|\zeta|<\delta_1 |\zeta|$,  
we obtain
$$
\left| \log \frac{\phi(\gamma) - \phi(\zeta)}{\gamma-\zeta} - \log \frac{\phi(\zeta)}{\zeta} \right|_{\calC} 
\le  C(K -1) \sum_{m=\blue{M_2}}^{\infty} \iint_{B_m} \frac{|\zeta| \tinyhsp dx\tinyhsp dy}{| z(z-\zeta)(z-\gamma)|},
$$
for some \blue{$M_2\in\N$}. For $z \in B_m$, let $\theta'=\arg (\zeta/z)\in (\theta,2\pi-\theta)$. Then $\cos \theta' \le \cos \theta$ and 
\begin{equation}
\label{eq:replace-zeta-gamma}
|z - \zeta|^2 = |z|^2+|\zeta|^2 - 2\,|z| \, |\zeta| \cos \theta' \ge 2\,|z| \, |\zeta| (1- \cos \theta).
\end{equation}
Since $\theta'=\arg (\gamma/z)$, there is a similar inequality replacing $\zeta$ for $\gamma=(1+\delta)\zeta$ in \eqref{eq:replace-zeta-gamma}. Thus,
$$
|z - \zeta|\,|z - \zeta|\geqslant 2(1-\cos \theta)|z|\,|\zeta| \sqrt{1+\delta},
$$   
and
$$
\begin{array}{rl}
\ds\iint_{B_m} \frac{|\zeta| \tinyhsp dx\tinyhsp dy}{| z(z-\zeta)(z-\gamma)|} 
\hspace{-6pt}&\ds\le \frac{|\zeta|}{2(1-\cos \theta) |\zeta| \sqrt{1+\delta}} \iint_{B_m} \frac{\tinyhsp dx\tinyhsp dy}{|z|^2} \vspace{10pt}\\
&\ds\le  \frac{1}{2(1-\cos \theta) \sqrt{1+\delta}} \frac{2\pi r_m^2}{(\frac{3}{4} |\zeta_m|)^2}.  
\end{array}
$$
Note that since $\zeta\notin \overline{\text{supp}\,\mu_\phi}$, the quantity $\phi'(\zeta)$ exists. Again, the sum is finite by the assumption, and letting $\delta \to 0$, we obtain that there exists $\widetilde{C_3}>0$ such that
$$
\left| \log \phi'(\zeta) - \log \frac{\phi(\zeta)}{\zeta} \right|_{\calC}\leqslant \widetilde{C_3}.
$$
Together with Lemma \ref{lem:qc-estimate1}, we obtain the desired inequality by proceeding as before and putting $C_3:=\exp(\widetilde{C_3}+\eps)>1$. \blue{The last claim of the lemma follows from the fact that by choosing $M_2=M_2(\eta)\geqslant M_1(\eps)$ sufficiently large we can make both $\widetilde{C_3}$ and $\eps$ sufficiently small.}
\end{proof}

To conclude this appendix, we prove Lemma~\ref{lem:holom-dependence-of-centers} which says that the sequence of centers $(c_{N+1}(\mathbf{w}),c_{N+2}(\mathbf{w}),\dots)$ depends continuously on $\mathbf{w}=(w_N,w_{N+1},\dots)\in\overline{\D}(\frac{1}{2}, \frac{1}{8})^{\N_N}$. This result was important for the shooting argument in Section~\ref{sec:shooting}.

\begin{proof}[Proof of Lemma~\ref{lem:holom-dependence-of-centers}]
\label{proof:lem:holom-dependence-of-centers}
The sets $\phi_\mathbf{w}(E_{\pm n})$ may be distorted, but Lemmas~\ref{lem:qc-estimate1} and \ref{lem:qc-estimate2} ensure that they will still satisfy the Assumption~\ref{ass:qu-estimates} with different constants. So with $K$ replaced by $K^2$, we obtain a similar estimate as Lemma~\ref{lem:qc-estimate1}. In particular, 
for any $\eps>0$, and any compact set $X\subseteq \barD_R$, there exists $T \ge N$ such that if $\mathbf{w}, \mathbf{w}' \in \overline{\D}(\frac{1}{2}, \frac{1}{8})^{\N_N}$ and $w_n = w_n'$ for $N \le n <T$, then
$$
\left| \log \frac{\phi_{\mathbf{w'}} \circ \phi_{\mathbf{w}}^{-1}(\zeta)}{\zeta} \right| < \frac{\eps}{2eR},\quad \text{ for } \zeta\in \C\sminus\{0\},
$$ 
that is, if $\xi=\phi_\mathbf{w}^{-1}(\zeta)$,
$$
\left| \log \frac{\phi_{\mathbf{w'}}(\xi)}{\xi} - \log \frac { \phi_{\mathbf{w}}(\xi)}{\xi} \right| < \frac{\eps}{2eR},\quad \text{ for } \xi \in \C \sminus \{0\}.  
$$ 
Since $\log(\phi_{\mathbf{w}'}(\xi)/\xi),\log(\phi_{\mathbf{w}}(\xi)/\xi)\in \mathbb D$ and $|e^{z'}-e^z|\leqslant e|z'-z|$ for all $z,z'\in\mathbb D$, we have
$$
\left| \frac{\phi_{\mathbf{w'}}(\xi)}{\xi} - \frac { \phi_{\mathbf{w}}(\xi)}{\xi}\right|<\frac{\varepsilon}{2R},\quad \text{ for } \xi \in \C \sminus \{0\}, 
$$
and
$$
|\phi_{\mathbf{w}'}(\xi)-\phi_{\mathbf{w}}(\xi)|< \frac{\varepsilon}{2},\quad \text{ for } \xi\in X\subseteq \overline{\mathbb D}_R.
$$
If $\mathbf{w}'' \in\overline{\D}(\frac{1}{2}, \frac{1}{8})^{\N_N}$ is chosen so that $|w_n''-w_n|$ are small enough for $N \le n <T$ and $w_n''=w_n'$ for $T \le n$, then 
$||\mu_{\phi_{\mathbf{w}''}} - \mu_{\phi_{\mathbf{w}'}}||_{\infty}$ will be small, and we can achieve $|\phi_{\mathbf{w}''}(\zeta) - \phi_{\mathbf{w}'}(\zeta)|< \eps/2$ on $X$.  
Therefore we have $|\phi_{\mathbf{w}''}(\zeta) - \phi_{\mathbf{w}}(\zeta)|< \eps$ on $X$, and this proves the continuity of the map $\mathbf{w} \mapsto \phi_{\mathbf{w}}$.
 %(An alternative proof will be to use $L^p$-norm for $\mu_{\phi_{\mathbf{w}}}$ with $p$ sufficiently large, 
%because the proof of the measurable Riemann mapping theorem in \cite[Theorem 2 in Chapter 5.B]{ahlfors06} was carried out with $L^p$-norm with $p$ large, 
%so the continuity will follow from teh construction.) 
\par
Finally, recall that for each $n \ge N$, $c_n(\mathbf{w})$ was defined by a local composition of $\phi_{\mathbf{w}}$ and $g^{-1}$, so it is continuous.  
The final claim follows from the definition of the product topology.  
\end{proof}

\bibliography{bibliography}

\end{document}